\documentclass[fleqn]{zzz}
\usepackage{mathrsfs}
\usepackage{xcolor}
\usepackage{tikz}
\usepackage{amsmath,amsthm,amsbsy,amssymb}
\usepackage{graphicx}
\usepackage{rotating}
\usepackage{fixmath}
\usepackage[pdftex]{hyperref}
\usepackage{soul}
\usepackage{stmaryrd}
\usepackage{relsize}
\usepackage{amssymb}
\usepackage{bm}
\usepackage{comment}
\usepackage{soul}
\usetikzlibrary{decorations}

\pgfdeclaredecoration{dashsoliddouble}{initial}{
 \state{initial}[width=\pgfdecoratedinputsegmentlength]{
 \pgfmathsetlengthmacro\lw{.3pt+.5\pgflinewidth}
 \begin{pgfscope}
 \pgfpathmoveto{\pgfpoint{0pt}{\lw}}%
 \pgfpathlineto{\pgfpoint{\pgfdecoratedinputsegmentlength}{\lw}}%
 \pgfmathtruncatemacro\dashnum{%
 round((\pgfdecoratedinputsegmentlength-3pt)/6pt)
 }
 \pgfmathsetmacro\dashscale{%
 \pgfdecoratedinputsegmentlength/(\dashnum*6pt + 3pt)
 }
 \pgfmathsetlengthmacro\dashunit{3pt*\dashscale}
 \pgfsetdash{{\dashunit}{\dashunit}}{0pt}
 \pgfusepath{stroke}
 \pgfsetdash{}{0pt}
 \pgfpathmoveto{\pgfpoint{0pt}{-\lw}}%
 \pgfpathlineto{\pgfpoint{\pgfdecoratedinputsegmentlength}{-\lw}}%
 \pgfusepath{stroke}
 \end{pgfscope}
 }
}

\hypersetup{
colorlinks = true,
urlcolor = blue,
linkcolor = black,
}

\definecolor{Mycolor2}{HTML}{e85d04}
\sethlcolor{black}
 
\newcommand{\qhyp}[5]{\,\mbox{}_{#1}\phi_{#2}\!\left(\!\!\begin{array}{c}{#3}\\[0.10cm]{#4}\end{array};{#5}\right)}
\newcommand{\qphyp}[6]{\,{}_{#1}\phi_{{#2}}^{{#3}}\!\left(\!\!
\begin{array}{c}{#4}\\[0.10cm] {#5}\end{array};#6\!\right)}

\newtheorem{thm}{Theorem}[section]
\newtheorem{cor}[thm]{Corollary}

\newtheorem{rem}[thm]{Remark}
\newtheorem{lem}[thm]{Lemma}
\newtheorem{defn}[thm]{Definition}
\newtheorem{prop}[thm]{Proposition}

\makeatletter
\def\eqnarray{\stepcounter{equation}\let\@currentlabel=\theequation
\global\@eqnswtrue
\tabskip\@centering\let\\=\@eqncr
$$\halign to\displaywidth\bgroup\hfil\global\@eqcnt\z@
$\displaystyle\tabskip\z@{##}$&\global\@eqcnt\@ne
\hfil$\displaystyle{{}##{}}$\hfil
&\global\@eqcnt\tw@ $\displaystyle{##}$\hfil
\tabskip\@centering&\llap{##}\tabskip\z@\cr}

\def\endeqnarray{\@@eqncr\egroup
\global\advance\c@equation\m@ne$$\global\@ignoretrue}

\def\@yeqncr{\@ifnextchar [{\@xeqncr}{\@xeqncr[5pt]}}
\makeatother

\newcommand{\Z}{\mathbb{Z}} 
\newcommand{\C}{\mathbb{C}} 
\newcommand{\N}{\mathbb{N}} 

\newcommand{\RR}{{{\mathbb R}}}
\newcommand{\CC}{{{\mathbb C}}}
\newcommand{\CCast}{{{\mathbb C}^\ast}}
\newcommand{\CCdag}{{{\mathbb C}^\dag}}
\newcommand{\CCddag}{{{\mathbb C}^\ddag}}

\newcommand{\expe}{{\mathrm e}}

\newcommand{\dd}{{\mathrm d}}

\usepackage{scalerel,url}
\let\svus_
\catcode`_=\active %
\def_{\ifmmode\expandafter\svus\expandafter\bgroup\expandafter\lowerit
\else\svus\fi}
\def\lowerit#1{\ThisStyle{\raisebox{-2\LMpt}{$\SavedStyle#1$}}\egroup}
\makeatletter
\AtBeginDocument{
\check@mathfonts
\fontdimen13\textfont2=3.5pt
\fontdimen14\textfont2=3.5pt
\fontdimen16\textfont2=2.5pt
\fontdimen17\textfont2=2.5pt
 }
\makeatother

\begin{document}
\renewcommand{\PaperNumber}{***}

\FirstPageHeading

\ShortArticleName{Asymptotics, orthogonality relations and duality in the $q^{-1}$-symmetric Askey scheme}

\ArticleName{Asymptotics, orthogonality relations and duality for the\\$q$ and $q^{-1}$-symmetric polynomials in the $q$-Askey scheme}
\Author{Howard S. Cohl$\,^{\ast}\orcidB{}$, 
Roberto S. Costas-Santos$\,^{\dag}\orcidA{}$
and Xiang-Sheng Wang $^{\ddag}\orcidC{}$
}
\AuthorNameForHeading{H.~S.~Cohl, 
R.~S.~Costas-Santos, X.-S.~Wang
}
\Address{$^\ast$ Applied and Computational 
Mathematics Division, National Institute 
of Standards 
and Tech\-no\-lo\-gy, Mission Viejo, CA 92694, USA
\URLaddressD{
\href{http://www.nist.gov/itl/math/msg/howard-s-cohl.cfm}
{http://www.nist.gov/itl/math/msg/howard-s-cohl.cfm}
}
} 
\EmailD{howard.cohl@nist.gov} 

\Address{$^\dag$ Department of Quantitative 
Methods, Universidad Loyola Andaluc\'ia, 
E-41704 Seville, Spain
} 
\URLaddressD{
\href{http://www.rscosan.com}
{http://www.rscosan.com}
}
\EmailD{rscosa@gmail.com} 

\Address{$^\ddag$ Department of Mathematics, University of Louisiana at Lafayette, Lafayette, LA 70503, USA
} 
\URLaddressD{
\href{https://userweb.ucs.louisiana.edu/~xxw6637/}
{https://userweb.ucs.louisiana.edu/$\sim$xxw6637/}
}
\EmailD{xswang@louisiana.edu} 


\ArticleDates{Received~\today~in final form ????; 
Published online ????}
\Abstract{In this survey we summarize the current state of known orthogonality relations for the $q$ and $q^{-1}$-symmetric and dual subfamilies of the Askey--Wilson polynomials in the $q$-Askey scheme. 
These polynomials are the continuous dual $q$ and $q^{-1}$-Hahn polynomials, the $q$ and $q^{-1}$-Al-Salam--Chihara polynomials, the continuous big $q$ and $q^{-1}$-Hermite polynomials and the continuous $q$ and $q^{-1}$-Hermite polynomials and their dual counterparts which are connected with the big $q$-Jacobi polynomials, the little $q$-Jacobi polynomials and the $q$ and $q^{-1}$-Bessel polynomials. 
The $q^{-1}$-symmetric polynomials in the $q$-Askey scheme satisfy an indeterminate moment problem, satisfying an infinite number of orthogonality relations for these polynomials. Among the infinite number of orthogonality relations for the $q^{-1}$-symmetric families, we attempt to summarize those currently known.
These fall into several classes, including continuous orthogonality relations and infinite discrete (including bilateral) orthogonality relations. 
Using symmetric limits, we derive a new infinite discrete orthogonality relation for the continuous big $q^{-1}$-Hermite polynomials. 
Using duality relations, we explore 
orthogonality relations for and from the dual families associated with the $q$ and $q^{-1}$-symmetric subfamilies of the Askey--Wilson polynomials. 
In order to give a complete description of the convergence properties
for these polynomials, we provide the large
degree asymptotics using the Darboux
method for these polynomials. 
In order to apply the Darboux
method, we derive a generating function with two free parameters 
for the $q^{-1}$-Al-Salam--Chihara polynomials which has natural limits to the lower $q^{-1}$-symmetric families.
}


\Classification{33D45, 05A15, 42C05, 05E05, 33D15}


\tableofcontents
\addtocontents{toc}{\protect\color{black}}

\section{Mathematical preliminaries}\label{sec:2}

\noindent We adopt the following set 
notations: $\mathbb N_0:=\{0\}\cup
\mathbb N=\{0, 1, 2,\ldots\}$, and we 
use the sets $\mathbb Z$, $\mathbb R$, 
$\mathbb C$ which represent 
the integers, real numbers, and 
complex numbers respectively, 
$\CCast:=\CC\setminus\{0\}$, 
$\CCdag:=\{z\in\CCast: |z|<1\}$,\
$\CCddag:=\CC\setminus(\{0\}\cup\{z\in\CCast:|z|=1\})$.
\noindent Consider $q\in\CCdag$, $n\in\mathbb N_0$.
Define the sets 
\begin{eqnarray}
&&\hspace{-9.0cm}\Omega_q^n:=\{q^{-k}: 
k\in\mathbb N_0,~0\le k\le n-1\},\\
&&\hspace{-9.0cm}\Omega_q:=\Omega_q^\infty
=\{q^{-k}:k\in\mathbb N_0\},
\\ &&\hspace{-9.0cm}\Upsilon_q:=\{q^k:k\in\Z\}.
\end{eqnarray}
Recall the notion of a {\it multiset} 
which extends the definition of a set where the multiplicity of elements is allowed. This notion becomes important for 
hypergeometric functions, where numerator and denominator parameter entries may be identical.
\begin{defn}\label{def:1.1}
We adopt the following conventions for succinctly 
writing elements of multisets. To indicate 
sequential positive and negative 
elements, we write
\[
\pm a:=\{a,-a\}.
\]
\noindent We also adopt an analogous notation
\[
z^{\pm}:=\{z,z^{-1}\}.
\]
\end{defn}

\subsection{Binomial coefficients and {\it q}-shifted factorials}\label{sec:2.1}

Let $n,k\in\N_0$ such that $n\ge k$. Then the {\it binomial coefficient} is defined by
\begin{eqnarray}
&&\hspace{-11.7cm}\binom{n}{k}:=\frac{n!}{k!(n-k)!}. 
\end{eqnarray}
Note that $\binom{n}{2}=\frac12n(n-1)$, or for $\mu\in\CC$, we define $\binom{\mu}{2}=\frac12\mu(\mu-1)$.
\noindent We will use \cite[(1.2.58-59)]{GaspRah}
\begin{eqnarray}
\label{binomic}
&&\hspace{-9.0cm}\binom{n+k}{2}=
\binom{n}{2}+\binom{k}{2}+kn,\\
&&\hspace{-9.0cm}\binom{n-k}{2}
=\binom{n}{2}+\binom{k}{2}+k(1-n).
\label{binomid}
\end{eqnarray}

\medskip
\noindent We will need 
the 
$q$-shifted factorial 
$(a;q)_n:=(1-a)(1-qa)\cdots(1-q^{n-1}a)$, 
$n\in\mathbb N_0$. The $q$-binomial coefficient is given by \cite[(I.39)]{GaspRah}
\begin{eqnarray}
&&\hspace{-10.5cm}\begin{bmatrix}n\\[1pt]k\end{bmatrix}_q:=
\frac{(q;q)_n}{(q;q)_k(q;q)_{n-k}},
\end{eqnarray}
where in our applications, we use $n,k\in\mathbb N_0$
.
One may also define
\begin{eqnarray}
&&\hspace{-10.3cm}(a;q)_\infty:=\prod_{n=0}^\infty 
(1-aq^{n}),\label{poch.id:2}
\end{eqnarray}
where $|q|<1$. 
Furthermore, one has the following identity
\begin{eqnarray}
&&\hspace{-11.0cm}(a;q)_n:=\frac{(a;q)_\infty}
{(a q^n;q)_\infty},
\end{eqnarray}
where $a q^n\not \in \Omega_q$, which allows one to generalize $(a;q)_n$ for $n\not\in\N_0$. For instance, one has for \cite[(1.8.6)]{Koekoeketal},
\begin{eqnarray}
\hspace{-8.5cm}(a;q)_{-n}=q^{\binom{n}{2}}\frac{(-q/a)^n}{(q/a;q)_n},\quad n\in\Z.
\label{qPochneg}
\end{eqnarray}
We will also use the common 
notational product conventions
\begin{eqnarray}
&&\hspace{-4cm}(a_1, \ldots, a_k;q)_n:=
(a_1;q)_n\cdots(a_k;q)_n,\quad k\in\N_0, n\in\CC\cup\{\infty\}.\nonumber
\end{eqnarray}
The {\it $q$-shifted factorial} 
also satisfies the 
following useful properties
\cite[(1.8.10), (1.8.10), 
(1.8.11), (1.8.17), (1.8.19), (1.8.22)]
{Koekoeketal}:
\begin{eqnarray}
\label{poch.id:3} 
&&\hspace{-2.1cm}
(a;q^{-1})_n=q^{-\binom{n}{2}}(-a)^n(a^{-1};q)_n,\quad  a\ne 0, n\in\N_0,\\
&&\hspace{-2.1cm}(a;q)_{n+k}=(a;q)_k(aq^k;q)_n 
= (a;q)_n(aq^n;q)_k,\quad n,k\in\N_0,
\label{qPoch1}
\\
\label{poch.id:5}&&\hspace{-2.1cm} (a;q)_n
=(q^{1-n}/a;q)_n(-a)^nq^{\binom{n}{2}},\quad a\ne 0, n\in\N_0,\\
&&\hspace{-2.1cm}\frac{(a;q)_{n-k}}{(b;q)_{n-k}}=
\left(\frac{b}{a}\right)^k
\frac{(a;q)_n(\frac{q^{1-n}}{b};q)_k}
{(b;q)_n(\frac{q^{1-n}}{a};q)_k},\quad 
a,b\ne 0, n\in\N_0, k=0,1,2,\ldots,n, 
\label{qPoch2}\\
\label{qPochq3}
&&\hspace{-2.1cm}
(q^{-n-k};q)_k=q^{-\binom{k}{2}}
(-q)^{-k}q^{-nk}
\frac{(q;q)_k(q^{1+k};q)_n}{(q;q)_n},\quad n,k\in\N_0,
\\
\label{qPochq2}&&\hspace{-2.1cm}
(\pm a;q)_n=(a^2;q^2)_n,\quad n\in\N_0.
\end{eqnarray}
From \eqref{qPochq3}, one also has \cite[(1.8.18)]{Koekoeketal}
\begin{equation}
\hspace{0.9cm}(q^{-n};q)_k=(-1)^{k}q^{\binom{k}{2}}q^{-nk}\frac{(q;q)_n}{(q;q)_{n-k}},\quad n,k\in\N_0.
\label{qPochiden2}
\end{equation}
Note the equivalent representation of \eqref{poch.id:3} 
which is useful for obtaining limits, namely 
\begin{equation*}
\hspace{1.0cm}
a^n\left(\frac{x}{a};q\right)_n=
q^{\binom{n}{2}}(-x)^n\left(\frac{a}{x};q^{-1}\right)_n,
\end{equation*}
where $a\ne 0$. Therefore
\begin{equation}
\hspace{1.0cm}\lim_{a\to0}\,a^n\left(\frac{x}{a};q\right)_n=
\lim_{b\to\infty}\,\frac{1}{b^n}\left(xb;q\right)_n=
q^{\binom{n}{2}}(-x)^n.
\label{critlim}
\end{equation}

\noindent Furthermore, one has the following
useful identity
 \cite[p.~8]{Ismailetal2022},
 \cite[(6)]{CohlCostasSantos23},
\begin{eqnarray}
&&\hspace{-9.4cm}\label{sq}
(a^2;q)_\infty=(\pm a,\pm q^\frac12 a;q)_\infty.
\end{eqnarray}
\noindent The {\it theta function} $\vartheta(z;q)$ (sometimes referred to as a modified theta function 
\cite[(11.2.1)]{GaspRah})
is defined by Jacobi's triple product identity and is
given by {\cite[(1.6.1)]{GaspRah}} (see also \cite[(2.3)]{Koornwinder2014})
\begin{equation}
\hspace{1cm}\vartheta(z;q):=
(z,q/z;q)_\infty=\frac{1}{(q;q)_\infty}\sum_{n=-\infty}^\infty (-1)^nq^{\binom{n}{2}}z^n,
\label{tfdef}
\end{equation}
where $z\ne 0$, $|q|<1$. Note that $\vartheta(q^n;q)=0$ if
$n\in\Z$.
We will adopt the product convention
for theta functions for $a_k\in\C$ for $k\in\N$, namely
\begin{eqnarray}
&&\hspace{-7cm}\vartheta(a_1,\ldots,a_k;q):=\vartheta(a_1;q)\cdots\vartheta(a_k;q).\nonumber
\end{eqnarray}

\noindent 
Define the {\it Jackson $q$-integral} as in \cite[(1.11.2)]{GaspRah}
\begin{eqnarray}
&&\hspace{-2.3cm}\int_a^b f(u;q)\,{\mathrm d}_qu=(1-q)b\sum_{n=0}^\infty q^nf(q^nb;q)-(1-q)a\sum_{n=0}^\infty q^nf(q^na;q).
\label{qint}
\end{eqnarray}

\subsection{Basic hypergeometric series}\label{sec:2.1b}

Define the multisets ${\bf a}:=\{a_1,\ldots,a_r\}$,
${\bf b}:=\{b_1,\ldots,b_s\}$. 
The {\it basic hypergeometric series}, which we 
will often use, is defined for
$q,z\in\CCast$ such that $|q|<1$, $s,r\in\mathbb N_0$, 
$b_j\not\in\Omega_q$, $j=1, \ldots, s$, as
 \cite[(1.10.1)]{Koekoeketal}
\begin{equation}
\qhyp{r}{s}{\bf a}
{\bf b}
{q,z}:=
{}_r\phi_s({\bf a};{\bf b};q,z)
:=\sum_{k=0}^\infty
\frac{({\bf a};q)_k}
{(q,{\bf b};q)_k}
\left((-1)^kq^{\binom k2}\right)^{1+s-r}
z^k.
\label{2.11}
\end{equation}
\noindent {For $s+1>r$, ${}_{r}\phi_s$ is an entire
function of $z$, for $s+1=r$ then 
${}_{r}\phi_s$ is convergent for $|z|<1$, and 
for $s+1<r$ the series
is divergent {unless it is terminating}.}
Note that when we refer to a basic hypergeometric
function with {\it arbitrary argument} $z$, we 
mean that the argument does not necessarily 
depend on the other parameters, namely the $a_j$'s, 
$b_j$'s nor $q$. However, for the arbitrary 
argument $z$, it very-well may be that the domain 
of the argument is restricted, such as for $|z|<1$.

Redefine {the multiset} ${\bf a}:=
\{a_1,\ldots,a_{r-1}\}$.
If one has a {\it nonterminating} basic hypergeometric series which has a numerator parameter given by some $q^{-n}$ for $n\in\N_0$, then the basic hypergeometric series terminates because $(q^{-n};q)_k=0$ for $k\ge n+1$. This leads to the following definition.
A {\it terminating} basic hypergeometric series 
which appear in basic hypergeometric orthogonal
polynomials are defined as
\begin{equation}
\qhyp{r}{s}{q^{-n},{\bf a}}
{\bf b}{q,z}:=\sum_{k=0}^n
\frac{(q^{-n},{\bf a};q)_k}{(q,{\bf b};q)_k}
\left((-1)^kq^{\binom k2}\right)^{1+s-r}z^k,
\label{2.12}
\end{equation}
where $b_j\not\in\Omega_q^n$, $j=1, \ldots, s$.
{Whereas for $s+1<r$, a nonterminating basic 
hypergeometric series is divergent, for a 
terminating basic hypergeometric series, 
this is no longer the case.}

Note that we refer to a basic hypergeometric
series as {\it $\ell$-balanced} if
$q^\ell a_1\cdots a_r=b_1\cdots b_s$, 
and {\it balanced} if $\ell=1$.
A basic hypergeometric series ${}_{r+1}\phi_r$ is 
{\it well-poised} if 
the parameters satisfy the relations
\[
qa_1=b_1a_2=b_2a_3=\cdots=b_ra_{r+1}.
\]

In the sequel, we will use the following notation 
${}_{r+1}\phi_s^m$, $m\in\mathbb Z$
(originally due to van de Bult \& Rains
\cite[p.~4]{vandeBultRains09}), 
for basic hypergeometric series with
zero parameter entries.
Consider $p\in\mathbb N_0$. 
Redefine {the multiset} ${\bf a}:=
\{a_1,\ldots,a_{r+1}\}$.
Then define
\begin{equation}\label{topzero} 
{}_{r+1}\phi_s^{-p}\left(\begin{array}{c}{\bf a}\\
{\bf b}\end{array};q,z
\right)
:=
\qhyp{r+p+1}{s}
{{\bf a},\overbrace{0,\ldots,0}^{p}}
{\bf b}{q, z},
\end{equation}
\begin{equation}\label{botzero}
{}_{r+1}\phi_s^{\,p}\left(\begin{array}{c}{\bf a}\\
{\bf b}\end{array};q,z\right)
:=
\qhyp{r+1}{s+p}{\bf a}
{{\bf b},
\overbrace{0,\ldots,0}^{p}}{q,z},
\end{equation}
where $b_1,\ldots,b_s\not
\in\Omega_q\cup\{0\}$, and
${}_{r+1}\phi_s^0:={}_{r+1}\phi_{s}$.
The nonterminating basic hypergeometric series 
${}_{r+1}\phi_s^m({\bf a};{\bf b};q,z)$, 
is well-defined for 
$s-r+m\ge 0$. 
In particular ${}_{r+1}\phi_s^m$ is an entire function 
of $z$ for $s-r+m>0$, convergent for $|z|<1$ for $s-r+m=0$ 
and divergent if $s-r+m<0$.
Note that we will move interchangeably between the
van de Bult \& Rains notation and the alternative
notation with vanishing numerator and denominator parameters
which are used on the right-hand sides of \eqref{topzero} 
and \eqref{botzero}.

We will often use (frequently without mentioning) the 
following limit transition 
formulas, which can be found in 
\cite[(1.10.3-5)]{Koekoeketal}
\begin{eqnarray}
&&\hspace{-7.4cm}\label{limit1}
\lim_{\lambda\to\infty}
\qhyp{r}{s}{{\bf a},\lambda a_r}{{\bf b}}
{q,\frac{z}{\lambda}}
=\qhyp{r-1}{s}{{\bf a}}{{\bf b}}{q,a_rz},
\end{eqnarray}
where ${\bf a}:=
\{a_1,\ldots,a_{r-1}\}$, ${\bf b}:=
\{b_1,\ldots,b_{s}\}$,
\begin{eqnarray}
&&\hspace{-7.4cm}\label{limit2}\lim_{\lambda\to\infty}
\qhyp{r}{s}{{\bf a}}{{\bf b},
\lambda b_s}{q,\lambda z}=\qhyp{r}{s-1}{{\bf a}}
{{\bf b}}{q,\frac{z}{b_s}},
\end{eqnarray}
where ${\bf a}:=
\{a_1,\ldots,a_{r}\}$, ${\bf b}:=
\{b_1,\ldots,b_{s-1}\}$,
\begin{eqnarray}
&&\hspace{-7.0cm}\label{limit3}\lim_{\lambda\to\infty}
\qhyp{r}{s}{{\bf a},\lambda a_r}{{\bf b},
\lambda b_s}{q,z}
=\qhyp{r-1}{s-1}{{\bf a}}{{\bf b}}
{q,\frac{a_r}{b_s}z},
\end{eqnarray}
where ${\bf a}:=
\{a_1,\ldots,a_{r-1}\}$, ${\bf b}:=
\{b_1,\ldots,b_{s-1}\}$.
\subsection{Nonterminating and terminating summations}
One has the following important nonterminating summations. 
The {\it $q$-binomial theorem} states \cite[(1.11.1)]{Koekoeketal}
\begin{equation}
\label{qbinom}
\qhyp10{a}{-}{q,z}=\frac{(az;q)_\infty}{(z;q)_\infty}, 
\quad q\in\CCdag,\quad |z|<1.
\end{equation}

\noindent Also, one has the following $q$-analogue of the exponential 
function which is due to Euler \cite[(1.14.1)]{Koekoeketal} 
(see also \cite[Theorem 12.2.6]{Ismail:2009:CQO}).
\begin{thm}[Euler]
\label{Euler}
Let $q\in\CCdag$, $z\in\CC$ such that $|z|<1$. Then
\begin{eqnarray}
\label{qexp}
&&\hspace{-9cm}{\mathrm e}_q(z):=\qphyp00{-1}{-}{-}{q,z}
=\frac{1}{(z;q)_\infty},\\
&&\hspace{-9cm}{\mathrm E}_q(z)=\qhyp00{-}{-}{q,-z}=(-z;q)_\infty.
\label{qexp2}
\end{eqnarray}
\end{thm}
\begin{proof}
See proof of \cite[Theorem 12.2.6]{Ismail:2009:CQO}.
\end{proof}

\noindent Note that \eqref{qexp} is the $a=0$ special case of the 
$q$-binomial theorem \eqref{qbinom}.
There is also the $q$-Gauss summation 
\cite[\href{http://dlmf.nist.gov/17.6.E1}{(17.6.1)}]{NIST:DLMF}
\begin{equation}
\qhyp21{a,b}{c}{q,\frac{c}{ab}}=
\frac{(\frac{c}{a},\frac{c}{b};q)_\infty}
{(c,\frac{c}{ab};q)_\infty},
\label{qGs}
\end{equation}
where $|c|<|ab|$.

\medskip
One important terminating summation is the
terminating $q$-binomial sum, namely \cite[(1.11.2)]{Koekoeketal} or
\begin{equation}
\qhyp10{q^{-n}}{-}{q,z}=(q^{-n}z;q)_n=
q^{-\binom{n}{2}}\left(-\frac{z}{q}\right)^n(\tfrac{q}{z};q)_n, 
\label{termqbinom}
\end{equation}
which converges for all values of $z$, but the second equality is not valid for $z=0$.
Other important terminating summations include the 
$q$-Chu--Vandermonde sum 
\cite[\href{http://dlmf.nist.gov/17.6.E2}{(17.6.2)}]{NIST:DLMF}
\begin{equation}
\label{qChuVander}
\qhyp21{q^{-n},a}{b}{q,q}=a^n\frac{(\frac{b}{a};q)_n}{(b;q)_n},
\end{equation}
and the reversed version of the $q$-Chu--Vandermonde sum \cite[(II.7)]{GaspRah}
\begin{equation}
\qhyp21{q^{-n},a}{b}{q,\frac{q^nb}{a}}=\frac{(\frac{b}{a};q)_n}{(b;q)_n}.
\label{qChuVanderR}
\end{equation}
One useful limit of \eqref{qChuVander} is
\begin{equation}
\qhyp21{q^{-n},0}{a}{q,q}=\frac{q^{\binom{n}{2}}(-a)^n}{(a;q)_n}.
\label{limqChu}
\end{equation}

\subsection{Nonterminating and terminating transformations}
One also has the following
nonterminating transformations
\cite[(1.13.8-9)]{Koekoeketal}
\begin{eqnarray}
&&\hspace{-4.0cm}\qphyp{1}{0}{1}{a}{-}{q,z}
=(a,z;q)_\infty\qphyp{0}{1}{-2}{-}{z}{q,a}
=(z;q)_\infty\qhyp01{-}{z}{q,az}.
\label{trans10101}
\end{eqnarray}
Another useful nonterminating transformation is
\begin{equation}
\qhyp11{a}{b}{q,z}=(z;q)_\infty\qhyp12{\frac{b}{a}}{b,z}{q,az}.
\label{nt1112}
\end{equation}
{
One also has the following nonterminating transformation between a ${}_2\phi_2$ 
and a ${}_2\phi_1$ cf.~ \cite[(III.4)]{GaspRah}
\begin{equation}
\qhyp22{a,b}{c,\frac{abz}{c}}{q,z}=
\frac{(\frac{bz}{c};q)_\infty}{(\frac{abz}{c};q)_\infty}
\qhyp21{a,\frac{c}{b}}{c}{q,\frac{bz}{c}}.
\label{rel2122}
\end{equation}
}

\noindent Two important terminating transformations which we will
use are \cite[(17.9.8)]{NIST:DLMF}
\begin{equation}
\qhyp32{q^{-n},a,b}{c,d}{q,q}=\left(\frac{ab}{c}\right)^n
\frac{(\frac{cd}{ab};q)_n}{(d;q)_n}
\qhyp32{q^{-n},\frac{c}{a},\frac{c}{b}}{c,\frac{cd}{ab}}{q,q},
\label{3phi2term}
\end{equation}
and \cite[(III.13)]{GaspRah}
\begin{equation}
\qhyp32{q^{-n},a,b}{c,d}{q,\frac{q^ncd}{ab}}
=\frac{(\frac{d}{b};q)_n}{(d;q)_n}\qhyp32{q^{-n},b,\frac{c}{a}}{c,q^{1-n}\frac{b}{d}}{q,q}.
\label{3phi2sec}
\end{equation}

\noindent In \cite[Exercise 1.4ii]{GaspRah}, one 
finds the inversion formula for
terminating basic hypergeometric series.

\begin{thm}[Gasper \& Rahman's (2004) Inversion Theorem]
\label{thm:1.2}
Let $m, n, k, r, s\in\mathbb N_0$, $0\le k\le r$, 
$0\le m\le s$, 
$a_k, b_m\not\in
\Omega^n_q\cup\{0\}$,
$q\in\mathbb C^\ast$ such that $|q|\ne 1$,
${\bf a}:=
\{a_1,\ldots,a_{r}\}$, ${\bf b}:=
\{b_1,\ldots,b_{s}\}$.
Then,
\begin{eqnarray}
&&\hspace{-0.9cm}\qhyp{r+1}{s}{q^{-n},{\bf a}}
{{\bf b}}{q,z}=\frac{({\bf a};q)_n}
{({\bf b};q)_n}\left(\frac{z}
{q}\right)^n\left((-1)^nq^{\binom{n}{2}}
\right)^{s-r-1}
\label{inversion}
\sum_{k=0}^n
\frac{\left(q^{-n},\frac{q^{1-n}}{{\bf b}}
;q\right)_k}
{\left(q,\frac{q^{1-n}}{{\bf a}}
;q\right)_k}\left(\frac{b_1\cdots b_s}{a_1\cdots a_r}
\frac{q^{n+1}}{z}\right)^k.
\end{eqnarray}
\end{thm}
\noindent From the above inversion formula \eqref{inversion},
one may derive the following useful terminating 
basic hypergeometric transformation lemma.
\begin{lem}
\label{lem:1.3}
Let {$p\in\mathbb Z$}, $n,r,s\in\mathbb N_0$, 
$a_k, b_m\not\in\Omega^n_q\cup\{0\}$,
$z, q\in\mathbb C^\ast$ such that $|q|\ne 1$,
${\bf a}:=
\{a_1,\ldots,a_{r}\}$, ${\bf b}:=
\{b_1,\ldots,b_{s}\}$.
Then
\begin{eqnarray}
&&\hspace{-2.2cm}\qphyp{r+1}{{s}}{p}{q^{-n},{\bf a}}
{{\bf b}}{q,z}=\frac{({\bf a}; q)_n}
{({\bf b};q)_n}\left(\frac{z}{q}\right)^n
\left((-1)^nq^{\binom{n}{2}}\right)^{s-r{+p}-1}
\nonumber\\ 
&&\hspace{2.5cm}\times
\qphyp{{s+1}}{r}{{s-r+p}}
{q^{-n},\frac{q^{1-n}}{{\bf b}}, 
}{\frac{q^{1-n}}{{\bf a}
}}{q,\frac{b_1\cdots b_{{s}}}
{a_1\cdots a_r}\frac{q^{(1-p)n+p+1}}{z}}.
\label{ivg3}
\end{eqnarray}
\end{lem}
\begin{proof}
In a straightforward calculation, if we write 
 \eqref{inversion} and we apply \eqref{poch.id:3} 
assuming all the parameters are nonzero, and then 
we apply identities \eqref{limit1} and \eqref{limit2} 
one obtains \eqref{ivg3}.
This completes the proof.
\end{proof}



\begin{cor}\label{cor:1.5}
Let $n,r\in\mathbb N_0$, $z, q\in\mathbb C^\ast$ such 
that $|q|\ne 1$, and for $0\le k\le r$, let $a_k, 
b_k\not\in\Omega^n_q\cup\{0\}$,
${\bf a}:=
\{a_1,\ldots,a_{r}\}$, ${\bf b}:=
\{b_1,\ldots,b_{r}\}$. Then,
\begin{eqnarray}
&&\hspace{-1.8cm}\qhyp{r+1}{r}{q^{-n},{\bf a}}
{{\bf b}}{q,z}
=
\label{cor:1.5:r1}
q^{-\binom{n}{2}}
(-1)^n
\frac{({\bf a};q)_n}
{({\bf b};q)_n}
\left(\frac{z}{q}\right)^n\!\!\!
\qhyp{r+1}{r}{q^{-n},
\frac{q^{1-n}}{{\bf b}}
}
{\frac{q^{1-n}}{{\bf a}}
}{q,
\frac{q^{n+1}}{z}\frac{b_1\cdots b_r}
{a_1\cdots a_r}}.
\end{eqnarray}
\end{cor}
\begin{proof}
Take $r=s$, $p=0$ in \eqref{ivg3}, which completes
the proof.
\end{proof}
\noindent Note that in Corollary \ref{cor:1.5}
if the terminating basic hypergeometric
series on the left-hand side is balanced
then the argument of the terminating basic 
hypergeometric series on the right-hand side 
is $q^2/z$.

\medskip
Another equality we can use is the following 
connecting relation between terminating
basic hypergeometric series with base $q$, and 
with base $q^{-1}$,
${\bf a}:=\{a_1, \ldots, a_r\}$,
${\bf b}:=\{b_1, \ldots, b_r\}$,
\begin{eqnarray} 
&&\hspace{-2.4cm}\qhyp{r+1}{r}{q^{-n},
{\bf a}
}
{{\bf b}
}{q,z}=
\qhyp{r+1}{r}{q^{n}, {\bf a}^{-1}
}
{
{\bf b}^{-1}
}{q^{-1}, 
\dfrac{a_1 a_2\cdots a_r}
{b_1 b_2\cdots b_r}\dfrac{z}{q^{n+1}}}\nonumber\\
&&\hspace{0.5cm}=
q^{-\binom{n}{2}}
\left(-\frac zq\right)^n
\frac{({\bf a};q)_n}
{({\bf b};q)_n}
\qhyp{r+1}{r}{q^{-n},
\frac{q^{1-n}}{{\bf b}}
}
{\frac{q^{1-n}}{{\bf a}}
}
{q,\frac{b_1\cdots b_r}{a_1\cdots a_r}\frac{q^{n+1}}{z}}.
\label{qtopiden} 
\end{eqnarray}
To understand the procedure for obtaining 
the $q^{-1}$-analogues of the basic 
hypergeo\-metric orthogonal polynomials studied 
in this manuscript, let us consider a special 
case in detail. 
Let $n\in\mathbb N_0$,
\begin{equation}
\label{freqs}
f_{n,r}(q):=f_{n,r}(q;z(q);{\bf a}(q),{\bf b}(q)):=g_r(q)
\qhyp{r+1}{r}{q^{-n},{\bf a}(q)}{{\bf b}(q)}{q,z(q)},
\end{equation}
where 
${\bf a}(q):=\left\{{a_1(q)},\ldots,{a_r(q)}\right\}$,
${\bf b}(q):=\left\{{b_1(q)},\ldots,{b_r(q)}\right\}$,
\noindent which will suffice, for instance, for 
the study of the 
terminating basic hypergeometric 
representations for the 
Askey--Wilson polynomials. 
In order to obtain the corresponding
$q^{-1}$-hypergeometric representations of $f_{n,r}(q)$, one 
only needs to consider the corresponding $q^{-1}$-function:
\begin{equation}
f_{n,r}(q^{-1})=g_r(q^{-1})
\qhyp{r+1}{r}{q^n,{\bf a}(q^{-1})}{{\bf b}(q^{-1})}{q^{-1},z(q^{-1})}.
\label{invertedrep}
\end{equation}
\begin{prop}\label{thm:2.6}
Let $r,k\in\mathbb N_0$, $0\le k\le r$, $a_k(q)\in\mathbb 
C$, $b_k(q)\in\Omega_q$, $q\in\mathbb C^\ast$ such that 
$|q|\ne 1$, $z(q)\in\mathbb C$.
Define ${\bf a}(q):=\{a_1(q),\ldots,a_r(q)\}$, 
${\bf b}(q):=\{b_1(q),\ldots,b_r(q)\}$ and a multiplier 
function $g_r(q):=g_r(q;z(q);{\bf a}(q);{\bf b}(q))$ 
which is not of basic hypergeometric type (some 
multiplicative combination of powers and $q$-Pochhammer symbols), 
and $z(q):=z(q;{\bf a}(q);{\bf b}(q))$. Then defining 
$f_{n,r}(q)$ as in \eqref{freqs}, one has
\begin{equation}
\label{termreprr}
f_{n,r}(q^{-1})
=g_r(q^{-1})\qhyp{r+1}{r}{q^{-n},{\bf a}^{-1}(q^{-1})}
{{\bf b}^{-1}(q^{-1})}{q,\frac{q^{n+1}a_1(q^{-1})
\cdots a_r(q^{-1}) z(q^{-1})}{b_1(q^{-1})\cdots b_r(q^{-1})}}.
\end{equation}
\end{prop}
\begin{proof}
By using \eqref{poch.id:3} repeatedly with 
the definition \eqref{2.11} in \eqref{invertedrep}, 
one obtains the $q$-inverted terminating 
representation \eqref{termreprr}, which 
corresponds to the original terminating basic 
hypergeo\-me\-tric representation \eqref{freqs}. 
This completes the proof.
\end{proof}

\medskip
Now consider the more general case. 
Let $r,s\in\mathbb N_0$, $0\le t\le r$, $0\le u\le 
s$, and let
\begin{equation}
\hspace{-0.2cm}\left.
\begin{array}{c}
{\bf a}(q):=\{a_1(q),\ldots,a_{r-t}(q),\overbrace{0,\ldots,0}^t\}
\\[0.3cm]
{\bf b}(q):=\{b_1(q),\ldots,b_{s-u}(q),\underbrace{0,\ldots,0}_u\}
\end{array}
\right\},
\end{equation}
where either $t>0$, $u=0$, or $u>0$, $t=0$, or $t=u=0$, 
and as above, a multiplier function 
$g_{r,s,t,u}(q):=g_{r,s,t,u}(q;z(q);{\bf a}(q);{\bf b}(q))$ 
and $z(q):=z(q;{\bf a}(q);{\bf b}(q))$.
Define 
\begin{equation}
\label{frstueq}
f_{r,s,t,u}(q):=
g_{r,s,t,u}(q)
\qhyp{r+1}{s}{q^{-n},{\bf a}(q)}
{{\bf b}(q)}{q,z(q)}.
\end{equation}
To obtain the $q$-inverted representation of 
$f_{r,s,t,u}$, one must again compute
\begin{equation}
f_{r,s,t,u}(q^{-1})=g_{r,s,t,u}(q^{-1})
\qhyp{r+1}{s}{q^n,{\bf a}(q^{-1})}{{\bf b}(q^{-1})}
{q^{-1},z(q^{-1})}.
\label{invertedgenrep}
\end{equation}
This can be obtained by repeated use of 
\eqref{poch.id:3} using the definition \eqref{2.11} 
and various combinations of \eqref{limit1}--
\eqref{limit3}.


\section{The \(q\) and \(q^{-1}\)-symmetric polynomials and their dual families}
\label{sec:2.2.1}

We will study a subset of basic hypergeometric orthogonal 
polynomials in the $q$-Askey scheme.
Basic hypergeometric orthogonal polynomials satisfy 
orthogonality relations, which are given either as 
an integral or a sum (see \S\ref{sec:3.8} below).
We will consider the $q$, and $q^{-1}$, symmetric families in the $q$-Askey-scheme, namely the Askey--Wilson, continuous dual $q$-Hahn, Al-Salam--Chihara, continuous big $q$-Hermite and continuous $q$-Hermite polynomials and {as well as} their $q^{-1}$-analogues.
If one considers duality (see \S\ref{sec:3.8-dua} below) 
then one finds that there exist duality relations 
between the $q$ and $q^{-1}$-symmetric subfamilies of 
the Askey--Wilson polynomials and other well-studied 
families of basic hypergeometric orthogonal polynomials 
in the $q$-Askey scheme. 
These are the big $q$-Jacobi polynomials and functions, 
little $q$-Jacobi polynomials and functions and the 
$q$-Bessel polynomials and the $q^{-1}$-Bessel functions. 
In the remainder of this section, we will define and 
present some important 
properties of these polynomials and functions.

\medskip
\begin{rem}
\label{rem:2.7}
Throughout the paper, we will examine orthogonal 
polynomials in $x=\frac12(z+z^{-1})$.
Note that, in this case, $x=x(z)$ is invariant under 
the map $z\mapsto z^{-1}$, so all functions (including 
polynomials) in $x$ will also satisfy this invariance.
\end{rem}

\noindent The Askey--Wilson polynomials
$p_n(x;a,b,c,d|q)$ are at the top of 
the {\it symmetric} family of basic hypergeometric orthogonal polynomials in the $q$-Askey scheme and are symmetric in four parameters $a,b,c,d$.
There exists a sequence of subfamilies of 
the Askey--Wilson polynomials, which are symmetric 
in three (and two) parameters and are obtained by 
setting one (two) of the parameters of the Askey-Wilson polynomials equal to zero. These are the continuous dual $q$-Hahn polynomials $p_n(x;a,b,c|q)$ 
(and the Al-Salam--Chihara polynomials $Q_n(x;a,b|q)$).
By continuing to set parameters equal to zero, 
one obtains the continuous big $q$-Hermite 
polynomials $H_n(x;a|q)$, and the continuous 
$q$-Hermite polynomials $H_n(x|q)$. 
By replacing $q\mapsto q^{-1}$ we obtain new 
subfamilies of the  Askey--Wilson polynomials, 
which are also symmetric in their parameters, 
namely the continuous 
dual $q^{-1}$-Hahn, $q^{-1}$-Al-Salam--Chihara, and 
as well the continuous big $q^{-1}$-Hermite and 
continuous $q^{-1}$-Hermite polynomials.
We refer to these families as the $q$ and 
$q^{-1}$-symmetric subfamilies of the 
Askey--Wilson polynomials.

The continuous dual $q$-Hahn, 
the Al-Salam--Chihara, 
the continuous big $q$-Hermite, 
and the continuous $q$-Hermite 
polynomials are the $d\to c\to b\to a\to 0$ limit 
cases of the Askey--Wilson polynomials, namely
\begin{eqnarray}
&&\hspace{-7.8cm}
p_n(x;a,b,c|q)=\lim_{d\to 0}p_n(x;a,b,c,d|q),\label{cdqHl}\\
&&\hspace{-7.8cm}
Q_n(x;a,b|q)=\lim_{c\to 0}p_n(x;a,b,c|q),\label{ASCl}\\
&&\hspace{-7.8cm}
H_n(x;a|q)=\lim_{b\to 0}Q_n(x;a,b|q),\label{cbqHl}\\
&&\hspace{-7.8cm}
H_n(x|q)=\lim_{a\to 0}H_n(x;a|q).\label{cqHl}
\end{eqnarray}
The continuous dual $q$-Hahn and Al-Salam--Chihara polynomials 
are symmetric in the variables $a,b,c$, and $a,b$ respectively.
By starting with representations of the Askey--Wilson 
polynomials \eqref{aw:def3}, we can obtain terminating basic 
hypergeometric series representations of the
symmetric family.

\medskip
\noindent The $q^{-1}$-Askey--Wilson polynomials
$p_n(x;a,b,c,d|q^{-1})$, which are simply renormalized Askey--Wilson 
polynomials with parameters given by their reciprocals,
{are given by} 
\begin{eqnarray}
&&\hspace{-5cm}p_n(x;a,b,c,d|q^{-1})
=q^{-3\binom{n}{2}}(-abcd)^np_n(x;\tfrac{1}{a},\tfrac{1}{b},
\tfrac{1}{c},\tfrac{1}{d}|q).
\end{eqnarray}
This easily follows from Theorem \ref{AWthm}, 
Proposition \ref{thm:2.6}, and Remark \ref{rem:2.7}.
The continuous dual $q^{-1}$-Hahn polynomials can 
be obtained from the Askey--Wilson polynomials as follows
\begin{equation}
\hspace{1cm}p_n(x;a,b,c|q^{-1})=q^{-3\binom{n}{2}}
\left(-abc\right)^n
\lim_{d\to0}d^n\,p_n(x;\tfrac{1}{a},\tfrac{1}{b},
\tfrac{1}{c},\tfrac{1}{d}|q).
\label{limtcdqiH}
\end{equation}
Furthermore, the other members of the $q^{-1}$-symmetric 
family (also a set of symmetric polynomials 
in their parameters $a,b,c$) can also be obtained as 
$c\to b\to a\to 0$ limit cases, {namely} 
\begin{eqnarray}
&&\hspace{-7.7cm}
Q_n(x;a,b|q^{-1})=\lim_{c\to 0}
p_n(x;a,b,c|q^{-1}),\label{qiASCl}\\
&&\hspace{-7.7cm}
H_n(x;a|q^{-1})=\lim_{b\to 0}Q_n(x;a,b|q^{-1}),
\label{cbqiHl}\\
&&\hspace{-7.7cm}
H_n(x|q^{-1})=\lim_{a\to 0}H_n(x;a|q^{-1}).\label{cqiHl}
\end{eqnarray}
In the sequel, we will adopt the following notation for these polynomial sequences:~$p_n[z;a,b,c|q^{\pm 1}]:=p_n(x;a,b,c|q^{\pm 1})$, 
$Q_n[z;a,b|q^{\pm 1}]:=Q_n(x;a,b|q^{\pm 1})$,
$H_n[z;a|q^{\pm 1}]:=H_n(x;a|q^{\pm 1})$, 
$H_n[z|q^{\pm 1}]:=H_n(x|q^{\pm 1})$, 
where $x=\frac12(z+z^{-1})$.

\begin{rem}
Regarding the $q^{-1}$-symmetric subfamilies of the Askey--Wilson polynomials, 
throughout this paper, we have adopted the notation from \cite{Koekoeketal} for these families such that 
we treat only the mapping $q\mapsto q^{-1}$ from the standard families. 
Using this notation, 
as is well-known, continuous orthogonality relations for these polynomials (Theorem \ref{theo33} and Corollaries \ref{corr311}, \ref{corr318}, \ref{corr326}, below),
occur with measures which have support on the imaginary axis of the arguments (see, for instance, \cite{Askey89cqiH,BergIsmail1996,ChristansenIsmail2006,IsmailZhangZhou2022,KoelinkStokman2003}
, \cite[(18.28.18)]{NIST:DLMF}). Typically, in the literature, 
the way this is avoided is by replacing the $z$, the 
parameters $a,b,c$, etc., 
and the power-series parameter $t$ in the case of generating functions with 
the imaginary unit multiplied by them. This converts continuous orthogonality 
relations to those with support on the real axis. The reader should note that 
one can readily convert from our notation to these alternative notations (or vice versa) 
by making these imaginary replacements in the polynomials. Note that these imaginary replacements 
convert from the polynomials being evaluated at $x=\frac12(z+z^{-1})$ to 
polynomials being evaluated at $x=\frac12(z-z^{-1})$, and therefore the invariance 
under the map $z\mapsto z^{-1}$ will be lost. With our chosen notation, 
the invariance is retained, albeit, for instance, at the expense of the continuous orthogonality 
relations being evaluated with support on the imaginary axis.
\end{rem}

\subsection{The Askey--Wilson polynomials}
\label{sec:2.2.1b}
The following result describes different ways of representing the {\it Askey--Wilson polynomials} as
terminating ${}_4\phi_3$ basic hypergeometric series 
\cite[(15)]{CohlCostasSantos20b}.

\begin{thm}
Let $n\in\mathbb N_0$, $q\in\CCddag$,
$x=\frac12(z+z^{-1})$, $z\in\CCast$, 
$a,b,c,d\in\CCast$.
Then, the Askey--Wilson polynomials have the following terminating 
basic hypergeometric series representations:
\begin{eqnarray}
\hspace{0.30cm}
\label{aw:def1} 
&&\hspace{-1.2cm}p_n(x;a,b,c,d|q) := a^{-n} (ab,ac,ad;q)_n 
\qhyp43{q^{-n},q^{n-1}abcd, az^{\pm}}{ab,ac,ad}{q,q}\\
\label{aw:def2} &&\hspace{1.6cm}=q^{-\binom{n}{2}} (-a)^{-n} 
\frac{(\frac{abcd}{q};q)_{2n}
(a z^{\pm};q)_n}
{(\frac{abcd}{q};q)_n}
\qhyp43{q^{-n},
\frac{q^{1-n}}{ab},
\frac{q^{1-n}}{ac},
\frac{q^{1-n}}{ad}
}
{\frac{q^{2-2n}}{abcd},\frac{q^{1-n}}{a}z^{\pm}}{q,q}\\
\label{aw:def3} &&\hspace{1.6cm}=z^n(ab,cz^{-1},dz^{-1};q)_n
\qhyp43{q^{-n},az,bz,\frac{q^{1-n}}{cd}}
{ab,\frac{q^{1-n}}{c}z,\frac{q^{1-n}}{d}z}{q,q}.
\end{eqnarray}
\label{AWthm}
\end{thm}
\begin{proof}
See proof of \cite[Theorem 7]{CohlCostasSantos20b}.
\end{proof}

\subsection{The continuous dual {\it q} and $q^{-1}$-Hahn 
polynomials}\label{sec:3.3}
The {\it continuous dual $q$-Hahn polynomials} are symmetric 
in three parameters $a,b,c$.
The following result describes different ways of representing the {\it continuous dual $q$-Hahn polynomials}.

\begin{cor}\label{cor:4.1} 
Let $n\in\mathbb N_0$, 
$x=\frac12(z+z^{-1})$, $z\in\CCast$, 
$q\in\CCddag$, $a,b,c\in\CCast$.
Then, the continuous dual $q$-Hahn polynomials have 
the following terminating basic hypergeometric series 
representations:
\begin{eqnarray}
\label{cdqH:def1} &&\hspace{-1.45cm}p_n(x;a,b,c|q) 
:= a^{-n} (ab,ac;q)_n 
\qhyp{3}{2}{q^{-n}, az^{\pm}}{ab,ac}{q,q}\\
\label{cdqH:def2} &&\hspace{1.05cm}=q^{-\binom{n}{2}}(-a)^{-n} 
(az^{\pm};q)_n\qhyp{3}{2}{q^{-n}, 
\frac{q^{1-n}}{ab},
\frac{q^{1-n}}{ac}
}{\frac{q^{1-n}}{a}z^{\pm}}
{q,q^n bc}\\\label{cdqH:def3} &&\hspace{1.05cm}=z^n 
(ab,cz^{-1};q)_n\qhyp{3}{2}{q^{-n}, az, bz}
{ab,\frac{q^{1-n}}{c}z}{q,\frac{q}{cz}}\\
\label{cdqH:def4} &&\hspace{1.05cm}=z^n (az^{-1},bz^{-1};q)_n 
\qhyp{3}{2}{q^{-n},cz,\frac{q^{1-n}}{ab}}
{\frac{q^{1-n}}{a}z,\frac{q^{1-n}}{b}z}{q,q}
.
\end{eqnarray}
\end{cor}
\begin{proof} 
The representation \eqref{cdqH:def1} is derived by starting 
with \eqref{aw:def1} and replacing $b$, $c$, or $d\to 0$
(see also \cite[(14.3.1)]{Koekoeketal});
\eqref{cdqH:def2} is derived using \eqref{aw:def2} 
and taking, for instance, $d\to 0$;
\eqref{cdqH:def3} is derived by using \eqref{aw:def3} 
and taking $d\to 0$; 
\eqref{cdqH:def4} is derived by using \eqref{aw:def3} and 
taking $b\to 0$ and replacing $d\to b$. 
This completes the proof.
\end{proof}

\noindent 
In the next result we present some
representations  of the continuous dual 
$q^{-1}$-Hahn polynomials.
\begin{cor}
\label{cor:3.4}
Let $p_n(x;a,b,c|q)$ and all the respective parameters 
be defined as previously. Then, the continuous dual 
$q^{-1}$-Hahn polynomials have the following terminating 
basic hypergeometric series representations:
\begin{eqnarray}
&&\hspace{-1.55cm}
\label{cdqiH:1} p_n(x;a,b,c|q^{-1})=
q^{-2\binom{n}{2}}
(abc)^n \left(
\frac{1}{ab},\frac{1}{ac}
;q\right)_n\qhyp{3}{2}{q^{-n},\frac{z^{\pm}}{a}}
{
\frac{1}{ab},\frac{1}{ac}
}{q,\frac{q^n}{bc}}\\
\label{cdqiH:2} 
&&\hspace{1.3cm}\hspace{-1mm}
=q^{-\binom{n}{2}}(-a)^n\left(\frac{z^{\pm}}{a};q\right)_n 
\qhyp{3}{2}{q^{-n},
q^{1-n}ab,
q^{1-n}ac
}{q^{1-n}az^{\pm}}{q,q}\\
\label{cdqiH:3} 
&&\hspace{1.3cm}\hspace{-1mm}
=q^{-2\binom{n}{2}} (abc)^n\left(\frac{1}{ab},\frac{z}
{c};q\right)_n 
\qhyp{3}{2}{q^{-n}, 
\frac{1}{az}, 
\frac{1}{bz}
}{
\frac{q^{1-n}c}{z},
\frac{1}{ab}
}{q,q}\\
\label{cdqiH:4} 
&&\hspace{1.3cm}\hspace{-1mm}
=q^{-2\binom{n}{2}} 
\left(\frac{ab}{z}\right)^n
\left(\frac{z}{a},\frac{z}
{b};q\right)_n\qhyp{3}{2}{q^{-n},\frac{1}
{cz},q^{1-n}ab}
{\frac{q^{1-n}a}{z},\frac{q^{1-n}b}{z}}
{q,\frac{qc}{z}}.
\end{eqnarray}
\end{cor}
\begin{proof}
Each inverse representation is derived from the 
corresponding representation by applying the 
map $q\mapsto q^{-1}$ and using \eqref{poch.id:3}.
\end{proof}

\medskip
\noindent 
One has the following generating function for the 
continuous dual $q^{-1}$-Hahn polynomials 
${\sf G}(t;a,b,c|q)$ given by 
\begin{eqnarray}
&&\hspace{-2.7cm}{\sf G}(t;a,b,c|q):=\sum_{n=0}^\infty 
\frac{t^n\,q^{2\binom{n}{2}}\,p_n(x;a,b,c|q^{-1})}
{(q,\frac{1}{ab};q)_n}
=\frac{(bt;q)_\infty}{(abct;q)_\infty}
\qhyp22{\frac{z^\pm}{a}}{\frac{1}{ab},bt}{q,at},
\label{cdqiHgf2}
\end{eqnarray}
which follows from \cite[(5.21)]{IsmailZhangZhou2022}, with replacement of notation
\begin{equation}
V_n(x;a,b,c|q)=\frac{q^{2\binom{n}{2}}\left(-\frac{q^2}{bc}\right)^n\,i^{-n}}{(-\frac{q^2}{ab},-\frac{q^2}{bc};q)_n}p_n\left[iz;\frac{ia}{q},\frac{ib}{q},\frac{ic}{q}\,\bigg|q^{-1}\right],
\label{IsmailVnot}
\end{equation}
and after it has been corrected.
Note the corrected generating function in the notation of \cite{IsmailZhangZhou2022} should be 
\begin{equation}
\sum_{n=0}^\infty \frac{(-\frac{q^2}{t_1t_2},-\frac{q^2}{t_2t_3};q)_n}{(q,-\frac{q^2}{t_1t_3};q)_n}V_n(x;t_1,t_2,t_3|q)\left(\frac{t\,t_2}{t_1}\right)^n=\frac{(-\frac{t}{z};q)_\infty}{(\frac{t\,t_2}{q};q)_\infty}
\qhyp21{\frac{qz}{t_1},\frac{qz}{t_3}}{-\frac{q^2}{{t_1t_3}}}{q,-\frac{t}{z}}.
\end{equation}
\subsection{The {\it q} and $q^{-1}$-Al-Salam--Chihara 
polynomials}\label{sec:3.4}

\medskip
The {\it Al-Salam--Chihara  polynomials} are symmetric 
in the parameters $a,b$.
The following result describes different ways of representing these polynomials.

\begin{cor}\label{cor:3.5}
Let $n\in\mathbb N_0$, 
$x=\frac12(z+z^{-1})$, $z\in\CCast$, 
$q\in\CCddag$, $a,b\in\CCast$.
Then, the Al-Salam--Chihara polynomials have the following terminating basic hypergeometric series representations:
\begin{eqnarray}
\label{ASC:def1} 
&&\hspace{-3.8cm}Q_n(x;a,b|q):=a^{-n}(ab;q)_n
\qphyp{3}{1}{1}{q^{-n}, az^{\pm}}{ab}{q,q}\\
\label{ASC:def2} &&\hspace{-1.7cm}= q^{-\binom{n}{2}} 
(-a)^{-n} (az^{\pm};q)_n\qhyp22{q^{-n},\frac{q^{1-n}}{ab}}
{\frac{q^{1-n}}{a}z^{\pm}}{q,\frac{qb}{a}}\\
\label{ASC:def5} &&\hspace{-1.7cm}= z^n (ab;q)_n\qhyp31{q^{-n}, 
az,bz}{ab}{q,\frac{q^n}{z^2}}\\
\label{ASC:def4} &&\hspace{-1.7cm}=z^{n}(a z^{-1};q)_n 
\qhyp{2}{1}{q^{-n},bz}
{\frac{q^{1-n}}{a}z}{q,\frac{q}{az}}\\
\label{ASC:def3} &&\hspace{-1.7cm}=z^{n}(
az^{-1},bz^{-1}
;q)_n
\qphyp{2}{2}{-1}{q^{-n},\frac{q^{1-n}}{ab}}
{
\frac{q^{1-n}}{a}z,
\frac{q^{1-n}}{b}z
}{q,q}
.
\end{eqnarray}
\end{cor}
\begin{proof}
The representation \eqref{ASC:def1} is derived by taking \eqref{cdqH:def1} 
and replacing $c\mapsto 0$ (see also 
\cite[(14.8.1)]{Koekoeketal}); \eqref{ASC:def2} is derived 
by taking \eqref{cdqH:def2} 
and replacing $c\mapsto 0$; \eqref{ASC:def5} is derived by 
taking \eqref{cdqH:def3} and replacing $c\mapsto 0$;
\eqref{ASC:def4} is derived by taking \eqref{cdqH:def3} 
replacing $b\mapsto 0$ (see also \cite[(14.8.1)]{Koekoeketal})
and interchanging $c$ and $a$;
\eqref{ASC:def3} is derived by taking \eqref{cdqH:def4} 
and replacing $c\mapsto 0$.
The limit formulas \eqref{limit1} and \eqref{limit2} are used whenever 
applicable. 
Also, whenever necessary, the parameters should be re-named such 
that they are in the multiset $\{a,b\}$. This completes the proof.
\end{proof}

\noindent Using the Al-Salam--Chihara polynomial terminating basic hypergeometric 
representations, we can compute their $q^{-1}$-analogues.

\begin{cor}\label{cor:3.6}
Let $Q_n(x;a,b|q)$ and the respective parameters 
be defined as previously. 
Then, the $q^{-1}$-Al-Salam--Chihara polynomials 
have the following terminating basic hypergeometric series representations
\begin{eqnarray}
\label{qiASC:1}&&\hspace{-1.8cm}Q_n(x;a,b|q^{-1})=
q^{-\binom{n}{2}}
(-b)^n 
\left(\frac{1}{ab}
;q\right)_n\qhyp31{q^{-n}, 
\frac{z^{\pm}}{a}
}{
\frac{1}{ab}
}
{q,\frac{q^na}{b}}\\\label{qiASC:2} &&\hspace{0.27cm}
=q^{-\binom{n}{2}}(-a)^{n}\left(\frac{z^{\pm}}{a};q\right)_n 
\qphyp{2}{2}{-1}{q^{-n},q^{1-n}ab}
{q^{1-n}az^{\pm}}{q,q}\\
\label{qiASC:5} &&\hspace{0.27cm}=q^{-\binom{n}{2}}(-ab
z)^n\left(\frac{1}{ab};q\right)_n 
\qphyp{3}{1}{1}{q^{-n},
\frac{1}{az},
\frac{1}{bz}
}
{\frac{1}{ab}}{q,q}\\
\label{qiASC:3} &&\hspace{0.27cm}=
q^{-\genfrac{(}{)}{0pt}{}{n}{2}} (-a)^n 
\left(\frac{1}{az};q\right)_n\qhyp{2}{1}{q^{-n}, 
\frac{z}{b}}
{q^{1-n}az}{q, qbz}\\
\label{qiASC:4} &&\hspace{0.27cm}=
q^{-2\binom{n}{2}}(abz)^n
\left(
\frac{1}{az},
\frac{1}{bz}
;q\right)_n
\qhyp22{q^{-n}, q^{1-n}ab}{
q^{1-n}az,q^{1-n}bz
}
{q,qz^2}.
\end{eqnarray}
\end{cor}
\begin{proof}
Each inverse representation is derived from the 
corresponding representation by applying the 
map $q\mapsto q^{-1}$ and using \eqref{poch.id:3}.
\end{proof}

\medskip
\noindent
First, we present a generating function for the $q^{-1}$-Al-Salam--Chihara 
polynomials, which is a generalization of the generating function 
\cite[Theorem 8.2]{Ismail2020}. Note that the proof of this 
generating function is the same as that used in \cite{Ismail2020}.
\begin{thm}
\label{dggf} 
Let $q\in\CCdag$, $x=\frac12(z+z^{-1})$, $z\in\CCast$, 
$a,b,t,\delta\in\CC$, $|t|<1$, $\gamma\in\CCdag$.
 Then
\begin{eqnarray}
&&\hspace{-1.8cm}\sum_{n=0}^\infty
\frac{q^{\binom{n}{2}}t^n(\gamma;q)_n}{(q,\delta;q)_n}Q_n(x;a,b|q^{-1})=
\frac{(\gamma,-tz^\pm;q)_\infty}{(\delta,-at,-bt;q)_\infty}
\qhyp32{\frac{\delta}{\gamma},-at,-bt}{-tz^\pm}{q,\gamma}.
\label{qiASCgfIsm2}
\end{eqnarray}
\end{thm}
\begin{proof}
As in the proof of \cite[Theorem 8.2]{Ismail2020}, start with 
the left-hand side of \eqref{qiASCgfIsm2} and insert the 
representation of the $q^{-1}$-Al-Salam--Chihara polynomials 
\eqref{qiASC:3} with $z\mapsto z^{-1}$. Then, write the terminating 
${}_2\phi_1$ as a finite sum over $k$, reverse the order of the sum, 
shift $n\mapsto n+k$ and apply the Heine transformation 
\cite[(III.1)]{GaspRah} so that the argument of the resulting 
${}_2\phi_1$ is $q^k\gamma$, then expressing this as an infinite 
series, and reversing the sum produces a sum that can be evaluated 
using the $q$-binomial theorem \eqref{qbinom}, which completes the proof.
\end{proof}

\noindent 
For several reasons, the above generating function \eqref{qiASCgfIsm2} is extremely 
interesting. For instance, it has limits to the 
continuous big $q^{-1}$-Hermite polynomials and the continuous 
$q^{-1}$-Hermite polynomials without losing the free $\gamma$ and 
$\delta$ parameters! It also has some interesting limits. 
If one takes $\gamma\to 0$ then we obtain the following.
\begin{cor}
Let $q\in\CCdag$, $x=\frac12(z+z^{-1})$, $z\in\CCast$, $a,b,t,\delta\in\CC$, 
$|t|<1$. Then
\begin{eqnarray}
&&\hspace{-2.5cm}\sum_{n=0}^\infty
\frac{q^{\binom{n}{2}}t^n}{(q,\delta;q)_n}Q_n(x;a,b|q^{-1})=
\frac{(-tz^\pm;q)_\infty}{(\delta,-at,-bt;q)_\infty}\qhyp22{-at,-bt}{-tz^\pm}{q,\delta}.
\label{qiASCgfIsm3}
\end{eqnarray}
\end{cor}
\begin{proof}
Taking the limit as $\gamma\to 0$ in \eqref{qiASCgfIsm2} completes the proof.
\end{proof}
\noindent If one takes the limit as $\delta\to 0$ in \eqref{qiASCgfIsm2}, then one obtains the following generating function.
\begin{cor}Let $q\in\CCdag$, $x=\frac12(z+z^{-1})$, $a,b,t
\in\CC$, 
$|t|<1$, $z,\gamma\in\CCast$ such that $|\gamma|<1$. Then
\begin{eqnarray}
&&\hspace{-2cm}\sum_{n=0}^\infty
\frac{q^{\binom{n}{2}}t^n(\gamma;q)_n}{(q;q)_n}Q_n(x;a,b|q^{-1})=
\frac{(\gamma,-tz^\pm;q)_\infty}{(-at,-bt;q)_\infty}\qhyp32{-at,-bt,0}{-tz^\pm}{q,\gamma}.
\label{qiASCgfIsm4}
\end{eqnarray}
\end{cor}
\noindent If one takes the limit $\gamma\to 0$ and $\delta=1/(ab)$ in \eqref{qiASCgfIsm2}, then one obtains 
the following generating function.
\begin{cor}
Let $q\in\CCdag$, $x=\frac12(z+z^{-1})$, $z\in\CCast$, $a,b,t,\delta\in\CC$, 
$|t|<1$. Then
\begin{eqnarray}
&&\hspace{-4cm}\sum_{n=0}^\infty
\frac{q^{\binom{n}{2}}t^n}{(q,\frac{1}{ab};q)_n}Q_n(x;a,b|q^{-1})=
\frac{1}{(-bt;q)_\infty}\qhyp21{\frac{z^\pm}{a}}{\frac{1}{ab}}{q,-at}.
\label{gfIsmg2}
\end{eqnarray}
\end{cor}
\begin{proof}
Taking the limit $\gamma\to 0$ in \eqref{qiASCgfIsm2} followed by $\delta\mapsto \frac{1}{ab}$ using \eqref{rel2122} followed by Heine's transformation \cite[(III.1)]{GaspRah} with $\{a,b,c,z\}\mapsto\{\frac{z}{b},-at,-tz,\frac{1}{az}\}$, completes the proof.
\end{proof}
\noindent 
If one takes $\delta=\gamma$ in Theorem 
\ref{dggf}, then the nonterminating ${}_3\phi_2$ becomes unity and one obtains the following
generating function for $q^{-1}$-Al-Salam--Chihara polynomials.
\begin{cor}Let $q\in\CCdag$, $x=\frac12(z+z^{-1})$, $z\in\CCast$, $a,b,t\in\CC$, 
$|t|<1$. Then
\begin{eqnarray}
&&\hspace{-6.5cm}\sum_{n=0}^\infty
\frac{q^{\binom{n}{2}}t^n}{(q;q)_n}Q_n(x;a,b|q^{-1})=
\frac{(-tz^\pm;q)_\infty}{(-at,-bt;q)_\infty}.
\label{qiASCgfIsm3x}
\end{eqnarray}
\end{cor}
\begin{proof}
Taking $\delta=\gamma$ in Theorem \ref{dggf} completes the proof.
\end{proof}

\subsection{The continuous big {\it q} and $q^{-1}$-Hermite 
polynomials}\label{sec:3.5}

\medskip
The following result describes different ways 
of representing the the {\it continuous big 
$q$-Hermite polynomials}.

\begin{cor} 
Let 
$n\in\mathbb N_0$, 
$q\in\CCddag$,
$a\in\CCast$,
$x=\frac12(z+z^{-1})$, $z\in\CCast$.
The continuous big $q$-Hermite polynomials have the following terminating basic hypergeometric series representations:
\begin{eqnarray}
\label{cbqH:def1}
&&\hspace{-4.0cm}H_n(x;a|q)
:=a^{-n}\qphyp{3}{0}{2}{q^{-n}, az^\pm }{-}{q,q}\\
\label{cbqH:def2} 
&&\hspace{-2.1cm}= 
q^{-\binom{n}{2}} (-a)^{-n}
(az^\pm;q)_n 
\qhyp12{q^{-n}}{\frac{q^{1-n}z^\pm}{a}}
{q,\frac{q^{2-n}}{a^2}}\\
\label{cbqH:def3}
&&\hspace{-2.1cm}=z^{n}(az^{-1};q)_n 
\qphyp{1}{1}{-1}{q^{-n}}{\frac{q^{1-n}}{a}z}{q,\frac{q}{az}}\\
\label{cbqH:def4}
&&\hspace{-2.1cm}=z^n\qhyp{2}{0}{q^{-n}, az }{-}{q,\frac{q^n}{z^2}}.
\end{eqnarray}
\end{cor}
\begin{proof}
The representation \eqref{cbqH:def1} is derived by taking \eqref{ASC:def1} and replacing $a_2\mapsto 0$
(see also \cite[(14.18.1)]{Koekoeketal});
\eqref{cbqH:def2} is derived by taking \eqref{ASC:def2} and replacing $a_2\mapsto 0$;
\eqref{cbqH:def3} is derived by taking \eqref{ASC:def3} and replacing $a_2\mapsto 0$; 
\eqref{cbqH:def4} is derived by taking \eqref{ASC:def4} or \eqref{ASC:def5} and 
replacing $b\mapsto 0$
(see also \cite[(14.18.1)]{Koekoeketal}). This completes the proof.
\end{proof}

\noindent Using the continuous big $q$-Hermite polynomials, we can compute 
their $q^{-1}$ representations.

\begin{cor}
Let $H_n(x;a|q)$ and the respective parameters be defined as previously. 
Then, the continuous big $q^{-1}$-Hermite polynomials have the following terminating basic hypergeometric series representations:
\begin{eqnarray}
\label{cbqiH:1} &&\hspace{-3.7cm}H_n(x;a|q^{-1}) =a^{-n}\qhyp{3}{0}{q^{-n},\frac{z^{\pm}}
{a}}{-}{q,q^na^2}\\\label{cbqiH:2}&&\hspace{-1.45cm}=q^{-\binom{n}{2}}(-a)^{n} 
\left(\frac{z^{\pm}}{a};q\right)_n\qphyp{1}{2}{-2}{q^{-n}}
{q^{1-n}az^{\pm}}{q,q}\\\label{cbqiH:3}&&\hspace{-1.45cm}=
q^{-\genfrac{(}{)}{0pt}{}{n}{2}}(-a)^n
\left(\frac{1}{az};q\right)_n\qhyp{1}{1}{q^{-n}}
{q^{1-n}az}{q,qz^2}\\
\label{cbqiH:4} &&
\hspace{-1.45cm}=z^n\qphyp{2}{0}{1}{q^{-n},\frac{1}{az}}{-}
{q,\frac{qa}{z}}.
\end{eqnarray}
\end{cor}
\begin{proof}
For \eqref{cbqiH:1}--\eqref{cbqiH:4}, each inverse representation is derived from the corresponding representation 
by applying the map $q\mapsto 1/q$ and using \eqref{poch.id:3}.
\end{proof}


\subsection{The continuous {\it q} and $q^{-1}$-Hermite 
polynomials}\label{sec:3.6}

\medskip
In the next result we present  
some representations for the {\it continuous $q$-Hermite polynomials}.

\begin{cor}
\label{cor32}
Let $n\in\mathbb N_0$, $q\in\CCddag$,
$x=\frac12(z+z^{-1})$, $z\in\CCast$. Then, one has the 
following terminating basic hypergeometric representation 
for the continuous $q$-Hermite polynomials:
\begin{eqnarray}
\label{cqHrep}
&&\hspace{-8.5cm}H_n(x|q):=
z^n
\qphyp10{-1}{q^{-n}}{-}{q,\frac{q^{n}}{z^2}}.
\end{eqnarray}
\end{cor}
\begin{proof}
To obtain \eqref{cqHrep}, start with \eqref{aw:def3} and take the limit as 
$d\to c\to b\to a\to 0$ sequentially.
\end{proof}

\noindent Similarly, we can compute a basic 
hypergeometric representation of the continuous 
$q^{-1}$-Hermite polynomials.

\begin{cor} 
Let $H_n(x|q)$ and the respective parameters be defined as previously.
The continuous $q^{-1}$-Hermite polynomials have the following terminating basic hypergeometric series representation:
\begin{eqnarray}
\label{cqiH:def1}
&&\hspace{-8.5cm}H_n(x|q^{-1})
=z^n\qphyp{1}{0}{1}{q^{-n}}{-}{q,\frac{q}{z^2}}.
\end{eqnarray}
\end{cor}

\begin{proof}
The inverse representation \eqref{cqiH:def1}
is derived from Corollary \ref{cor32} by applying the map $q\mapsto 
q^{-1}$ and using \eqref{poch.id:3}.
\end{proof}

\noindent Note that there exist connection relations between the continuous $q$-Hermite polynomials and the continuous $q^{-1}$-Hermite polynomials \cite[cf.~Exercises 2.29-30]{CohlIsmail20}
\begin{eqnarray}
&&\hspace{-5.3cm}H_n(x|q)
=(q;q)_n\sum_{k=0}^{\lfloor\frac{n}{2}\rfloor}\frac{q^{3\binom{k}{2}}\left(-q^{1-n}\right)^k}{(q;q)_k(q;q)_{n-2k}}H_{n-2k}(x|q^{-1}),
\label{concqHqiH}\\
&&\hspace{-5.3cm}H_n(x|q^{-1})
=(q;q)_n\sum_{k=0}^{\lfloor\frac{n}{2}\rfloor}\frac{q^{2\binom{k}{2}}\left(q^{1-n}\right)^k}{(q;q)_k(q;q)_{n-2k}}H_{n-2k}(x|q).
\label{concqiHqH}
\end{eqnarray}

\subsection{The big {\it q}-Jacobi polynomials and functions}

The {\it big $q$-Jacobi polynomials} $P_n(x;a,b,c;q)$ can be defined with the following terminating basic hypergeometric representations.
\begin{thm}
Let $n\in\N_0$, $q\in\CCddag$, $a,b,c,x\in\CC$. Then, the big $q$-Jacobi polynomials are defined as
\begin{eqnarray}
&&\hspace{-4cm}P_n(x;a,b,c;q):=\qhyp32{q^{-n},q^{n+1}ab,x}{qa,qc}{q,q}
\label{defbqJ}\\
&&\hspace{-1.45cm}=q^{\binom{n}{2}}(-qa)^n\frac{(\frac{bx}{c};q)_n}{(qa;q)_n}\qhyp32{q^{-n},\frac{q^{-n}c}{ab},\frac{qc}{x}}{qc,\frac{q^{1-n}c}{bx}}{q,q}\label{defbqJ2}\\
&&\hspace{-1.45cm}=q^{\binom{n}{2}}(-qc)^n\frac{(\frac{bx}{c};q)_n}{(qc;q)_n}\qhyp32{q^{-n},\frac{q^{-n}}{b},\frac{qa}{x}}{qa,\frac{q^{1-n}c}{bx}}{q,q}.\label{defbqJ3}
\end{eqnarray}
\end{thm}

\begin{proof}
The definition of the big $q$-Jacobi polynomials can be found here \cite[(14.5.1)]{Koekoeketal}. The second two representations can be obtained using \eqref{3phi2term}.
\end{proof}

\noindent 
The big $q$-Jacobi polynomials satisfy the following useful symmetry relation
\begin{eqnarray}
&&\hspace{-8.5cm}P_n(x;a,b,c;q)=P_n\left(x;c,\frac{ab}{c},a;q\right),
\label{bqJsym}
\end{eqnarray}
which 
directly follows by comparing the terminating basic hypergeometric representations \eqref{defbqJ2}, \eqref{defbqJ3} and is 
evident in \eqref{defbqJ}.
Furthermore, it is interesting to realize that the big $q^{-1}$-Jacobi polynomials are given by renormalized big $q$-Jacobi polynomials with different parameters and arguments.

\begin{thm}
Let $n\in\N_0$, $q\in\CCddag$, $a,b,c\in\CCast$. Then
\begin{eqnarray}
&&\hspace{-4.5cm}P_n(x;a,b,c;q^{-1})=q^{-\binom{n}{2}}\left(-\frac{ab}{qc}\right)^n
\frac{(\frac{qc}{ab};q)_n}{(\frac{q}{c};q)_n}
P_n\left(\frac{qx}{a};\frac{1}{a},\frac{1}{b},\frac{c}{ab};q\right)\label{bqJi1}\\
&&\hspace{-1.6cm}=q^{-\binom{n}{2}}\left(-\frac{ab}{qc}\right)^n
\frac{(\frac{qc}{ab};q)_n}{(\frac{q}{c};q)_n}
P_n\left(\frac{qx}{a};\frac{c}{ab},\frac{1}{c},\frac{1}{a};q\right).\label{bqJi2}
\end{eqnarray}
\end{thm}
\begin{proof}
Start with \eqref{defbqJ} and make the replacement
$q\mapsto q^{-1}$ and simplify using \eqref{poch.id:3}. Then, applying the transformation
\eqref{3phi2sec} and identifying the parameters using \eqref{defbqJ3} yields \eqref{bqJi1}. Then, applying the symmetry relation \eqref{bqJsym}, provides \eqref{bqJi2}. This completes the proof. 
\end{proof}

\medskip
\noindent The big $q$-Jacobi
function is a natural extension
of the big $q$-Jacobi
polynomials \eqref{defbqJ}, such that the degree is allowed
to be any complex number (as opposed to 
a non-negative integer). We define this as 
\begin{eqnarray}
&&\hspace{-7.75cm}P_\mu(x;a,b,c;q):=
\qhyp32{q^{-\mu},q^{\mu+1}ab,x}{qa,qc}{q,q}.
\label{bqJf1}
\end{eqnarray}
The big $q$-Jacobi function reduces to the big $q$-Jacobi polynomial when $\mu=n\in\mathbb N_0$.
{Note that the big $q$-Jacobi function \eqref{bqJf1} also satisfies the symmetry relation \eqref{bqJsym}, namely
\begin{equation}
P_\mu(x;a,b,c;q)=P_\mu\left(x;c,\frac{ab}{c},a;q\right).
\label{bqJfsym}
\end{equation}
}

\begin{rem}
It is important to compare the big $q$-Jacobi function, which appears 
in \eqref{bqJf1}, with 
that similar function which appears in the important
papers by Koelink, Stokman and Rosengren
\cite{KoelinkRosengren2002,KoelinkStokman2001,KoelinkStokman2003}. They define
the big $q$-Jacobi function as 
\cite[(3.2)]{KoelinkStokman2001} 
\begin{equation}
\Phi_\mu(x;a,b,c;q):=\qhyp32{a{\mu}^\pm,-\frac{1}{x}}
{ab,ac}{q,-bcx}.
\end{equation}
The Koelink--Stokman big $q$-Jacobi function is
related to our big $q$-Jacobi function as follows:
\begin{eqnarray}
&&\hspace{-3.3cm}\Phi_\mu(x;a,b,c;q)=\frac{(q^{\mu+1}\frac{ab}{c},\frac{q^{-\mu}}{c};q)_\infty}{(\frac{qab}{c},\frac{1}{c};q)_\infty}
P_\mu(x;a,b,c;q)\nonumber\\
&&\hspace{0.2cm}+\frac{(q^{-\mu},x,q^{\mu+1}ab,\frac{qa}{c};q)_\infty}{(qa,c,\frac{qab}{c},\frac{x}{c};q)_\infty}\qhyp32{\frac{q^{\mu+1}ab}{c},\frac{q^{-\mu}}{c},\frac{x}{c}}{\frac{q}{c},\frac{qa}{c}}{q,q}.
\end{eqnarray}
The Koelink--Stokman 
big $q$-Jacobi
function can be identified with \eqref{bqJf1},
by taking 
\begin{eqnarray}
&&\hspace{-7.0cm}(a,b,c,{\mu},x)\mapsto\left(\sqrt{qab},\sqrt{\frac{qa}{b}},\frac{\sqrt{q}c}{\sqrt{ab}},q^{\mu+\frac12}\sqrt{ab},-\frac{1}{x}\right).
\end{eqnarray}
If $\mu=n\in\mathbb N_0$, then the big $q$-Jacobi polynomial is related to the Koelink--Stokman big $q$-Jacobi function by
\begin{equation}
\Phi_n(x;a,b,c;q)=q^{-\binom{n}{2}}(-qc)^{-n}\frac{(qc;q)_n}{(\frac{qab}{c};q)_n}P_n(x;a,b,c;q).
\end{equation}
\end{rem}

\subsection{The little {\it q}-Jacobi polynomials and functions}

\noindent {The {\it little $q$-Jacobi polynomials} $p_n(x;a,b;q)$ can be defined with the following terminating basic hypergeometric representations.}

{
\begin{thm}Let $n\in\N_0$, $q\in\CCddag$, $x,a,b,c\in\CCast$. Then
\begin{eqnarray}
\label{deflqJ}
&&\hspace{-2.8cm}p_n(x;a,b;q):=\qhyp21{q^{-n},q^{n+1}ab}{qa}{q,qx}\\
&&\hspace{-0.6cm}=q^{-\binom{n}{2}}(-qb)^{-n}\frac{(qb;q)_n}{(qa;q)_n}
\qhyp32{q^{-n},q^{n+1}ab,qbx}{qb,0}{q,q}\label{deflqJ2}\\
&&\hspace{-0.6cm}=(qbx;q)_n
\qhyp32{q^{-n},\frac{q^{-n}}{b},0}{qa,\frac{q^{-n}}{bx}}{q,q},
\label{deflqJ3}
\end{eqnarray}
where $|qx|<1$ in \eqref{deflqJ}.
\end{thm}
\begin{proof}
See \cite[(14.12.1)]{Koekoeketal}, \cite[(66, 67)]{KoornwinderMazzocco2018}.
\end{proof}
}

\noindent 
It is also interesting to recognize that the little $q^{-1}$-Jacobi polynomials are given by renormalized little $q$-Jacobi polynomials with reciprocal parameters and a different argument.
\begin{thm}
Let $n\in\N_0$, $q\in\CCddag$, $a,b\in\CCast$. Then
\begin{eqnarray}
&&\hspace{-8.5cm}p_n(x;a,b;q^{-1})=
p_n\left(\frac{bx}{q};\frac{1}{a},\frac{1}{b};q\right)\label{lqJi1}.
\end{eqnarray}
\end{thm}
\begin{proof}
Start with \eqref{deflqJ} and make the replacement
$q\mapsto q^{-1}$ and simplify using \eqref{poch.id:3}. Then, identifying the parameters and argument with \eqref{deflqJ} completes the proof. 
\end{proof}

The {\it little $q$-Jacobi
function} is a natural extension 
of the little $q$-Jacobi
polynomial \eqref{deflqJ}, such that the degree is allowed
to be any complex number (as opposed to 
a non-negative integer). We define this as 
\begin{eqnarray}
&&\hspace{-0.7cm}p_\mu(x;a,b;q):=
\qhyp21{q^{-\mu},q^{\mu+1}ab}{qa}{q,qx}
\label{lqJf1}\\
&&\hspace{1.5cm}=\frac{(q^{\mu+1}a,\frac{q^{-\mu}}{b};q)_\infty}{(qa,\frac{1}{b};q)_\infty}\qhyp32{q^{-\mu},q^{\mu+1}ab,qbx}{qb,0}{q,q}\nonumber\\
&&\hspace{3cm}+\frac{(q^{-\mu},q^{\mu+1}ab,qbx;q)_\infty}{(b,qa,qx;q)_\infty}\qhyp32{q^{\mu+1}a,\frac{q^{-\mu}}{b},qx}{\frac{q}{b},0}{q,q}\\
&&\hspace{1.5cm}=\frac{(qbx;q)_\infty}{(q^{\mu+1}bx;q)_\infty}
\qhyp32{q^{-\mu},\frac{q^{-\mu}}{b},0}
{qa,\frac{q^{-\mu}}{bx}}{q,q}\nonumber\\
&&\hspace{3cm}+
\frac{(q^{-\mu},q^{\mu+2}abx,
\frac{q^{-\mu}}{b};q)_\infty}{(qa,qx,\frac{q^{-\mu-1}}{bx};q)_\infty}
\qhyp32{qx,qbx,0}{q^{\mu+2}bx,q^{\mu+2}abx}{q,q},
\end{eqnarray}
where $|qx|<1$ for \eqref{lqJf1}.
Both of these functions reduce to the big and little $q$-Jacobi polynomials when $\mu=n\in\mathbb N_0$.
The little $q$-Jacobi function can be obtained from the big $q$-Jacobi function using the following limit transition
\begin{equation}
p_\mu(x;a,b;q)=\lim_{c\to\infty}P_\mu(qcx;a,b,c;q).
\end{equation}

\begin{rem}
It is important to compare the little $q$-Jacobi function, which appears above
in \eqref{lqJf1} with 
that similar function which appears in the important
papers by Koelink, Stokman and Rosengren
\cite{KoelinkRosengren2002,KoelinkStokman2001,KoelinkStokman2003}. 
In \cite[(4.2)]{KoelinkStokman2001}, the little $q$-Jacobi function is defined as 
\begin{equation}
\phi_\mu(x;a,b;q):=\qhyp21{a\mu^\pm}{ab}{q,-bx}.
\label{KSRlJf}
\end{equation}
The Koelink--Stokman 
little $q$-Jacobi
function can be identified with \eqref{lqJf1}
by taking 
\begin{eqnarray}
&&\hspace{-4cm}
(a,b,\mu,x)\mapsto\left(\sqrt{qab},\sqrt{qa/b},q^{\mu+\frac12}\sqrt{ab},-qbx\right).
\end{eqnarray}
\end{rem}

\subsection{The {\it q} and $q^{-1}$-Bessel polynomials and functions}

\noindent The {\it $q$-Bessel polynomials} $y_n(x;a;q)$ can 
be defined with the following terminating basic 
hypergeometric representations, namely,
\begin{eqnarray}
&&\hspace{-0.0cm}y_n(x;a;q):=\qhyp21{q^{-n},-q^na}{0}{q,qx}\\
&&\hspace{1.85cm}=(-q^nax)^n\qhyp21{q^{-n},\frac{1}{x}}{0}{q,-\frac{q^{1-n}}{a}}\\
&&\hspace{1.85cm}=q^{-\binom{n}{2}}(-x)^n (\tfrac{1}{x};q)_n\qhyp11{q^{-n}}{q^{1-n}x}{q,-q^{n+1}ax}\\
&&\hspace{1.85cm}=q^{2\binom{n}{2}}(-qa)^n\qhyp30{q^{-n},-q^na,\tfrac{1}{x}}{-}{q,-\frac{x}{a}}\\
&&\hspace{1.85cm}=q^{-\binom{n}{2}}(-x)^n\frac{(\frac{1}{x};q)_n(-a;q)_{2n}}{(-a;q)_n}\qhyp32{q^{-n},0,0}{q^{1-n}x,-\frac{q^{1-2n}}{a}}{q,q}.
\label{qB:7}
\end{eqnarray}

\noindent One also has the following terminating basic hypergeometric representations for the 
$q^{-1}$-Bessel polynomials $y_n(x;a;q^{-1})$, namely,
\begin{eqnarray}
&&\hspace{-1cm}y_n(x;a;q^{-1}):=\qhyp20{q^{-n},-\frac{q^n}{a}}{-}{q,-ax}\\
&&\hspace{3cm}\hspace{-1.65cm}=q^{-2\binom{n}{2}}\left(-\frac{ax}{q}\right)^n\qhyp20{q^{-n},x}{-}{q,-\frac{q^{2n}}{ax}}\\
&&\hspace{3cm}\hspace{-1.65cm}=(x;q)_n
\qhyp21{q^{-n},0}{\frac{q^{1-n}}{x}}{q,-q^{1-n}a}\\
&&\hspace{3cm}\hspace{-1.65cm}=q^{-2\binom{n}{2}}\left(-\frac{a}{q}\right)^n\qhyp32{q^{-n},-\frac{q^n}{a},x}{0,0}{q,q}
\label{qiB:def4}\\
&&\hspace{3cm}\hspace{-1.65cm}=q^{-3\binom{n}{2}}\left(\frac{a}{q}\right)^n\frac{(x;q)_n(-\frac{1}{a};q)_{2n}}{(-\frac{1}{a};q)_n}\qhyp12{q^{-n}}{-q^{1-2n}a,\frac{q^{1-n}}{x}}{q,-\frac{q^{2-2n}a}{x}}.
\end{eqnarray}
\vspace{0.2cm}

\noindent {
There exists a limit transition from little $q$-Jacobi polynomials to the $q$-Bessel polynomials
\cite[p.~529]{Koekoeketal}
\begin{equation}
\lim_{b\to 0}p_n\left(x;b,-\frac{a}{qb};q\right)
=y_n(x;a;q).
\end{equation}
Similarly, it is not hard to demonstrate by comparing basic hypergeometric representations that there also exists a limit transition from the little $q$-Jacobi polynomials to the $q^{-1}$-Bessel polynomials 
\begin{equation}
\lim_{b\to 0}p_n\left(\frac{x}{qb};-\frac{1}{qab},b;q\right)=y_n(x;a;q^{-1}).
\end{equation}
}

\noindent {Now define the {\it $q^{-1}$-Bessel function} by extending the definition of the $q^{-1}$-Bessel polynomial to arbitrary degree values $\mu\in\CC$, namely
\begin{equation}
y_\mu(x;a;q^{-1}):=q^{-2\binom{\mu}{2}}\left(-\frac{a}{q}\right)^\mu\qhyp32{q^{-\mu},-\frac{q^\mu}{a},x}{0,0}{q,q},
\label{qiBf}
\end{equation}
so that the basic hypergeometric series representation is convergent for all $q\in\CCdag$.}

\section{Duality relations for the $q$ and $q^{-1}$-symmetric and their dual families}\label{sec:3.8-dua}
\noindent There exist duality relations for the $q$ and $q^{-1}$-symmetric subfamilies of the Askey--Wilson polynomials. Some of these duality relations are well-known and others are not as well-known. We now summarize the hierarchy of duality relations for the $q$ and $q^{-1}$-symmetric subfamilies of the Askey--Wilson polynomials which 
are depicted in Figure \ref{Figdual}. 

\begin{figure}[!htb]
\centering
\caption{This figure depicts the $q$-Askey scheme 
including the $q$ and $q^{-1}$-symmetric subfamilies of the Askey--Wilson polynomials and as well their 
corresponding dual families. Arrows represent limit transitions between the subfamilies of the Askey--Wilson polynomials. Dashed lines represent polynomial duality and double dashed lines represent function duality.\label{Figdual}}
\begin{tikzpicture}[level distance=.05cm,sibling distance=.04cm,scale=0.62,every node/.style={scale=0.62},inner sep=11pt]

\node (cqiH) at (13.80,-20.40) {
{\setlength{\fboxrule}{.020cm}
\fbox{
$\begin{array}{c}
{\displaystyle \,}\\[-12pt]
\text{\!continuous $q^{-1}$-Hermite}\\[1pt] 
\text{\phantom{\!}polynomials}
\end{array}$
}
}
};

\node (cqH) at (4.20,-20.40) {
{\setlength{\fboxrule}{.020cm}\fbox{
$\begin{array}{c}
{\displaystyle \,}\\[-12pt]
\text{\!continuous $q$-Hermite}\\[1pt] 
\text{\phantom{\!}polynomials}
\end{array}$
}
}
};

\node (qB) at (11.40,-17.40) {
{\setlength{\fboxrule}{.020cm}
\fbox{
$\begin{array}{c}
{\displaystyle \,}\\[-12pt]
\text{\!$q$-Bessel}\\[1pt] 
\text{\phantom{\!}polynomials}
\end{array}$
}
}
};

\node (cbqiH) at (19.20,-17.40) {
{\setlength{\fboxrule}{.020cm}
\fbox{
$\begin{array}{c}
{\displaystyle \,}\\[-12pt]
\text{\!continuous big $q^{-1}$-Hermite}\\[1pt] 
\text{\phantom{\!}polynomials}
\end{array}$
}
}
};

\node (cbqH) at (-0.80,-17.40) {
{\setlength{\fboxrule}{.020cm}\fbox{
$\begin{array}{c}
{\displaystyle \,}\\[-12pt]
\text{\!continuous big $q$-Hermite}\\[1pt] 
\text{\phantom{\!}polynomials}
\end{array}$
}
}
};

\node (qiB) at (6.64,-17.40) {
{\setlength{\fboxrule}{.020cm}\fbox{
$\begin{array}{c}
{\displaystyle \,}\\[-12pt]
\text{\!$q^{-1}$-Bessel}\\[1pt] 
\text{\phantom{\!}functions}
\end{array}$
}
}
};

\node (ASC) at (2.20,-14.20) {
{\setlength{\fboxrule}{.020cm}\fbox{
$\begin{array}{c}
{\displaystyle \,}\\[-12pt]
\text{\!Al-Salam--Chihara}\\[1pt] 
\text{\phantom{\!}polynomials}
\end{array}$
}
}
};

\node (lqJ) at (8.90,-14.20) {
{\setlength{\fboxrule}{.020cm}\fbox{
$\begin{array}{c}
{\displaystyle \,}\\[-12pt]
\text{\!little $q$-Jacobi}\\[1pt] 
\text{\phantom{\!}functions/polynomials}
\end{array}$
}
}
};

\node (qiASC) at (15.90,-14.20) {
{\setlength{\fboxrule}{.020cm}\fbox{
$\begin{array}{c}
{\displaystyle \,}\\[-12pt]
\text{\!$q^{-1}$-Al-Salam--Chihara}\\[1pt] 
\text{\phantom{\!}polynomials}
\end{array}$}}};

\node (cdqH) at (2.50,-11.20) {
{\setlength{\fboxrule}{.020cm}\fbox{
$\begin{array}{c}
{\displaystyle \,}\\[-12pt]
\text{\!continuous dual $q$-Hahn}\\[1pt] 
\text{\phantom{\!}polynomials}
\end{array}$}}};

\node (bqJ) at (8.90,-11.20) {
{\setlength{\fboxrule}{.020cm}\fbox{
$\begin{array}{c}
{\displaystyle \,}\\[-12pt]
\text{\!big $q$-Jacobi}\\[1pt] 
\text{\phantom{\!}functions/polynomials}
\end{array}$}}};

\node (cdqiH) at (15.80,-11.20) {
{\setlength{\fboxrule}{.020cm}\fbox{
$\begin{array}{c}
{\displaystyle \,}\\[-12pt]
\text{\!continuous dual $q^{-1}$-Hahn}\\[1pt] 
\text{\phantom{\!}polynomials}
\end{array}$
}
}
};

\node (AW) at (8.90,-8.50) {
{\setlength{\fboxrule}{.020cm}\fbox{
$\begin{array}{c}
{\displaystyle \,}\\[-12pt]
\text{\!Askey--Wilson}\\[1pt] 
\text{\phantom{\!}polynomials}
\end{array}$
}
}
};

\draw[->,line width=1.5pt] (AW)--(bqJ);
\draw[->,line width=1.5pt] (AW)--(cdqH);
\draw[->,line width=1.5pt] (AW)--(cdqiH);
\draw[->,line width=1.5pt] (bqJ)--(lqJ);
\draw[->,line width=1.5pt] (lqJ)--(qiB);
\draw[->,line width=1.5pt] (lqJ)--(qB);
\draw[->,line width=1.5pt] (cdqH)--(ASC);
\draw[->,line width=1.5pt] (cdqiH)--(qiASC);
\draw[->,line width=1.5pt] (ASC)--(cbqH);
\draw[->,line width=1.5pt] (qiASC)--(cbqiH);
\draw[->,line width=1.5pt] (cbqH)--(cqH);
\draw[->,line width=1.5pt] (cbqiH)--(cqiH);
\draw[decoration={dashsoliddouble}, decorate,line width=1.5pt] (cbqH)--(qiB);
\draw[dashed,line width=1.5pt] (qB)--(cbqiH);
\draw[decoration={dashsoliddouble}, decorate,line width=1.5pt] (ASC)--(lqJ);
\draw[dashed,line width=1.5pt] (lqJ)--(qiASC);
\draw[decoration={dashsoliddouble}, decorate,line width=1.5pt] (cdqH)--(bqJ);
\draw[dashed,line width=1.5pt] (bqJ)--(cdqiH);

\end{tikzpicture}
\end{figure}
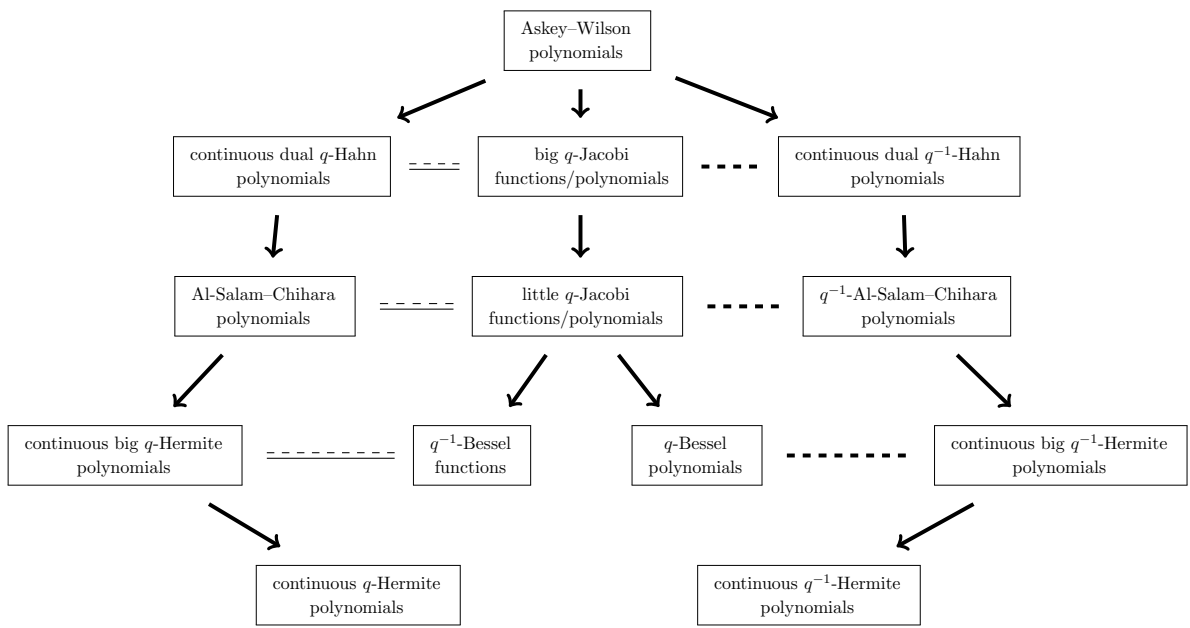

\medskip


\subsection{Duality relations for continuous dual $q$ and $q^{-1}$-Hahn polynomials}

\noindent First we describe the duality relations between the continuous dual $q$-Hahn and continuous dual $q^{-1}$-Hahn polynomials
with the big $q$-Jacobi polynomials. 
These duality relations are given in the following theorem. 

\begin{thm}
\label{thm311}
Let $q\in\CCddag$, $m,n\in\mathbb N_0$,
${\bf a}:=\{a_1,a_2,a_3\}$, $a_k\in\CCast$, 
$k,p,r,t\in{\bf 3}:=\{1,2,3\}$, such that $r\in{\bf 3}\setminus\{p\}$,
$t\in{\bf 3}\setminus\{p,r\}$,  $a_{pr}:=a_pa_r$, $a_{pt}:=a_pa_t$. Then
\begin{eqnarray}
&&\hspace{-0.20cm}
p_n\left[q^{m}a_p;{\bf a}|q\right]=
p_n\left[\frac{q^{-m}}{a_p};{\bf a}|q\right]=
p_n\left(\tfrac12\left(q^{m}a_p+\frac{q^{-m}}{a_p}\right);{\bf a}|q\right)\nonumber\\
&&\hspace{5.10cm}
=\frac{(a_{pr},a_{pt};q)_n}{a_p^n}\,
P_m\left(q^{-n};\frac{a_{pr}}{q},\frac{a_p}{a_r},\frac{a_{pt}}{q};q\right),
\label{dcdqHbqJabthm}\\
\nonumber\\
&&\hspace{-0.20cm}
p_n\left[\frac{q^{m}}{a_p};{\bf a}|q^{-1}\right]=
p_n[q^{-m}a_p;{\bf a}|q^{-1}]=
p_n\left(\tfrac12\left(\frac{q^{m}}{a_p}+q^{-m}a_p\right);{\bf a}|q^{-1}\right)\nonumber\\
&&\hspace{1cm}=q^{-2\binom{n}{2}-\binom{m}{2}}(a_1a_2a_3)^n\!\left(\frac{-a_p}{qa_t}\right)^{\!m}\!
\frac{\left(\frac{1}{a_{pr}},\frac{1}{a_{pt}};q\right)_n\!\left(\frac{qa_t}{a_p};q\right)_m}{\left(\frac{1}{a_{pt}};q\right)_{m}}
P_m\left(\frac{q^n}{a_{pr}};\frac{1}{qa_{pr}},\frac{a_r}{a_p},\frac{a_t}{a_p};q\right).
\label{dcdqiHbqJabthm}
\end{eqnarray}
\end{thm}
\begin{proof}
See cf.~\cite[(44)]{KoornwinderMazzocco2018}, \cite[\S4.3]{AtakishiyevKlimyk2006}.
\end{proof}

\noindent As an example of these duality relations, consider $(p,r,t)=(1,2,3)$ in which the duality relations reveal themselves through
\begin{eqnarray}
&&\hspace{-0.25cm}
p_n\left[q^{m}a;a,b,c|q\right]=
p_n\left[\frac{q^{-m}}{a};a,b,c|q\right]=
p_n\left(\tfrac12\left(q^{m}a+\frac{q^{-m}}{a}\right);a,b,c|q\right)\nonumber\\
&&\hspace{6.07cm}
=\frac{\left(ab,ac;q\right)_n}{a^n}\,
P_m\left(q^{-n};\frac{ab}{q},\frac{a}{b},\frac{ac}{q};q\right),
\label{dcdqHbqJab}\\
&&\hspace{-0.25cm}
p_n\left[\frac{q^{m}}{a};a,b,c|q^{-1}\right]=
p_n[q^{-m}a;a,b,c|q^{-1}]=
p_n\left(\tfrac12\left(\frac{q^{m}}{a}+q^{-m}a\right);a,b,c|q^{-1}\right)\nonumber\\
&&\hspace{2.07cm}
=q^{-2\binom{n}{2}-\binom{m}{2}}(abc)^n\!\left(\frac{-a}{qc}\right)^{\!m}\!
\frac{\left(\frac{1}{ab},\frac{1}{ac};q\right)_n\!\left(\frac{qc}{a};q\right)_m}{\left(\frac{1}{ac};q\right)_{m}}
P_m\left(\frac{q^n}{ab};\frac{1}{qab},\frac{b}{a},\frac{c}{a};q\right).
\label{dcdqiHbqJab}
\end{eqnarray}

\noindent
One can invert the above relation to compute the duality relation between big $q$-Jacobi polynomials and continuous dual $q^{-1}$-Hahn polynomials, namely,
\begin{eqnarray}
&&\hspace{-0.5cm}P_m(q^{n+1}a;a,b,c;q)\nonumber\\
&&\hspace{0.5cm}=q^{2\binom{n}{2}}
q^{\binom{m}{2}}\left(\frac{\sqrt{q^3a^3b}}{c}\right)^n
\frac{
(-qc)^m(\frac{qab}{c};q)_m}{(qa,\frac{qab}{c};q)_n(qc;q)_m}
p_n\left[q^{m+\frac12}\sqrt{ab};\frac{1}{\sqrt{qab}},\sqrt{\frac{b}{qa}},\frac{c}{\sqrt{qab}}\Bigg|q^{{-1}}\right].
\label{dualA}
\end{eqnarray}
Similarly, one can use the symmetry of the big $q$-Jacobi polynomials \eqref{bqJsym} to compute the following interesting duality relation
\begin{eqnarray}
&&\hspace{-0.5cm}P_m(q^{n+1}c;a,b,c;q)=P_m\left(q^{n+1}c;c,\frac{ab}{c},a;q\right)\nonumber\\
&&\hspace{0.5cm}=q^{2\binom{n}{2}}
q^{\binom{m}{2}}\left(\frac{c\sqrt{q^3b}}{\sqrt{a}}\right)^n
\frac{
(-qa)^m(qb;q)_m}{(qb,qc;q)_n(qa;q)_m}
p_n\left[q^{m+\frac12}\sqrt{ab};\frac{1}{\sqrt{qab}},\sqrt{\frac{a}{qb}},\frac{1}{c}\sqrt{\frac{ab}{q}}\Bigg|q^{{-1}}\right].
\label{dualC}
\end{eqnarray}

\noindent 
In the next result we present

the duality relation 
between the continuous dual $q$-Hahn 
polynomials and the big $q$-Jacobi function.
\begin{thm}
\label{thm3.33}
Let $n\in\mathbb N_0$, $\mu\in \C$, $q\in\CCddag$, $a,b,c\in\CCast$. 
Then, one has the following duality relations 
for the continuous dual $q$-Hahn polynomials 
with the big $q$-Jacobi function:
\begin{eqnarray}
&&\hspace{-3.8cm}p_n\left[\frac{q^{-\mu}}{a};a,b,c|q\right]
=
\frac{(ab,ac;q)_n}{a^n}
P_\mu\left(q^{-n};\frac{ab}{q},\frac{a}{b},\frac{ac}{q}\right),
\label{cdqHbqJfdl}\\
&&\hspace{-3.8cm}P_\mu(q^{-n};a,b,c;q)=\frac{(qab)^{\frac12n}}{(qa,qc;q)_n}p_n\left[q^{\mu+\frac12}\sqrt{ab},\sqrt{qab},\sqrt{\frac{qa}{b}},\frac{q^\frac12 c}{\sqrt{ab}}\bigg|q\right].
\label{cdqHbqJfdl2}
\end{eqnarray}
\end{thm}
\begin{proof}
The duality relation \eqref{cdqHbqJfdl} follows by comparing
the terminating basic hypergeometric 
${}_3\phi_2$ representation of the continuous dual $q$-Hahn polynomials \eqref{cdqH:def1} with $z=q^{-\mu}/a$ against the nonterminating basic hypergeometric ${}_3\phi_2$ representation of the big $q$-Jacobi function \eqref{bqJf1} (which is terminating because of the choice $x=q^{-n}$). One obtains \eqref{cdqHbqJfdl2} by replacing $(a,b,c)\mapsto(\sqrt{qab},\sqrt{\frac{qa}{b}},\frac{\sqrt{q}c}{\sqrt{ab}})$ in \eqref{cdqHbqJfdl} and solving
for the big $q$-Jacobi function.
This completes the proof.
\end{proof}

\subsection{Duality relations for $q$ and $q^{-1}$-Al-Salam--Chihara polynomials}

\noindent As well there exist duality relations between the Al-Salam--Chihara and $q^{-1}$-Al-Salam--Chihara polynomials
with the little $q$-Jacobi polynomials. 

\begin{thm}
\label{thm312}
Let $q\in\CCddag$, $m,n\in\mathbb N_0$, $a,b\in\CCast$. Then
\begin{eqnarray}
&&\hspace{-0.0cm}
Q_n\left[\frac{q^{-m}}{a};a,b|q\right]
=Q_n\left[q^{m}a;a,b|q\right]
=Q_n\left(\tfrac12\left(q^{m}a+\frac{q^{-m}}{a}\right);a,b|q\right)\nonumber\\
&&\hspace{2.9cm}=q^{\binom{m}{2}}\frac{(-ab)^m}{a^n}
\frac{(ab;q)_n\left(\frac{qa}{b};q\right)_m}{(ab;q)_m}\,p_m\left(\frac{q^{-n}}{ab};\frac{a}{b},\frac{ab}{q};q\right),
\label{dASClqJa}\\
&&\hspace{-0.0cm}
Q_n\left[q^{m}b;a,b|q\right]
=Q_n\left[\frac{q^{-m}}{b};a,b|q\right]
=Q_n\left(\tfrac12\left(q^{m}b+\frac{q^{-m}}{b}\right);a,b|q\right)\nonumber\\
&&\hspace{2.9cm}=q^{\binom{m}{2}}\frac{(-ab)^m}{b^n}
\frac{(ab;q)_n\left(\frac{qb}{a};q\right)_m}{(ab;q)_m}\,p_m\left(\frac{q^{-n}}{ab};\frac{b}{a},\frac{ab}{q};q\right),\\
&&\hspace{-0.0cm}
Q_n\left[\frac{q^{m}}{a};a,b|q^{-1}\right]
=Q_n[q^{-m}a;a,b|q^{-1}]
=Q_n\left(\tfrac12\left(\frac{q^{m}}{a}+q^{-m}a\right);a,b|q^{-1}\right)\nonumber\\
&&\hspace{3.05cm}
=q^{-\binom{n}{2}-\binom{m}{2}}(-b)^n\left(-\frac{a}{qb}\right)^m
\frac{\left(\frac{1}{ab};q\right)_n\left(\frac{qb}{a};q\right)_m}
{\left(\frac{1}{ab};q\right)_m}\,
p_m\left(q^n;\frac{b}{a},\frac{1}{qab};q\right),
\label{dqiASClqJa}\\
&&\hspace{-0.0cm}
Q_n\left[\frac{q^{m}}{b};a,b|q^{-1}\right]
=Q_n[q^{-m}b;a,b|q^{-1}]
=Q_n\left(\tfrac12\left(\frac{q^{m}}{b}+q^{-m}b\right);a,b|q^{-1}\right)\nonumber\\
&&\hspace{3.05cm}
=q^{-\binom{n}{2}-\binom{m}{2}}(-a)^n\left(-\frac{b}{qa}\right)^m
\frac{\left(\frac{1}{ab};q\right)_n\left(\frac{qa}{b};q\right)_m}
{\left(\frac{1}{ab};q\right)_m}\,
p_m\left(q^n;\frac{a}{b},\frac{1}{qab};q\right).
\end{eqnarray}
\end{thm}
\begin{proof}
See \cite[(75)]{KoornwinderMazzocco2018}, \cite[p.~8]{Groenevelt2021}.
\end{proof}

\noindent In the next result we present
the duality relation between
the Al-Salam--Chihara polynomials with the little $q$-Jacobi function.
\begin{thm}
\label{thm3.35}
Let $n\in\mathbb N_0$, $\mu\in \C$, 
$q\in\CCdag$, $a,b,c\in\CCast$. 
Then, one has the following duality relations for the
Al-Salam--Chihara polynomials with the little $q$-Jacobi function:
\begin{eqnarray}
&&\hspace{-1.9cm}Q_n\left[\frac{q^{-\mu}}{a};a,b|q\right]
=\frac{(ab;q)_n}{a^n}
\frac{(\frac{qa}{b},\frac{q}{ab};q)_\infty}{(q^{\mu+1}\frac{a}{b},q^{1-\mu}\frac{1}{ab};q)_\infty}
p_\mu\left(\frac{q^{-n}}{ab};\frac{a}{b},\frac{ab}{q};q\right),
\label{ASClqJfdl}\\
&&\hspace{-1.9cm}p_\mu\left(\frac{q^{-1-n}}{b};a,b;q\right)=\frac{(qab)^{\frac12n}}{(qb;q)_n}
\frac{(q^{\mu+1}a,\frac{q^{-\mu}}{b};q)_\infty}{(qa,\frac{1}{b};q)_\infty}
Q_n\left[q^{\mu+\frac12}\sqrt{ab};\sqrt{qab},\sqrt{\frac{qb}{a}}\bigg|q\right].
\label{ASClqJfdl2}
\end{eqnarray}
\end{thm}
\begin{proof}
To derive the duality relation \eqref{ASClqJfdl}, use 
the ${}_2\phi_1$ representation of the Al-Salam--Chihara 
polynomials \eqref{cdqH:def1} with $a$ and $b$ 
interchanged due to their symmetry.
Then, the following values are considered: 
$(a,b,c,z)=(q^{-n},az,q^{1-n}z/b,q/(bz))$ use 
\cite[(III.2)]{GaspRah} 
which produces a nonterminating representation of the
Al-Salam--Chihara polynomials.
Then replacing $z=q^{-\mu}/a$,
the duality relation then follows by comparing
the resulting expression 
against the nonterminating basic hypergeometric ${}_2\phi_1$ representation of the little $q$-Jacobi function \eqref{lqJf1}.
To obtain \eqref{ASClqJfdl2}, one must replace $(a,b)\mapsto(\sqrt{qab},\sqrt{qb/a})$ in
\eqref{ASClqJfdl} and solve for the little $q$-Jacobi function. 
This completes the proof.
\end{proof}

\subsection{Duality relations for continuous big $q$ and big $q^{-1}$-Hermite polynomials}

\noindent In the next result we present
the duality relation between the continuous big 
$q$-Hermite polynomials with the 
$q^{-1}$-Bessel function.
\begin{thm}
\label{thm3.36}
Let $n\in\mathbb N_0$, $\mu\in \C$, $q\in\CCddag$, $a,b,c\in\CCast$. 
Then, one has the following duality relations 
for the continuous big $q$-Hermite polynomials with the $q^{-1}$-Bessel function:
\begin{eqnarray}
&&\hspace{-5cm}H_n\left[q^\mu a;a|q\right]=q^{2\binom{\mu}{2}}a^{-n}(qa^2)^\mu
\,y_\mu\left(q^{-n};-\frac{1}{a^2};q^{-1}\right),
\label{dHnymu}\\
&&\hspace{-5cm}y_\mu(q^{-n};a;q)=q^{-2\binom{\mu}{2}}(-a)^{-\frac12 n}\left(-\frac{a}{q}\right)^\mu
H_n\left[\frac{q^{\mu}}{\sqrt{-a}};\frac{1}{\sqrt{-a}};q\right],
\label{dHnymu2}
\end{eqnarray}
where in \eqref{dHnymu2}, the principal branch of the square root is taken.
\end{thm}
\begin{proof}
The duality relation between the continuous big $q$-Hermite polynomials and the 
$q^{-1}$-Bessel function can be found, for instance, by comparing the basic hypergeometric representations \eqref{cbqH:def1}, 
\eqref{qiBf}.
To obtain \eqref{dHnymu2}, one must invert \eqref{dHnymu}.
This completes the proof.
\end{proof}

\noindent 
The duality relation between the continuous $q^{-1}$-Hermite polynomials and the $q$-Bessel polynomials is given as follows.
\begin{thm}
Let $n,m\in\N_0$, $q\in\CCddag$, $a\in\CCast$
\begin{eqnarray}
&&\hspace{-6cm}H_n[q^{-m}a;a|q^{-1}]=q^{-\binom{m}{2}}a^{-n}\left(\frac{a^2}{q}\right)^m y_m(q^n;-\tfrac{1}{a^2};q),
\label{dHqinym}\\
&&\hspace{-6cm}y_m(q^n;a;q)=q^{2\binom{m}{2}}\frac{(-qa)^m}{(-a)^{\frac12n}}
H_n\left[\frac{q^{-m}}{\sqrt{-a}};\frac{1}{\sqrt{-a}};q^{-1}\right],
\label{dHqinym2}
\end{eqnarray}
where in \eqref{dHqinym2}, the principal branch of the square root is taken.
\end{thm}

\begin{proof}
For instance, comparing the terminating basic hypergeometric series representations given by \eqref{cbqiH:2} and \eqref{qB:7} completes the proof. In order to obtain \eqref{dHqinym2}, one must invert \eqref{dHqinym}.
\end{proof}

\section{Orthogonality relations for the $q$ and $q^{-1}$-symmetric and their dual families}
\label{sec:3.8}

There is an interesting history of the duality relations described in the previous subsection and the corresponding orthogonality relations which we present below. Some important literature corresponding to these include 
Rosengren (2000) \cite{RosengrenCONM2000} 
and also 
\cite{AtakishiyevKlimyk2006,Groenevelt2021,Koorwinder2018}
where dual orthogonality relations are discussed.
For instance, the dual orthogonality relation 
between little $q$-Jacobi polynomials and $q^{-1}$-Al-Salam--Chihara polynomials correspond to the (visibly 
dual) relations \cite[(4.5), (4.6)]{RosengrenCONM2000}. 
This dual orthogonality relation is also discussed by 
Groenevelt (2004), see 
\cite[see especially Remark 3.1]{Groenevelt2004}.
Discrete orthogonality of continuous 
dual $q^{-1}$-Hahn polynomials was found by Rosengren in 
\cite[(4.16)]{RosengrenCONM2000}. 
To see what it is explicitly, one needs to 
use \cite[(4.16), (4.10), Proposition 4.3]{RosengrenCONM2000}. 
Although it is not explained in \cite{RosengrenCONM2000}, 
this orthogonality is dual to the big $q$-Jacobi 
polynomials and can also be obtained from the 
orthogonality of $q$-Racah polynomials by letting 
$N\to\infty$. The same orthogonality relation was 
written down in a readable way by Atakishiyev and 
Klimyk (2004) \cite[Section 8]{AtakishiyevKlimyk2004}.
A related dual orthogonality relation between the 
big $q$-Jacobi functions and a system containing 
the continuous 
dual $q^{-1}$-Hahn polynomials and $q$-Bessel 
functions was discussed in Koelink and Stokman (2001) 
\cite{KoelinkStokman2001}. 

\medskip
\noindent For a nice discussion of the existence 
of dual orthogonality in the general setting, one
might study in detail 
\cite[Chapter 2]{Ismail:2009:CQO}.
As pointed out by Ismail in \cite[Section 2.5]{Ismail:2009:CQO}, the orthogonality and the existence of dual orthogonality (duality) in \cite[(2.5.1), (2.5.3)]{Ismail:2009:CQO}, is guaranteed in the finite case.
This fact may have been discovered much earlier by Chebyshev \cite{Ismailpriv2023}.
The existence of dual orthogonality in the infinite case (the case where there are an infinite number of zeros) is demonstrated by Markov's theorem \cite[Theorem 2.6.1]{Ismail:2009:CQO}
which requires the uniqueness 
of the measure of orthogonality.
For the $q^{-1}${-symmetric} polynomials,
the measure of orthogonality is 
not unique and one has an indeterminate moment problem \cite[Chapter 21]{Ismail:2009:CQO}. For the $q^{-1}$-Hermite polynomials, this
is completely described in \cite[Chapter 21]{Ismail:2009:CQO} (see also \cite{IsmailMasson1994}). For more advanced cases including continuous big $q^{-1}$-Hermite, $q^{-1}$-Al-Salam--Chihara and continuous dual $q^{-1}$-Hahn 
polynomials, 
the indeterminate moment problem is yet unsolved.


\medskip
\noindent 
Let $n,n'\in\N_0$, $x\in\CC$, ${\bf a}\in\CC$, be a set of parameters. 
Consider a sequence of orthogonal polynomials $\{p_n\}$, orthogonal with respect to either a continuous measure $\dd\mu(x)$ or a discrete measure $\dd\mu_n$, with support (either bounded or unbounded) given respectively by $A$ and $B$. Further, consider situations 
where the measure can be written respectively in terms of a continuous weight function 
$\dd\mu(x)={\sf w}(x;{\bf a})\,\dd x$ or a discrete
weight function 
$\dd\mu_n={\sf w}_n({\bf a})$.
A sequence of orthogonal polynomials are orthogonal if and only if it satisfies a three-term recurrence relation with coefficients $A_n, B_n, C_n\in\RR$ given by \cite[\S 2]{AskeyIsmail84}
\begin{equation}
p_{n+1}(x;{\bf a})=(A_nx+B_n)\,p_n(x;{\bf a})-C_n\,p_{n-1}(x;{\bf a}).
\end{equation}
We may assume that $p_{-1}=0$, $p_0=1$.
Define ${\sf W}({\bf a})$, often referred to as the {\it total mass}, as the integral or sum of the weight function over its support, given respectively by
\begin{eqnarray}
&&\hspace{-11.8cm}{\sf W}({\bf a})=\int_A w(x;{\bf a})\,\dd x,\\
&&\hspace{-11.8cm}{\sf W}({\bf a})=\sum_{n\in B} w_{n}({\bf a}).
\end{eqnarray}
In this case, one has the following continuous or discrete orthogonality relations
respectively, 
\begin{eqnarray}
&&\hspace{-8.5cm}\int_A p_{n}(x;{\bf a}) p_{n'}(x;{\bf a}) \,{\sf w}(x;{\bf a})\,\dd x=h_n({\bf a}) \delta_{n,n'},\\
&&\hspace{-8.5cm}\sum_{n\in B} p_{n}({\bf a}) p_{n'}({\bf a}) \,{\sf w}_n({\bf a})=h_n({\bf a}) \delta_{n,n'},
\end{eqnarray}
where $\delta_{n,m}$ is the Kronecker delta, the $L^2$ (or $\ell^2$)-norm of orthogonality is given by \cite[(2.5)]{AskeyIsmail84}
\begin{equation}
h_n({\bf a})={\sf W}({\bf a})\frac{A_0}{A_n}\prod_{k=1}^n C_k.
\end{equation}

Consider a sequence of eigenfunctions $\psi_n$ orthogonal with respect to a continuous or discrete weight function ${\sf w}$.
In the case of continuous orthogonality
\begin{eqnarray}
&&\hspace{-8cm}\int_a^b \psi_{n}(x;{\bf a}) \psi_{n'}(x;{\bf a}) \,{\sf w}(x;{\bf a})\,\dd x=h_n({\bf a}) \delta_{n,n'}.
\end{eqnarray}
If there is completeness of the eigenfunctions $\psi_n$ in some space of functions, then one often has a corresponding closure relation (see e.g., \cite{CohlCostasSantos22c}, \cite[Theorem 2.1]{Ismailetal2022}).
then, the closure relation is given by
\begin{eqnarray}
&&\hspace{-8.8cm}\sum_{n=0}^\infty \frac{1}{h_n({\bf a})}\psi_n(x;{\bf a})\psi_{n}(y;{\bf a})=\frac{\delta(x-y)}{{\sf w}(x;{\bf a})},
\label{closC}
\end{eqnarray}
where $\delta$ is the Dirac delta distribution.
In the case of an infinite discrete orthogonality
\begin{eqnarray}
&&\hspace{-8.2cm}\sum_{n=0}^\infty \psi_{n,m}({\bf a}) \psi_{n,m'}({\bf a}) \,{\sf w}_n({\bf a})=h_m({\bf a}) \delta_{m,m'},
\end{eqnarray}
then, the closure relation is given by
\begin{eqnarray}
&&\hspace{-8.5cm}\sum_{m=0}^\infty \frac{1}{h_m({\bf a})}\psi_{m,n}({\bf a})\psi_{m,n'}({\bf a})=\frac{\delta_{n,n'}}{{\sf w}_{n}({\bf a})}.
\label{closD}
\end{eqnarray}

\subsection{The Askey--Wilson polynomials}
\medskip
\noindent
The Askey--Wilson polynomials satisfy
the following continuous orthogonality
relation. 

\begin{thm}{\cite[(14.1.2)]{Koekoeketal}}
\label{AWorth}
Let $m,n\in\N_0$, $q\in\CCdag$, $x=\cos\theta$, $a,b,c,d$ are real, or occur in complex conjugate pairs if complex, $q^nab,q^nac,q^nad,q^nbc,q^nbd,q^ncd\not\in\Omega_q$ and $\max(|a|,|b|,|c|,|d|)<1$. Then 
the Askey--Wilson polynomials $p_n(x;{\bf a}|q)$ satisfy the following continuous orthogonality relation 
\cite[(14.1.2)]{Koekoeketal}
\begin{equation}
\int_0^\pi p_m(x;{\bf a}|q)p_n(x;{\bf a}|q)w_q(x;{\bf a})\,
{\mathrm d}\theta=h_n({\bf a}|q)\delta_{m,n},
\label{AWO}
\end{equation}
where
\begin{eqnarray}
&&\hspace{-6.9cm}
w_q(\cos\theta;{\bf a}):=
\frac{(\expe^{\pm 2i\theta};q)_\infty}
{({\bf a}\expe^{\pm i\theta};q)_\infty}=
\label{AWw}
\frac{(\pm\expe^{\pm i\theta},\pm q^\frac12 \expe^{\pm i\theta};q)_\infty}
{({\bf a}\expe^{\pm i\theta};q)_\infty},
\end{eqnarray}
\begin{eqnarray}
&&\hspace{-6cm}h_n({\bf a}|q):=\frac{2\pi(q^{n-1}abcd;q)_n(q^{2n}abcd;q)_\infty}{(q^{n+1},q^nab,q^nac
,q^nad,q^nbc,q^nbd,q^ncd;q)_\infty},
\label{AWn}
\end{eqnarray}
where ${\mathbf a}:=\{a,b,c,d\}$.
\end{thm}
\begin{proof}
The second equality for $w_q$ is due to
\eqref{sq}.
The proof of this orthogonality relation can be found in many places including the proof of \cite[Theorem 15.2.1]{Ismail:2009:CQO} (see also the Askey--Wilson monograph \cite[\S2]{AskeyWilson85}).
\end{proof}

\noindent Orthogonality relations for the Askey--Wilson polynomials in
the complex plane are studied systematically in 
\cite[\S 3]{MR2832754}.

\subsection{The continuous dual $q$ and $q^{-1}$-Hahn polynomials}

\medskip
\noindent
The continuous dual $q$-Hahn polynomials $p_n(x;{\bf a}|q)$ satisfy the following continuous orthogonality relation. 

\begin{thm}{\cite[(14.3.2)]{Koekoeketal}}
Let $m,n\in\N_0$, $q\in\CCdag$, $x=\cos\theta$, and $a,b,c$ are real, or one is real, and the other two are complex conjugates, $q^nab,q^nac,q^nbc\not\in\Omega_q$,
$\max(|a|,|b|,|c|)<1$. Then 
the continuous dual $q$-Hahn polynomials satisfy the following continuous orthogonality relation 
\cite[(14.3.2)]{Koekoeketal}
\begin{equation}
\int_0^\pi p_m(x;{\bf a}|q)p_n(x;{\bf a}|q)w_q(x;{\bf a})\,
{\mathrm d}\theta=h_n({\bf a}|q)\delta_{m,n},
\label{cdqHO}
\end{equation}
where
\begin{eqnarray}
&&\hspace{-3.8cm}
w_q(\cos\theta;{\bf a}):=
\frac{(\expe^{\pm 2i\theta};q)_\infty}
{({\bf a}\expe^{\pm i\theta};q)_\infty},
\quad
h_n({\bf a}|q):=\frac{2\pi}{(q^{n+1},q^nab,q^nac
,q^nbc;q)_\infty}.
\label{cdqHwn}
\end{eqnarray}
\end{thm}
\medskip
\begin{proof}
The proof of this orthogonality relation follows by using the proof of the orthogonality relation for Askey--Wilson polynomials, Theorem \ref{AWorth}, with $d\to 0$.
\end{proof}


\medskip
\noindent Continuous orthogonality for the continuous 
dual $q^{-1}$-Hahn polynomials was worked out by Ismail 
and collaborators in \cite{IsmailZhangZhou2022}. 
There 
they showed that these polynomials form a symmetric 
(in three parameters) infinite family of orthogonal
polynomials with continuous orthogonality relation given as follows.

\begin{thm}
\label{theo33}
Let $n,m\in\N_0$, $q\in\CCdag$, 
$a,b,c\in\CCast$, $ab,ac,bc\not\in\Upsilon_q$. 
Then
\begin{eqnarray}
\int_{0}^{i\infty}p_n(\tfrac12(z+z^{-1});a,b,c|q^{-1})
p_{m}(\tfrac12(z+z^{-1});a,b,c|q^{-1})
w(z;a,b,c|q)\,\dd z=h_n(a,b,c|q)\delta_{m,n},
\label{cdqiHO}
\end{eqnarray}
where
\begin{equation}
w(z;a,b,c|q):=\frac{(qaz^\pm,qbz^\pm,qcz^\pm;q)_\infty}{z(qz^{\pm 2};q)_\infty},
\label{cdqiHw}
\end{equation}
\begin{equation}
h_n(a,b,c|q):=q^{-4\binom{n}{2}}
(q,qab,qac,qbc;q)_\infty(q,\tfrac{1}{ab},\tfrac{1}{ac},\tfrac{1}{bc};q)_n
\left(\frac{a^2b^2c^2}{q}\right)^n\,\log q^{-1}.
\label{cdqiHn}
\end{equation}
\label{cdqiHiOt}
\end{thm}

\begin{proof}
See \cite[(5.4)]{IsmailZhangZhou2022}, 
and we have rewritten their result 
by setting
$(z,t_1,t_2,t_3)\mapsto(iz,iqa,iqb,iqc)$,
and renormalizing their ${}_3\phi_2$ polynomials properly.
Next, we shall prove that the weight function $w(z;a,b,c,|q)$ decays to zero faster than any polynomial of $x=(z+z^{-1})/2$ as $|z|\to\infty$ or $|z|\to0$.
For any large $z$, we define
\begin{equation}
 N(z):=\left\lfloor\frac{\log|z|}{\log|q^{-1}|}\right\rfloor,~~
 r(Z):=q^{N(z)}z,
\end{equation}
where $\lfloor x\rfloor$ 
is the greatest integer less than or 
equal to $x$. As $|z|\to\infty$, we have
$N(z)\to\infty$
and
\[
|r(z)|=|q|^{N(z)}|z|=\expe^{\log|z|+N(z)\log|q|}\in[|q|,|q|^{-1}].
\]
For any $a\in\CCast$, we have
\begin{equation}
 \lim_{|z|\to\infty}\frac{\log(az;q)_\infty}{N(z)^2}
 =\lim_{|z|\to\infty}\frac{\log(ar(z)q^{-N(z)};q)_{N(z)}}{N(z)^2}
 =\lim_{|z|\to\infty}\frac{\log q^{-N(z)^2/2}}{N(z)^2}
 =\frac{\log q^{-1}}{2}.
\end{equation}
Consequently, we obtain
\begin{equation}
 \lim_{|z|\to\infty}\frac{\log w(z;a,b,c|q)}{N(z)^2}
 =\frac{-3\log q}{2}+2\log q=\frac{\log q}{2};
\end{equation}
namely, 
$\log |w(z;a,b,c|q)|\sim(\log|z|)^2/(2\log|q|)\to-\infty$ as $|z|\to\infty$.
A similar argument shows that this asymptotic formula also holds as $|z|\to0$. Hence,
the integral on the left-hand side of \eqref{cdqiHO} converges for all nonnegative integers $n$ and $m$ and for all complex numbers $a$, $b$, and $c$.
\end{proof}

\noindent 
The total mass of the above orthogonality relation is connected to the Askey $q$-beta integral (\cite{MR993347}, \cite[Exercise 6.10]{GaspRah})
\begin{eqnarray}
&&\hspace{-2.0cm}\int_0^{i\infty}
\frac{(-az^\pm,-bz^\pm,-cz^\pm,-dz^\pm;q)_\infty}{(qz^{\pm 2}
;q)_\infty}\,\frac{\dd z}{z}
=\frac{(q,\frac{ab}{q},\frac{ac}{q},\frac{ad}{q},\frac{bc}{q},\frac{bd}{q},\frac{cd}{q};q)_\infty}{(\frac{abcd}{q^3};q)_\infty}\log q^{-1},
\end{eqnarray}
where $q\in\CCdag$, $|abcd|<|q|^3$.

\medskip
\noindent 
There exists an infinite discrete orthogonality relation for
the continuous dual $q^{-1}$-Hahn 
polynomials (see 
cf.~\cite[(4.30)]{AtakishiyevKlimyk2006} 
and Theorem \ref{cor:3.4} below).
To investigate the region 
of convergence for the parameters 
involved in the following infinite 
discrete orthogonality relation, we 
will need the following result.

\begin{lem}
Let $n\in\N_0$, $q\in\CCdag$, $a,b,c\in\CCast$. Then one has 
as $m\to\infty$, that
\begin{eqnarray}
&&\hspace{-2cm}p_n[q^{-m}a;a,b,c|q^{-1}]\sim
q^{-nm}a^n.
\end{eqnarray}
\label{lem346}
\end{lem}

\begin{proof}
Start with 
\eqref{cdqiH:3} 
and replace $z=q^{-m}a$, then 
one has as $m\to\infty$,
\[
p_n[q^{-m}a;a,b,c|q^{-1}]\sim q^{nm-\binom{n}{2}}(-a^2b)^n(\tfrac{1}{ab};q)_n\qhyp21{q^{-n},0}{\frac{1}{ab}}{q,q},
\]
which can be summed using \eqref{limqChu}.
The asymptotic expression can then easily be obtained by expressing the result as a ratio of two $q$-shifted factorials and then using \eqref{qPochneg} twice and \eqref{binomid}. 
\end{proof}

\begin{thm}
Let $n,n'\in\N_0$, $q\in\CCdag$, $a,b,c\in\CCast$. Then
\begin{eqnarray}
&&\hspace{-0.9cm}\sum_{m=0}^\infty
q^{\binom{m}{2}}(-qbc)^m
\frac{(\frac{q}{a^2};q)_{2m}(\frac{1}{ab},\frac{1}{ac},
\frac{1}{a^2};q)_m}{(\frac{1}{a^2};q)_{2m}(q,\frac{qb}{a},\frac{qc}{a}
;q)_m}
\,p_n\left[q^{-m}a;a,b,c|q^{-1}\right]
p_{n'}\left[q^{-m}a;a,b,c|q^{-1}\right]\nonumber\\
&&\hspace{4.8cm}=
q^{-4\binom{n}{2}}\left(\frac{a^2b^2c^2}{q}\right)^n \frac{(\frac{q}{a^2},qbc;q)_\infty(q,\frac{1}{ab},\frac{1}{ac},\frac{1}{bc};q)_n}{(\frac{qb}{a},\frac{qc}{a};q)_\infty}\delta_{n,n'}.
\label{AKporth}
\end{eqnarray}
\end{thm}

\begin{proof}
The orthogonality relation \eqref{AKporth}
is obtained from \cite[(4.30)]{AtakishiyevKlimyk2006}.
Consider $n=n'$. The left-hand side of \eqref{AKporth} as ${\sf R}_{n}(a,b,c;q)$, then 
using Lemma \ref{lem346}, we have 
as $m\to\infty$, the summand behaves like 
\begin{equation}
{\sf r}_{m,n}(a,b,c;q)\sim a^{2n}
\frac{(\frac{1}{ab},\frac{1}{ac},\frac{q}{a^2};q)_\infty}{(q,\frac{qb}{a},\frac{qc}{a};q)_\infty}
q^{\binom{m}{2}}
\left(-\frac{bc}{q^{2n-1}}\right)^m,
\end{equation}
where $\sum_m{\sf r}_{m,n}(a,b,c;q)=
{\sf R}_{n}(a,b,c;q)$. 
Hence, 
by using the direct comparison test 
${\sf R}_{n}(a,b,c;q)$ converges 
since the infinite series associated with 
\eqref{AKporth} is convergent. Therefore, 
it is convergent for all values of $a,b,c$ 
and $n\in\N_0$. This completes the proof.
\end{proof}

\noindent 
There is another orthogonality relation 
for the continuous dual $q^{-1}$-Hahn 
polynomials, which is obtained through 
duality from the $q$-integral orthogonality 
of the big $q$-Jacobi polynomials 
\eqref{bqJO}. In order to study the 
convergence properties of these 
polynomials, we will need the following 
result.

\begin{lem}
Let $n\in\N_0$, $q\in\CCdag$, $a,b,c\in\CCast$. Then
as $n\to\infty$, one has 
\begin{equation}
p_n(x;a,b,c|q^{-1})\sim q^{-2\binom{n}{2}}
(abc)^n(\tfrac{1}{ab},\tfrac{1}{ac};q)_\infty
\qhyp22{\frac{z^\pm}{a}}{\frac{1}{ab},\frac{1}{ac}}{q,\frac{1}{bc}}.
\label{pn-large-n}
\end{equation}
The error term of the asymptotic formula is of order ${\mathcal{O}}\left(\frac{|q|^{-2\binom{n}{2}}|abc|^n}{n}\right)$.
\label{lem46}
\end{lem}

\begin{proof}
Start by using the generating function of continuous dual $q^{-1}$-Hahn polynomials ${\sf G}(t;a,b,c|q)$ given by 
\eqref{cdqiHgf2}.
Since
\begin{equation}
{\sf G}_1(a,b,c|q):= \lim_{t\to(abc)^{-1}}(1-abct){\sf G}(t;a,b,c|q)={1\over(q;q)_\infty}{}_2\phi_2\left(\begin{array}{c}
 \frac{z^{\pm}}{a}\\ \frac{1}{ab},\frac{1}{ac}
 \end{array};q,\frac{1}{bc}\right),
\end{equation}
we obtain from Darboux's method that
\begin{equation}
 {q^{2\binom{n}{2}}p_n(x;a,b,c|q^{-1})\over(abc)^n(\frac{1}{ab},\frac{1}{ac};q)_n}=(q;q)_\infty {\sf G}_1(a,b,c|q)+\mathcal{O}\left(\frac{1}{n}\right)
 ={}_2\phi_2\left(\begin{array}{c}
 \frac{z^{\pm}}{a}\\ \frac{1}{ab},\frac{1}{ac}
 \end{array};q,\frac{1}{bc}\right)+\mathcal{O}\left(\frac{1}{n}\right).
\end{equation}
This proves \eqref{pn-large-n}.
\end{proof}

In order to determine the required constraints on the parameters for the following infinite discrete orthogonality relation for continuous dual $q^{-1}$-Hahn polynomials, we will need the following asymptotic results as $n\to\infty$ which both follow from the above lemma.

\begin{lem}
Let $n\in\N_0$, $q\in\CCdag$, $a,b,c\in\CCast$. Then
as $n\to\infty$, one has 
\begin{eqnarray}
&&\hspace{-2cm}p_n[q^{-m};a,b,c|q^{-1}]
\sim q^{-2\binom{n}{2}}(abc)^n(\tfrac{1}{ab},\tfrac{1}{ac};q)_\infty \qhyp22{q^{-m},\frac{q^m}{a^2}}{\frac{1}{ab},\frac{1}{ac}}{q,\frac{1}{bc}},
\label{pn-large-n1}\\
&&\hspace{-2cm}p_n[q^{-m};a,\tfrac{1}{qb},\tfrac{1}{qc}|q^{-1}]
\sim q^{-2\binom{n}{2}}\left(\frac{a}{q^2bc}\right)^n(\tfrac{qb}{a},\tfrac{qc}{a};q)_\infty \qhyp22{q^{-m},\frac{q^m}{a^2}}{\frac{qb}{a},\frac{qc}{a}}{q,q^2bc}.
\label{pn-large-n2}
\end{eqnarray}
\label{lem47}
\end{lem}

\begin{proof}
Start with Lemma \ref{lem46}, replace $z=q^{-m}a$ for \eqref{pn-large-n1} and then replace $(b,c)\mapsto({1}/({qb}),{1}/({qc}))$ in order to obtain \eqref{pn-large-n2}.
\end{proof}

It is given in the following theorem. Note, however, that the orthogonality relation involves two different families of continuous dual $q^{-1}$-Hahn polynomials.
\begin{thm}
Let $m,m'\in\N_0$, $q\in\CCdag$, 
Let $a,b,c\in\CCast$. Then
\begin{eqnarray}
&&\hspace{1.0cm}
\frac{(\frac{1}{qbc};q)_\infty}{(\frac{1}{ab},\frac{1}{ac};q)_\infty}\left(\frac{b}{a}\right)^{2m}\frac{(\frac{1}{ab},\frac{1}{ab};q)_m}{(\frac{qb}{a},\frac{qb}{a};q)_m}
\sum_{n=0}^\infty
\frac{q^{4\binom{n}{2}}\left(\frac{q}{a^2b^2c^2}\right)^n}{(q,\frac{1}{ab},\frac{1}{ac},\frac{1}{bc};q)_n}
p_n\left[\frac{q^m}{a};a,b,c\bigg|q^{-1}\right]\!
p_n\left[\frac{q^{m'}}{a};a,b,c\bigg|q^{-1}\right]
\nonumber\\
&&\hspace{0.2cm}+
\frac{(qbc;q)_\infty}{(\frac{qb}{a},\frac{qc}{a};q)_\infty}\left(\frac{1}{qac}\right)^{2m}\frac{(\frac{qc}{a},\frac{qc}{a};q)_m}{(\frac{1}{ac},\frac{1}{ac};q)_m}
\sum_{n=0}^\infty
\frac{q^{4\binom{n}{2}}\left(\frac{q^5b^2c^2}{a^2}\right)^n}{(q,\frac{qb}{a},\frac{qc}{a},q^2bc;q)_n}
p_n\left[\frac{q^m}{a};a,\frac{1}{qb},\frac{1}{qc}\bigg|q^{-1}\right]\!
p_n\left[\frac{q^{m'}}{a};a,\frac{1}{qb},\frac{1}{qc}\bigg|q^{-1}\right]
\nonumber\\
&&\hspace{1.5cm}=q^{-\binom{m}{2}}\left(-\frac{b}{qa^2c}\right)^m
\frac{(\frac{q}{a^2},qbc,\frac{1}{qbc};q)_\infty}{(\frac{1}{ab},\frac{1}{ac},\frac{qb}{a},\frac{qc}{a};q)_\infty}
\frac{(\frac{1}{a^2};q)_{2m}(q,\frac{1}{ab},\frac{qc}{a};q)_m}{(\frac{q}{a^2};q)_{2m}(\frac{1}{a^2},\frac{1}{ac},\frac{qb}{a};q)_m}\delta_{m,m'}.
\label{DcdqiHO}
\end{eqnarray}
\end{thm}
\begin{proof}
Start with the $q$-integral orthogonality of the big $q$-Jacobi polynomials \eqref{bqJO} and express the $q$-integral as two infinite series using \eqref{qint}.
One may express the big $q$-Jacobi polynomial with argument 
$q^{n+1}a$ as a continuous dual $q^{-1}$-Hahn polynomial 
using \eqref{dualA}. For the big $q$-Jacobi polynomial with argument $q^{n+1}c$, one may express it as a continuous dual $q^{-1}$-Hahn polynomial using the duality relation \eqref{dualC} which takes advantage of the symmetry of big $q$-Jacobi polynomials given by \eqref{bqJsym}. Then, one may insert these two duality relations and simplify. Finally, after making the global replacement 
\begin{equation}
(a,b,c)\mapsto \left(\frac{b}{a},\frac{1}{qab},\frac{1}{qac}\right),
\end{equation}
and simplifying, this produces the orthogonality relation. In order to obtain the range of parameters for convergence, consider $m=m'$. Define the two terms on the left-hand side of \eqref{DcdqiHO} as ${\sf V}_{m}(a,b,c;q)$ and 
${\sf W}_{m}(a,b,c;q)$ then 
using Lemma \ref{lem47}, we have 
as $n\to\infty$, both summands behave like 
\begin{equation}
{\sf v}_{n,m}(a,b,c;q)\sim 
\frac{(\frac{1}{ab},\frac{1}{ac};q)_\infty}{(q,\frac{1}{bc};q)_\infty}
\left[\qhyp22{q^{-m},\frac{q^m}{a^2}}{\frac{1}{ab},\frac{1}{ac}}{q,\frac{1}{bc}}\right]^2
q^n,
\end{equation}
and
\begin{equation}
{\sf w}_{n,m}(a,b,c;q)\sim 
\frac{(\frac{qb}{a},\frac{qc}{a};q)_\infty}{(q,q^2bc;q)_\infty}
\left[\qhyp22{q^{-m},\frac{q^m}{a^2}}{\frac{qb}{a},\frac{qc}{a}}{q,q^2bc}\right]^2
q^n,
\end{equation}
where $\sum_n{\sf v}_{n,m}(a,b,c;q)=
{\sf V}_{m}(a,b;q)$, $\sum_n{\sf w}_{n,m}(a,b,c;q)=
{\sf W}_{m}(a,b,c;q)$.
Hence, 
by using the direct comparison test 
${\sf V}_{m}(a,b,c;q)$,
${\sf W}_{m}(a,b,c;q)$,
converges if $|q|<1$
since both infinite series associated with \eqref{DcdqiHO} are convergent. 
This is because both asymptotic infinite series are nonterminating ${}_4\phi_3$'s with vanishing numerator parameters and argument $q$.
This completes the proof.
\end{proof}

\begin{rem}
Note that the second term on the left-hand side of \eqref{DcdqiHO} can be obtained from the first term by 
replacing
$(b,c)\mapsto\left(1/({qc}),{1}/({qb})\right)$,
and using the symmetry of the continuous dual $q^{-1}$-Hahn polynomials in the $b,c$ parameters.
\end{rem}

\noindent {
There exists an infinite discrete bilateral orthogonality relation for continuous dual $q^{-1}$-Hahn polynomials, which is derived in
\cite[Theorem 5.2]{IsmailZhangZhou2022}. In order to study the convergence properties of this infinite bilateral sum, we will need the following estimates.
}
{
\begin{lem}\label{lem:4.10}
Let $n\in\N_0$, $q\in\CCdag$, $\alpha,a,b,c\in\CCast$. Then
as $k\to\pm \infty$, one has 
\begin{equation}
p_n\left[\frac{q^{-k}}{\alpha};a,b,c|q^{-1}\right]\sim \alpha^{\mp n}
q^{\mp nk}.
\end{equation}
\end{lem}
}
{
\begin{proof}
As in the proof of Lemma \ref{lem346}, first start with \eqref{cdqiH:3} and replace $z=q^{-k}/\alpha$. First taking the limit as $k\to\infty$  using \eqref{limqChu} and then replacing $z\mapsto z^{-1}$ followed by taking the limit as $k\to-\infty$
using \eqref{limqChu} completes the proof.
\end{proof}
}

\noindent In the next result we present the 
property of orthogonality
for the continuous dual $q^{-1}$-Hahn polynomials.
\begin{thm}
\label{thm410}
Let $m,n\in\N_0$, $q\in\CCdag$, $\alpha,a,b,c\in\CCast$,
with $a\alpha^{\pm}, b\alpha^{\pm}, c\alpha^{\pm}\not 
\in \Omega_q$. Then, there is the following infinite discrete bilateral orthogonality relation for continuous dual $q^{-1}$-Hahn polynomials:
\begin{eqnarray}
&&\sum_{k=-\infty}^\infty p_m\left[\frac{q^{-k}}{\alpha};a,b,c|q^{-1}\right]
\overline{p_n\left[\frac{q^{-k}}{\alpha};a,b,c|q^{-1}\right]}
\frac{(\frac{\alpha}{a},\frac{\alpha}{b},\frac{\alpha}{c};q)_k}{(q\alpha a,q\alpha b,q\alpha c;q)_k}q^{\binom{k}{2}}(-q\alpha abc)^k
(1-q^{2k}\alpha^2)
\nonumber\\
&&\hspace{2cm}
=q^{-4\binom{n}{2}}\frac{(q,\alpha^2,\frac{q}{\alpha^2},qab,qac,qbc;q)_\infty}{(qa\alpha^\pm,qb\alpha^\pm,qc\alpha^\pm;q)_\infty}\left(\frac{a^2b^2c^2}{q}\right)^n\left(q,\frac{1}{ab},\frac{1}{ac},\frac{1}{bc};q\right)_n\delta_{m,n}.\label{eq:189}
\end{eqnarray}
\end{thm}
{
\begin{proof}
Start with \cite[Theorem 5.2]{IsmailZhangZhou2022} corrected such that
\begin{eqnarray}
&&\hspace{-0.4cm}\sum_{k=-\infty}^\infty
V_n(x_k;{\bf t}|q)V_m(x_k;{\bf t}|q)
\left(-t_1z_k,\frac{t_1}{z_k},-t_2z_k,\frac{t_2}{z_k},-t_3z_k,\frac{t_3}{z_k};q\right)_\infty
\alpha^{4k}q^{k(2k-1)}(1+\alpha^2q^{2k})\nonumber\\
&&\hspace{1cm}=\left(q,-\alpha^2,-\frac{q}{\alpha^2},-\frac{t_1t_2}{q},
-\frac{t_1t_3}{q},-\frac{t_2t_3}{q};q\right)_\infty
\left(\frac{t_1^2}{q^3}\right)^n
\frac{(q,-q^2/t_1t_3;q)_n}
{(-q^2/t_1t_2,-q^2/t_2t_3)_n}\delta_{m,n},
\end{eqnarray}
where ${\bf t}:=\{t_1,t_2,t_3\}$, $z_k:=q^{-k}/\alpha$.
Replacing $V_n$ with $p_n$ by using 
\eqref{IsmailVnot},
where $x=x_k=\frac12(z_k-z_k^{-1})$, $t_1=qa/i$, $t_2=qb/i$, 
$t_3=qc/i$, writing the infinite $q$-shifted factorials as finite 
$q$-shifted factorials and replacing $\alpha\mapsto i\alpha$, 
and after a straightforward manipulations, we obtain \eqref{eq:189}. In order to prove the convergence of 
the former bilateral series, we take into account Lemma \ref{lem:4.10} and the comparison test.
Such a bilateral sum can be split into two parts: 
The first one for $n=m$
\begin{align*}
\alpha^{-2n} \sum_{k=0}^\infty 
\frac{(\frac{\alpha}{a},\frac{\alpha}{b},\frac{\alpha}{c};q)_k}
{(q\alpha a,q\alpha b,q\alpha c;q)_k}&q^{\binom{k}{2}}
(-q\alpha abc q^{-2n})^k(1-q^{2k}\alpha^2)\\=&
(1-\alpha^2)\alpha^{-2n}
\qhyp{5}{5}{\frac \alpha a,\frac \alpha b,\frac \alpha c,{\pm}q\alpha}
{q\alpha a,q\alpha b,q\alpha c,\pm \alpha}{q \alpha a b c}
\end{align*}
converges because it is an entire function, 
and the second one \cite[(17.4.3)]{NIST:DLMF}, using 
\eqref{critlim}, it is equal to 
\begin{align*}
(1-\alpha^2)\alpha^{2n}
\qhyp{5}{5}{\frac 1{\alpha a},\frac 1{\alpha b},\frac 1{\alpha c},{\pm}\frac q\alpha}
{\frac{qa}\alpha,\frac{qb}\alpha,\frac{qc}\alpha,\pm \frac 1\alpha}{q \alpha a b c}
\end{align*}
converges because it is also an entire function. Hence the result 
holds.
\end{proof}
}

{
\begin{rem}
If one adopts duality for continuous dual $q^{-1}$-Hahn polynomials with big $q$-Jacobi polynomials \eqref{dcdqiHbqJab}, one sees that the above infinite discrete bilateral 
orthogonality relation is equivalent with Theorem \ref{thm316} below. This is true because the big $q$-Jacobi polynomials with negative degrees all vanish.
\end{rem}
}

\subsection{The $q$ and $q^{-1}$-Al-Salam--Chihara polynomials}
\noindent
Let ${\mathbf a}:=\{a,b\}$,
$a,b\in\CCast$. 
Then Al-Salam--Chihara polynomials $p_n(x;{\bf a}|q)$ satisfy the following
continuous orthogonality relation \cite[(14.8.2)]{Koekoeketal}.

\begin{thm}
Let $m,n\in\N_0$, $q\in\CCdag$, $x=\cos\theta$, if $a,b$ are real or complex conjugates, $q^nab\not\in\Omega_q$, $\max(|a|,|b|)<1$. Then, the Al-Salam--Chihara polynomials satisfy the following continuous orthogonality relation: 
\cite[(14.8.2)]{Koekoeketal}
\begin{equation}
\int_0^\pi Q_m(x;{\bf a}|q)Q_n(x;{\bf a}|q)w_q(x;{\bf a})\,
{\mathrm d}\theta=h_n({\bf a}|q)\delta_{m,n},
\label{ASCO}
\end{equation}
where
\begin{eqnarray}
&&\hspace{-5.6cm}
w_q(\cos\theta;{\bf a}):=
\frac{(\expe^{\pm 2i\theta};q)_\infty}
{({\bf a}\expe^{\pm i\theta};q)_\infty},
\label{ASChw}
\quad h_n({\bf a}|q):=\frac{2\pi}{(q^{n+1},q^nab;q)_\infty}.
\end{eqnarray}
\end{thm}
\begin{proof}
The proof of this orthogonality relation follows by using the proof of the orthogonality relation for Askey--Wilson polynomials, Theorem \ref{AWorth}, with $c,d\to 0$.
\end{proof}
\medskip


\medskip
\noindent
The $q^{-1}$-Al-Salam--Chihara polynomials satisfy the following continuous orthogonality relation.

\begin{cor}
\label{corr311}
Let $n,m\in\N_0$, $q\in\CCdag$, $x=\frac12(z+z^{-1})\in\CCast$, $a,b\in\CCast$, $ab\not\in\Upsilon_q$. Then
\begin{eqnarray}
&&\hspace{-1.5cm}\int_{0}^{i\infty}Q_n(\tfrac12(z+z^{-1});a,b|q^{-1})
Q_{m}(\tfrac12(z+z^{-1});a,b|q^{-1})
w(z;a,b|q)\,\dd z=h_n(a,b|q)\delta_{m,n},
\label{qiASCO}
\end{eqnarray}
where
\begin{equation}
w(z;a,b|q):=\frac{(qaz^\pm,qbz^\pm;q)_\infty}{z(qz^{\pm 2};q)_\infty},
\label{qiASCw}
\end{equation}
\begin{equation}
h_n(a,b|q):=q^{-2\binom{n}{2}}
(q,qab;q)_\infty(q,\tfrac{1}{ab};q)_n
\left(\frac{ab}{q}\right)^n\,\log q^{-1}.
\label{qiASCn}
\end{equation}
\label{cqiASCOc}
\end{cor}

\begin{proof}
This orthogonality relation can be 
found by taking the limit as $d\to0$ in
Theorem \ref{cdqiHiOt}.
\end{proof}

\noindent The following infinite 
series discrete orthogonality 
relation for the $q^{-1}$-Al-Salam--Chihara polynomials was first described 
in the Askey--Ismail (1984) memoir 
and can be recovered by combining 
\cite[(3.82), (3.81), (3.67), (3.40)]{AskeyIsmail84}.
See also Groenevelt (2021) 
\cite[(3.4)]{Groenevelt2021}.

\medskip
\noindent In order to investigate 
the region of convergence for the 
parameters involved in the 
following infinite discrete 
orthogonality relation, we will 
need the following result.

\begin{lem}
Let $n\in\N_0$, $q\in\CCdag$, $a,b\in\CCast$. Then one has 
as $m\to\infty$, that
\begin{equation}
Q_n[q^{-m}a;a,b|q^{-1}]
\sim q^{-\binom{n}{2}}(-a)^n(q^{-m};q)_n
\sim q^{-nm}a^n.
\end{equation}
\label{lem410}
\end{lem}

\begin{proof}
Start with \eqref{qiASC:3}, replace $z=q^{-m}a$ and as $m\to\infty$ then
\[
Q_m[q^{-m}a;a,b|q^{-1}]\sim q^{-\genfrac{(}{)}{0pt}{}{n}{2}} (-a)^n 
\qhyp{1}{0}{q^{-n}
}
{-}{q,q^{n-m}}\sim q^{-nm}a^n,
\]
which completes the proof.
\end{proof}

\begin{thm}
Let $n,n'\in\N_0$, $q\in\CCdag$, $a,b\in\CCast$. Then
\begin{eqnarray}
&&\hspace{-0.4cm}\sum_{m=0}^\infty
q^{2\binom{m}{2}}\left(\frac{qb}{a}\right)^m
\frac{(\frac{q}{a^2};q)_{2m}(\frac{1}{a^2},
\frac{1}{ab};q)_m}{(\frac{1}{a^2};q)_{2m}(q,\frac{qb}{a};q)_m}
\,Q_n\left[q^{-m}a;a,b;q^{-1}\right]
Q_{n'}\left[q^{-m}a;a,b;q^{-1}\right]\nonumber\\
&&\hspace{6.5cm}=
q^{-2\binom{n}{2}}\left(\frac{ab}{q}\right)^n 
\frac{(\frac{q}{a^2};q)_\infty(q,\frac{1}{ab};q)_n}
{(\frac{qb}{a};q)_\infty}\delta_{n,n'}.
\label{ASCorthi}
\end{eqnarray}
\end{thm}

\begin{proof}
This is obtained using cf.~\cite[(3.4)]{Groenevelt2021} (corrected 
so that the degrees of the $q^{-1}$-Al-Salam--Chihara polynomials 
are $n,n'$ respectively, and the Kronecker delta symbol 
is $\delta_{n,n'}$).
Consider $n=n'$. The left-hand side of \eqref{ASCorthi} as ${\sf P}_{n}(a,b;q)$, then 
using Lemma \ref{lem410}, we have 
as $m\to\infty$, the summand behaves like 
\begin{equation}
{\sf p}_{m,n}(a,b;q)\sim 
a^{2n}\frac{(\frac{1}{ab},\frac{q}{a^2};q)_\infty}{(q,\frac{qb}{a};q)_\infty}
q^{2\binom{m}{2}}
\left(\frac{b}{q^{2n-1}a}\right)^m,
\end{equation}
where $\sum_m{\sf p}_{m,n}(a,b;q)=
{\sf P}_{n}(a,b;q)$.
Hence, 
by using the direct comparison test 
${\sf P}_{n}(a,b;q)$ converges 
since the infinite series associated with \eqref{ASCorthi} is convergent. Therefore, it is convergent for all values of $a,b$, $a\ne 0$ and $n\in\N_0$. This completes the proof.
\end{proof}
 
\noindent There is also the orthogonality of $q^{-1}$-Al-Salam--Chihara polynomials which comes from the standard orthogonality relation of little $q$-Jacobi polynomials \cite[(14.12.2)]{Koekoeketal} by using the duality relation \eqref{dqiASClqJa}. In order to study the convergence properties of these polynomials, we will need the following result.

\begin{lem}
Let $n\in\N_0$, $q\in\CCdag$, $a,b\in\CCast$. Then
as $n\to\infty$, one has 
\begin{eqnarray}
&&\hspace{-0.5cm}Q_n(x;a,b|q^{-1})\sim q^{-\binom{n}{2}}\left\{ \begin{array}{ll}
\displaystyle 
(-b)^n\frac{(\frac{z^\pm}{b};q)_\infty}{(\frac{a}{b};q)_\infty}, &\qquad\mathrm{if}\ |a|<|b|, \nonumber \\[0.45cm]
\displaystyle 
(-a)^n\frac{(\frac{z^\pm}{a};q)_\infty}{(\frac{b}{a};q)_\infty},
&\qquad\mathrm{if}\ |a|>|b|, \nonumber \\[0.45cm]
\displaystyle 
(-a)^n\frac{(\frac{z^\pm}{a};q)_\infty}{(\frac{b}{a};q)_\infty}+
(-b)^n\frac{(\frac{z^\pm}{b};q)_\infty}{(\frac{a}{b};q)_\infty}, &\qquad\mathrm{if}\ |a|=|b|,\ a\ne b, \nonumber \\[0.45cm]
\displaystyle 
\left(n+1+\lambda(z;a;q)\right)(-a)^n\frac{(\frac{z^\pm}{a};q)_\infty}{(q;q)_\infty},
&\qquad\mathrm{if}\ a=b, \nonumber \\[0.2cm]
\end{array} \right.
\end{eqnarray}
where in the last formula,
\begin{equation}
 \lambda(z;a;q):=\frac
 {z{\mathrm E}'_q(-\frac{z}{a})}
 {a
 (\frac{z}{a};q)_\infty
 }+\frac{
 {\mathrm E}'_q(-\frac{1}{az})
 }{za
 (\frac{1}{az};q)_\infty
 }-\frac{2q{\mathrm E}'_q(-q)}
{(q;q)_\infty},
\end{equation}
with ${\mathrm E}_q(t)$ defined in \eqref{qexp2},
and
\begin{equation}
 {\mathrm E}'_q(t):=\sum_{k=1}^\infty\frac{
 q^{\binom{k}{2}}
 }{(q;q)_k}kt^{k-1}.
\label{hptq}
\end{equation}
The error term of the asymptotic formula is 
of order ${\mathcal{O}}
\left(\frac{|q|^{-\binom{n}{2}}(\max\{|a|,|b|\})^n}{n}\right)$.
\label{lem413}
\end{lem}
\begin{proof}
Start by using the generating function of 
$q^{-1}$-Al-Salam--Chihara polynomials 
${\sf H}(t;a,b|q)$ given by 
\eqref{gfIsmg2}.
Since
\begin{equation}
{\sf H}_1(a,b|q):= \lim_{t\to-\frac{1}{b}}
(1+bt){\sf H}(t;a,b|q)={1\over(q;q)_\infty}
\qhyp21{ \frac{z^{\pm}}{a}}{\frac{1}{ab}}
{q,\frac{a}{b}}=\frac{(\frac{z^\pm}{b};q)_\infty}
{(q,\frac{a^\pm }{b};q)_\infty},
\end{equation}
where we have used the $q$-Gauss sum \eqref{qGs},
which requires that $|a|<|b|$, and we obtain from 
Darboux's method that
\begin{equation}
 \frac{q^{\binom{n}{2}}Q_n(x;a,b|q^{-1})}{(-b)^n}=(q,\tfrac{1}{ab};q)_\infty {\sf H}_1(a,b|q)+\mathcal{O}\left(\frac{1}{n}\right)
 =\frac{(\frac{z^\pm}{b};q)_\infty}{(\frac{a}{b};q)_\infty}+\mathcal{O}\left(\frac{1}{n}\right),
\end{equation}
which proves the $|a|<|b|$ case.
Since $Q_n(x;a,b|q^{-1})$ is symmetric in $a$ and $b$, then for $|a|>|b|$, one also has the $|a|>|b|$ case. 
If $|a|=|b|$, with $a\neq b$.
By choosing $\delta=\gamma$ in \eqref{qiASCgfIsm2}, 
then one obtains \eqref{qiASCgfIsm3x} which
we refer to as ${\sf G}(t;a,b|q)$.
Since
\begin{equation}
{\sf G}_a(a,b|q):= \lim_{t\to-\frac{1}{a}}(1+at){\sf G}(t;a,b|q)=\frac{(\frac{z^\pm}{a};q)_\infty}{(q,\frac{b}{a};q)_\infty},
\end{equation}
and
\begin{equation}
{\sf G}_b(a,b|q):= \lim_{t\to-\frac{1}{b}}(1+bt){\sf G}(t;a,b|q)=\frac{(\frac{z^\pm}{b};q)_\infty}{(q,\frac{a}{b};q)_\infty},
\end{equation}
we obtain from Darboux's method that
\begin{eqnarray}
\hspace{-2.6cm}q^{\binom{n}{2}}Q_n(x;a,b|q^{-1})&=&(q;q)_n \left[(-a)^n{
\sf G}_a(a,b|q)+(-b)^n{\sf G}_b(a,b|q)+\mathcal{O}
\left(\frac{|a|^n}{n}\right)\right]
\nonumber\\\hspace{-2.6cm}&=&(-a)^n\frac{(\frac{z^\pm}{a};q)_\infty}{(\frac{b}{a};q)_\infty}+
(-b)^n\frac{(\frac{z^\pm}{b};q)_\infty}{(\frac{a}{b};q)_\infty}
+\mathcal{O}\left(\frac{|a|^n}{n}\right),
\end{eqnarray}
as $n\to\infty$.
Finally, we consider the critical case $a=b$. 
The generating function ${\sf G}(t;a,a|q)$ 
has a double pole at $t=-\frac{1}{a}$ with
\begin{equation}
{\sf G}_2(a|q):= \lim_{t\to-\frac{1}{a}}(1+at)^2{\sf G}(t;a,a|q)
=\frac{(\frac{z^\pm}{a};q)_\infty}{(q;q)_\infty^2},
\end{equation}
and
\begin{eqnarray}
&&\hspace{-5.4cm}{\sf G}_1(a|q):= \lim_{t\to-\frac{1}{a}}\left[(1+at)
{\sf G}(t;a,a|q)-\frac{{\sf G}_2(a|q)}{1+at}\right]
\nonumber\\
&&\hspace{-4.0cm}=\frac{(\frac{z^\pm}{a};q)_\infty}{(q;q)_\infty^2}
\left[\frac{z{\mathrm E}'_q(-\frac{z}{a})}{a{\mathrm E}_q(-\frac{z}
{a})}+\frac{{\mathrm E}'_q(-\frac{1}{az})}{za{\mathrm E}_q(-\frac{1}
{az})}-\frac{2q{\mathrm E}'_q(-q)}{{\mathrm E}_q(-q)}\right], 
\end{eqnarray}
where ${\mathrm E}_q(t)$ and ${\mathrm E}'_q(t)$ are defined by 
\eqref{qexp2}, \eqref{hptq}.
In other words, we have
\begin{equation}
{\sf G}(t;a,a|q)=\frac{{\sf G}_2(a|q)}{(1+at)^2}+\frac{{\sf G}_1(a|q)}
{1+at}+{\mathcal{O}}(1),
\end{equation}
as $t\to-\frac{1}{a}$.
By Darboux's method, we obtain
\begin{eqnarray}
&&\hspace{-3.5cm} q^{\binom{n}{2}}Q_n(x;a,b|q^{-1})
=(q;q)_n(-a)^n\left[(n+1)
{\sf G}_2(a|q)+{\sf G}_1(a|q)
+\mathcal{O}\left(\frac{1}
{n}\right)\right]\nonumber\\
&&\hspace{-0.3cm}=(-a)^n\frac{(\frac{z^\pm}{a};q)_\infty}
{(q;q)_\infty}\left[n+1+\lambda(z;a;q)
+\mathcal{O}\left(\frac{1}{n}
\right)\right],
\end{eqnarray}
as $n\to\infty$.
This completes the proof.
\end{proof}

\noindent If we let $m\in\N_0$, then for the special 
value $z=q^{-m}a$, we can obtain from Lemma \ref{lem413}, 
the following asymptotic expression.

\begin{lem}
Let $m,n\in\N_0$, $q\in\CCdag$, $a,b\in\CCast$. Then, as $n\to\infty$, one has 
\begin{eqnarray}
&&\hspace{-4cm}Q_n[q^{-m}a;a,b|q^{-1}]\sim q^{-\binom{m}{2}}\left(-\frac{a}
{qb}\right)^m(\tfrac{1}{ab};q)_\infty
\frac{(\frac{qb}{a};q)_m}{(\frac{1}{ab};q)_m}
q^{-\binom{n}{2}}(-b)^n.
\end{eqnarray} 
\label{lem414}
\end{lem}
\begin{proof}
Start with Lemma \ref{lem413} and replace $z=q^{-m}a$. Then, after 
simplification, the result is obtained.
\end{proof}

\begin{thm}
Let $m,m'\in\N_0$, $q\in\CCdag$, $a,b\in\CCast$, $|qb|<|a|$. Then 
\begin{eqnarray}
&&\hspace{-1.7cm}\sum_{n=0}^\infty \frac{q^{2\binom{n}{2}}\left(\frac{q}
{ab}\right)^n}
{(q,\frac{1}{ab};q)_n}Q_n[q^{-m}a;a,b|q^{-1}]Q_n[q^{-m'}a;a,b|q^{-1}]\nonumber\\
&&\hspace{3cm}=q^{-2\binom{m}{2}}\left(\frac{a}{qb}\right)^m
\frac{(\frac{q}{a^2};q)_\infty}{(\frac{qb}{a};q)_\infty}\frac{(\frac{1}{a^2};q)_{2m}(q,\frac{qb}{a};q)_m}{(\frac{q}{a^2};q)_{2m}(\frac{1}{a^2},\frac{1}{ab};q)_m}\delta_{m,m'}.
\label{idqiASCO}
\end{eqnarray}
\end{thm}

\begin{proof}
Starting with the orthogonality relation for little $q$-Jacobi polynomials given below as \eqref{lqJO}, applying the duality relation between little $q$-Jacobi polynomials and $q^{-1}$-Al-Salam--Chihara polynomials \eqref{dqiASClqJa} and simplifying provides the orthogonality relation.
If 
$m=m'$, let us define the left-hand side of \eqref{idqiASCO} as ${\sf U}_{m}(a,b;q)$, then 
using Lemma \ref{lem414}, we have 
as $n\to\infty$, the summand behaves like 
\begin{equation}
{\sf u}_{n,m}(a,b;q)\sim 
\frac{(\frac{1}{ab};q)_\infty}{(q;q)_\infty}q^{-2\binom{m}{2}}
\left(\frac{a}{qb}\right)^{2m}
\frac{(\frac{qb}{a},\frac{qb}{a};q)_m}{(\frac{1}{ab},\frac{1}{ab};q)_m}
\left(\frac{qb}{a}\right)^n
,
\end{equation}
where $\sum_n{\sf u}_{n,m}(a,b;q)=
{\sf U}_{m}(a,b;q)$.
Hence, 
by using the direct comparison test 
${\sf U}_{m}(a,b;q)$ converges if $|qb|<|a|$
since the infinite series associated with \eqref{idqiASCO} is convergent. 
This is because the asymptotic infinite series is a nonterminating ${}_2\phi_1$'s with vanishing numerator parameters and argument $q$.
There will be a singularity of the asymptotic series when
$\frac{1}{ab}\in\Omega_q$. 
This completes the proof.
\end{proof}

\noindent {
There also exists an infinite discrete bilateral orthogonality relation for the $q^{-1}$-Al-Salam--Chihara polynomials, which was
originally derived in 
\cite[Theorem 4.6]{Ismail2020}.
\begin{thm}\label{thm:4.20}
Let $m,n\in\N_0$, $q\in\CCdag$, $\alpha,a,b\in\CCast$,
with $a\alpha^{\pm}, b\alpha^{\pm}\not 
\in \Omega_q$. Then, there is the following infinite discrete bilateral orthogonality relation for  $q^{-1}$-Al-Salam--Chihara polynomials:
\begin{eqnarray}
&&\hspace{-1.5cm}\sum_{k=-\infty}^\infty Q_m\left[\frac{q^{-k}}{\alpha};a,b|q^{-1}\right]
\overline{Q_n\left[\frac{q^{-k}}{\alpha};a,b|q^{-1}\right]}
\frac{(\frac{\alpha}{a},\frac{\alpha}{b};q)_k}{(q\alpha a,q\alpha b;q)_k}q^{2\binom{k}{2}}(q\alpha^2 ab)^k
(1-q^{2k}\alpha^2)\nonumber\\
&&\hspace{3cm}=q^{-2\binom{n}{2}}\frac{(q,\alpha^2,\frac{q}{\alpha^2},qab;q)_\infty}{(qa\alpha^\pm,qb\alpha^\pm;q)_\infty}\left(\frac{ab}{q}\right)^n\left(q,\frac{1}{ab};q\right)_n\delta_{m,n}.
\end{eqnarray}
\end{thm}
}
{
\begin{proof}
One can translate the result in \cite[Theorem 4.6]{Ismail2020} or simply take the limit as $c\to 0$ in Theorem \ref{thm410}.
\end{proof}
}

{
\begin{rem}
If one adopts duality for the $q^{-1}$-Al-Salam--Chihara polynomials with little $q$-Jacobi polynomials \eqref{dqiASClqJa}, one sees that the above infinite discrete bilateral 
orthogonality relation is equivalent with Theorem \ref{thm314} below. This is true because the little $q$-Jacobi polynomials with negative degrees all vanish.
\end{rem}
}

\subsection{The continuous big $q$ and big $q^{-1}$-Hermite polynomials}

\noindent
Let 
$a\in\CCast$. 
Then continuous big $q$-Hermite polynomials $H_n(x;a|q)$ satisfy the following continuous orthogonality relation
\cite[(14.18.2)]{Koekoeketal}.
\begin{thm}Let $m,n\in\N_0$, $q\in\CCdag$, $x=\cos\theta$, $a\in(-1,1)$. Then the continuous big $q$-Hermite polynomials satisfy the following continuous orthogonality relation 
\cite[(14.18.2)]{Koekoeketal}
\begin{equation}
\int_0^\pi H_m(x;{a}|q)H_n(x;{a}|q)w_q(x;{a})\,
{\mathrm d}\theta=h_n(q)\delta_{m,n},
\label{cbqHO}
\end{equation}
where
\begin{eqnarray}
&&\hspace{-7.0cm}
w_q(\cos\theta;a):=
\frac{(\expe^{\pm 2i\theta};q)_\infty}
{(a\expe^{\pm i\theta};q)_\infty},
\quad h_n(q):=\frac{2\pi}{(q^{n+1};q)_\infty}.\label{cbqHn}
\end{eqnarray}
\end{thm}
\begin{proof}
The proof of this orthogonality relation follows by using the proof of the orthogonality relation for Askey--Wilson polynomials, Theorem \ref{AWorth}, with $b,c,d\to 0$.
\end{proof}
\medskip


\noindent The continuous big $q^{-1}$-Hermite polynomials satisfy the following orthogonality relation.

\begin{cor}
\label{corr318}
Let $n,m\in\N_0$, $q\in\CCdag$, $x=\frac12(z+z^{-1})\in\CCast$, $a\in\CCast$. Then
\begin{eqnarray}
&&\hspace{-3.2cm}\int_{0}^{i\infty}H_n(\tfrac12(z+z^{-1});a|q^{-1})
H_{m}(\tfrac12(z+z^{-1});a|q^{-1})
w(z;a|q)\,\dd z=h_n(q)\delta_{m,n},
\label{cbqiHO}
\end{eqnarray}
where
\begin{equation}
w(z;a|q):=\frac{(qaz^\pm;q)_\infty}{z(qz^{\pm 2};q)_\infty},\quad 
h_n(q):=q^{-\binom{n}{2}}
(q;q)_\infty(q;q)_n
\left(-\frac{1}{q}\right)^n\,\log q^{-1}.
\label{cbqiHn}
\end{equation}
\label{cbqiHOt}
\end{cor}
\begin{proof}
This orthogonality relation can be 
found by taking the limit as $d\to0$ in
Corollary \ref{cqiASCOc}.
\end{proof}

\noindent
There exists an infinite discrete orthogonality for continuous big $q^{-1}$-Hermite polynomials
which one can obtain from the infinite discrete orthogonality of $q^{-1}$-Al-Salam--Chihara polynomials by taking the limit as $b\to 0$. 

\medskip
\noindent
In order to investigate the region 
of convergence for the parameters 
involved in the following infinite 
discrete orthogonality relation, 
we will need the following result.

\begin{lem}
Let $n\in\N_0$, $q\in\CCdag$, $a\in\CCast$. Then one has 
as $m\to\infty$, that
\begin{equation}
H_n[q^{-m}a;a|q^{-1}]\sim q^{-nm}a^n.
\end{equation}
\label{lem416}
\end{lem}

\begin{proof}
Start with \eqref{cbqiH:4}, replace $z=q^{-m}a$ and as $m\to\infty$ then
\[
H_m[q^{-m}a;a|q^{-1}]\sim q^{-mn}a^n\qhyp10{q^{-n}}{-}{q,q^{m+1}}\sim q^{-nm}a^n,
\]
which completes the proof.
\end{proof}

\noindent 
An infinite discrete orthogonality relation for the continuous big $q^{-1}$-Hermite polynomials is given 
in the following theorem.
\begin{thm}
Let $n,n'\in\N_0$, $q\in\CCdag$, $a\in\CCast$. Then the continuous big $q^{-1}$-Hermite polynomials satisfy the following infinite discrete orthogonality relation
\begin{eqnarray}
&&\hspace{-0.9cm}\sum_{m=0}^\infty
q^{3\binom{m}{2}}\!\left(-\frac{q}{a^2}\right)^m
\frac{(\frac{q}{a^2};q)_{2m}(\frac{1}{a^2}
;q)_{m}}{(\frac{1}{a^2};q)_{2m}(q
;q)_m}
H_n[q^{-m}a;a|q^{-1}]H_{n'}[q^{-m}a;a|q^{-1}]=\frac{q^{-\binom{n}{2}}}{(-q)^n}
\left(\frac{q}{a^2};q\right)_\infty\!\!\!\!(q;q)_n\delta_{n,n'}.
\label{cbqHorthi}
\end{eqnarray}
\label{thm358}
\end{thm}
\begin{proof}
Start with cf.~\cite[(3.4)]{Groenevelt2021} (corrected 
so that the degrees of the Al-Salam--Chihara polynomials 
are $n,n'$ respectively, and the Kronecker delta symbol 
is $\delta_{n,n'}$).
Consider $n=n'$. The left-hand side of \eqref{cbqHorthi} as ${\sf S}_{n}(a;q)$, then 
using Lemma \ref{lem416}, we have 
as $m\to\infty$, the summand behaves like 
\begin{equation}
{\sf s}_{m,n}(a;q)\sim 
a^{2n}
\frac{(\frac{q}{a^2};q)_\infty}{(q;q)_\infty}
q^{3\binom{m}{2}}
\left(-\frac{1}{q^{2n-1}a^2}\right)^m,
\end{equation}
where $\sum_m{\sf s}_{m,n}(a;q)=
{\sf S}_{n}(a;q)$. 
Hence, 
by using the direct comparison test 
${\sf S}_{n}(a;q)$ converges 
since the infinite series associated with \eqref{cbqHorthi} is convergent. Therefore, it converges for all values of $a$, $a\ne 0$ and $n\in\N_0$. This completes the proof.
\end{proof}

\noindent
One may also obtain an infinite discrete orthogonality relation for continuous big $q^{-1}$-Hermite polynomials, which comes from the orthogonality relation with $q$-Bessel polynomials.
In order to study the convergence properties of this orthogonality relation, we will need the asymptotics of the continuous big $q^{-1}$-Hermite polynomials as $n\to\infty$.

\begin{lem}
Let $n\in\N_0$, $q\in\CCdag$, $a\in\CCast$. Then
as $n\to\infty$, one has 
\begin{equation}
H_n(x;a|q^{-1})\sim q^{-2\binom{n}{2}}
(-a)^n(\tfrac{z^\pm}{a};q)_\infty.
\label{Hn-large-n}
\end{equation}
The error term of the asymptotic formula is of order ${\mathcal{O}}\left(\frac{|q|^{-2\binom{n}{2}}|a|^n}{n}\right)$.
\label{lem418}
\end{lem}
\begin{proof}
Start by using the generating function of the continuous big $q^{-1}$-Hermite polynomials ${\sf I}(t;a|q)$ 
\begin{eqnarray}
\label{cbqinHegf-2}
&&\hspace{-7.6cm}
{\sf I}(t;a|q):=\sum_{n=0}^\infty 
\frac{q^{\binom{n}{2}}t^nH_n(x;a|q^{-1})}
{(q;q)_n}
=\frac{(-tz^\pm;q)_\infty}{(-ta;q)_\infty},
\end{eqnarray}
which follows from \eqref{qiASCgfIsm2} by setting $\gamma=\delta$ and taking the limit as $b\to 0$.
From ${\sf I}(t;a|q)$ 
we obtain the integral representation
\begin{equation}
 H_n(x;a|q^{-1})=q^{-\binom{n}{2}}{(q;q)_n\over2\pi i}\int_{C_R}{(-tz^\pm;q)_\infty\over(-ta;q)_\infty}{dt\over t^{n+1}},
\end{equation}
where $C_R$ is the circle centered at the origin with a radius $R=1/|a|$.
As $t\to -1/a$, we have
\begin{equation}
 {(-tz^\pm;q)_\infty\over(-ta;q)_\infty}\sim {(\frac{z^\pm}{a};q)_\infty\over(1+ta)(q;q)_\infty}.
\end{equation}
By the Darboux method, we obtain
\begin{align}
 {q^{\binom{n}{2}}H_n(x;a|q^{-1})\over (q;q)_n}
 ={(\frac{z^\pm}{a};q)_\infty(-a)^n\over(q;q)_\infty}\left(1+\mathcal{O}\left(\frac{1}{n}\right)\right),
\end{align}
as $n\to\infty$. 
This proves \eqref{Hn-large-n}.
\end{proof}

\noindent 
Unfortunately, for the special argument $z=q^{-m}a$, Lemma \ref{lem418} doesn't give the correct asymptotic result as $n\to\infty$. Instead, we will require a different result.
\begin{lem}
Let $m,n\in\N_0$, $q\in\CCdag$, $a\in\CCast$. Then 
as $n\to\infty$, one has
\begin{eqnarray}
&&\hspace{-9cm}
H_n[q^{-m}a;a|q^{-1}]\sim q^{-2\binom{m}{2}}\left(\frac{a^2}{q}\right)^m a^{-n}.
\end{eqnarray}
\label{lem420}
\end{lem}
\begin{proof}
From \eqref{cbqinHegf-2}, we have 
\begin{eqnarray}
&& \hspace{-1cm} 
\sum_{n=0}^\infty \frac{t^nq^{\binom{n}{2}} H_n(x;a|q^{-1})}{(q;q)_n}=(-tz^{-1};q)_\infty(-tz;q)_m
=\sum_{k=0}^m\begin{bmatrix}m\\[1pt]k\end{bmatrix}
_qq^{\binom{k}{2}}(tz)^k\sum_{j=0}^\infty{q^{\binom{j}{2}}\over(q,q)_j}(tz^{-1})^j. \nonumber
\end{eqnarray}
Assume $n\ge m$. By matching the coefficients of $t^n$ on both sides of the above identity, we have
\begin{align}
H_n(x;a|q^{-1})=\sum_{k=0}^m\begin{bmatrix}m\\[1pt]k\end{bmatrix}
_q{(q;q)_nq^{k^2-kn}z^{2k-n}\over(q;q)_{n-k}}\sim q^{m^2-nm}z^{2m-n},
\end{align}
as $n\to\infty$, which completes the proof.
\end{proof}

\noindent In the next result we present
a property of orthogonality
for the continuous big $q^{-1}$-Hermite polynomial which comes from the orthogonality relation for $q$-Bessel polynomials by applying the duality relation \eqref{dHqinym}.

\begin{thm}
Let $m,m'\in\N_0$, $q\in\CCdag$, $a\in\CCast$. Then, the continuous big $q^{-1}$-Hermite polynomials satisfy the following infinite discrete orthogonality relation:
\begin{eqnarray}
&&\hspace{-0.9cm}\sum_{n=0}^{\infty}\frac{q^{\binom{n}{2}}(-q)^n}{(q;q)_n}H_n[q^{-m}a;a|q^{-1}]H_n[q^{-m'}a;a|q^{-1}]
=q^{-3\binom{m}{2}}(\tfrac{q}{a^2};q)_\infty
\left(-\frac{a^2}{q}\right)^m
\frac{(\frac{1}{a^2};q)_{2m}(q;q)_m}{(\frac{q}{a^2};q)_{2m}(\frac{1}{a^2};q)_m}\delta_{m,m'}.
\label{qBesselcbqiHO}
\end{eqnarray} 
\label{thm359}
\end{thm}

\begin{proof}
One may obtain this orthogonality relation by starting with the orthogonality relation for $q$-Bessel polynomials Theorem \ref{thm368} and then using the duality relation with continuous big $q^{-1}$-Hermite polynomials or by taking the limit as $b\to 0$ in \eqref{idqiASCO}.
Now consider $m=m'$. Define the left-hand side of \eqref{qBesselcbqiHO} as ${\sf T}_{m}(a;q)$, then 
using Lemma \ref{lem420}, we have 
as $n\to\infty$, the summand behaves like 
\begin{equation}
{\sf t}_{n,m}(a;q)\sim 
\frac{q^{-4\binom{m}{2}}\left(a^2/q\right)^{2m}}{(q;q)_\infty}
q^{\binom{n}{2}}
\left(-\frac{q}{a^2}\right)^n,
\end{equation}
where $\sum_n{\sf t}_{n,m}(a;q)=
{\sf T}_{m}(a;q)$.
Hence, 
by using the direct comparison test 
${\sf T}_{m}(a;q)$ converges 
since the infinite series associated with \eqref{qBesselcbqiHO} is convergent. Therefore, it converges for all values of $a$, $a\ne 0$ and $m\in\N_0$. This completes the proof.
\end{proof}

\begin{rem}
Note that one may also obtain Theorem \ref{thm359} by using the closure relation \eqref{closD} applied to Theorem \ref{thm358}.
\end{rem}

\noindent {
There also exists an infinite discrete bilateral orthogonality relation for the continuous big $q^{-1}$-Hermite polynomials. 
\begin{thm}
Let $m,n\in\N_0$, $q\in\CCdag$, $\alpha,a\in\CCast$,
with $a\alpha^{\pm}\not \in \Omega_q$. Then, 
there is the following infinite discrete bilateral 
orthogonality relation for continuous big 
$q^{-1}$-Hermite polynomials:
\begin{eqnarray}
&&\hspace{-2.5cm}\sum_{k=-\infty}^\infty H_m\left[\frac{q^{-k}}{\alpha};a|q^{-1}\right]
\overline{H_n\left[\frac{q^{-k}}{\alpha};a|q^{-1}\right]}
\frac{(\frac{\alpha}{a};q)_k}{(q\alpha a;q)_k}q^{3\binom{k}{2}}(-q\alpha^3 a)^k
(1-q^{2k}\alpha^2)
\nonumber\\
&&\hspace{3cm}
=q^{-\binom{n}{2}}\frac{(q,\alpha^2,\frac{q}{\alpha^2};q)_\infty}{(qa\alpha^\pm;q)_\infty}\left(-\frac{1}{q}\right)^n\left(q;q\right)_n\delta_{m,n}.
\end{eqnarray}
\end{thm}
}
{
\begin{proof}
One can translate the result in \cite[Theorem 4.6]{Ismail2020} 
or simply take the limit as $b\to 0$ in Theorem \ref{thm:4.20}.
\end{proof}
}

{
\begin{rem}
If one adopts duality for the continuous big $q^{-1}$-Hermite polynomials with $q$-Bessel polynomials \eqref{dHqinym}, one sees that the above infinite discrete bilateral 
orthogonality relation is equivalent with Theorem \ref{thm248} below. This is true because the $q$-Bessel polynomials with negative degrees all vanish.
\end{rem}
}

\subsection{The continuous $q$ and $q^{-1}$-Hermite polynomials}

\noindent 
The continuous $q$-Hermite polynomials $H_n(x|q)$ satisfy the following
continuous orthogonality relation \cite[(14.26.2)]{Koekoeketal}.

\begin{thm}
Let $m,n\in\N_0$, $q\in\CCdag$, $x=\cos\theta$. Then 
the continuous $q$-Hermite polynomials satisfy the following continuous orthogonality relation 
\cite[(14.26.2)]{Koekoeketal}:
\begin{equation}
\int_0^\pi H_m(x|q)H_n(x|q)w_q(x)\,
{\mathrm d}\theta=h_n(q)\delta_{m,n},
\label{cqHO}
\end{equation}
where
$w_q(\cos\theta):=
(\expe^{\pm 2i\theta};q)_\infty$,
and $h_n(q)$ is given by \eqref{cbqHn}.
\end{thm}
\begin{proof}
The proof of this orthogonality relation follows by using the proof of the orthogonality relation for Askey--Wilson polynomials, Theorem \ref{AWorth}, with $a,b,c,d\to 0$.
\end{proof}

\medskip

\noindent
It is well-known that the continuous $q^{-1}$-Hermite polynomials satisfy an indeterminate moment problem, so there exists an infinite number of orthogonality relations for these polynomials. This was studied systematically in Ismail \& Masson (1994) \cite{IsmailMasson1994}. 
The first three orthogonality relations are continuous orthogonality relations, and the fourth orthogonality relation is an infinite discrete bilateral orthogonality relation. The weight functions for the continuous orthogonality relations are presented in \cite[Theorem 21.6.4]{Ismail:2009:CQO} and labeled as $w_1(x)$, $w_2(x)$ and $w_3(x;\alpha)$ therein. 

\medskip
\noindent The first orthogonality relation that we present which corresponds to the weight function $w_1(x)$ in \cite[Theorem 21.6.4]{Ismail:2009:CQO} can be derived from the corresponding orthogonality relation for continuous big $q^{-1}$-Hermite polynomials, Corollary \ref{cbqiHOt}, after taking the limit $a\to 0$. This orthogonality relation was initially derived by Askey in 1989 \cite[Theorem 2]{Askey89cqiH}.
\begin{cor}
\label{corr326}
Let $n,m\in\N_0$, $q\in\CCdag$, $x=\frac12(z+z^{-1})\in\CCast$. Then
\begin{eqnarray}
&&\hspace{-4.3cm}\int_{0}^{i\infty}H_n(\tfrac12(z+z^{-1})|q^{-1})
H_{m}(\tfrac12(z+z^{-1})|q^{-1})
w(z|q)\,\dd z=h_n(q)\delta_{m,n},
\label{cqiHO}
\end{eqnarray}
where
\begin{equation}
w(z|q):=\frac{1}{z(qz^{\pm 2};q)_\infty},
\label{cqiHw}
\end{equation}
and
$h_n(q)$ is defined in \eqref{cbqiHn}.
\end{cor}

\begin{proof}
This orthogonality relation can be 
found by taking the limit as $a\to0$ in
Corollary \ref{cbqiHOt}.
\end{proof}

\noindent 
The second continuous orthogonality relation, which corresponds to measure $w_2(x)$ in \cite[Theorem 21.6.4]{Ismail:2009:CQO}, is given as follows.
\begin{thm}
Let ${m,n}\in\N_0$, $q\in\CCdag$. Then
\begin{eqnarray}
&&\hspace{-0.7cm}\int_1^\infty H_{{m}}[iz|q^{-1}]H_{{n}}[iz|q^{-1}](1+\tfrac{1}{z^{2}})\exp\!\left({-}\frac{2(\log z)^2}{\log {q^{-1}}}\right)\dd z=(-q)^{-n} q^{-\frac18-\binom{n}{2}}(q;q)_n\sqrt{\frac{\pi\log q{^{-1}}}2}\delta_{{m,n}}.
\end{eqnarray} 
\end{thm}

\begin{proof}
See the proof of \cite[(21.7.7)]{Ismail:2009:CQO} (see also \cite{AFW1994}).
\end{proof}

\medskip
\noindent 
The third continuous orthogonality relation which corresponds to measure $w_3(x;\alpha)$ in \cite[Theorem 21.6.4]{Ismail:2009:CQO}, is given as follows.
\begin{thm}
Let $m,n\in\N_0$, $q\in\CCdag$, 
$\alpha\in \CC$ such that 
$\Im\alpha\ne 0$. 
Then
\begin{eqnarray}
&&\hspace{-2cm}\int_1^\infty 
\frac{H_{{m}}[iz|q^{-1}]H_{{n}}[iz|q^{-1}]\left(1+\frac{1}{z^2}\right)}{(\alpha z,\bar{\alpha}z,-\frac{q}{\alpha}z,-\frac{q}{\bar{\alpha}}z,
-\frac{\alpha}{z},-\frac{\bar{\alpha}}{z},\frac{q}{\alpha z},\frac{q}{\bar{\alpha}z};q)_\infty}\,\dd z
=
\frac{q^{-\binom{n}{2}}\pi i(q;q)_n }
{\alpha (-q)^n(q;q)_\infty\vartheta(-\alpha \bar{\alpha},\frac{\bar{\alpha}}{\alpha};q)
}
\delta_{{m,n}},
\end{eqnarray}
where $\vartheta$ is the modified Jacobi
theta function defined in \eqref{tfdef}.
\end{thm}
\begin{proof}
See proof in \cite[\S 21.5]{Ismail:2009:CQO}.
\end{proof}
\noindent 
The total mass of the above orthogonality relation is connected to the Ismail--Masson $q$-beta integral \cite[(7.30)]{IsmailMasson1994} (see also cf.~\cite[(3.13)]{IsmailRahman1995})
\begin{eqnarray}
&&\hspace{-0.7cm}\int_0^\infty
\frac{(iaz,-\frac{ia}{z},ibz,-\frac{ib}{z},icz,-\frac{ic}{z},idz,-\frac{id}{z};q)_\infty}{(fz,gz,-\frac{qz}{f},-\frac{qz}{g},-\frac{f}{z},-\frac{g}{z},\frac{q}{fz},\frac{q}{gz};q)_\infty}\left(1+\frac{1}{z^{2}}\right)\,\dd z
=\frac{2\pi i (\frac{ab}{q},\frac{ac}{q},\frac{ad}{q},\frac{bc}{q},\frac{bd}{q},\frac{cd}{q};q)_\infty}{(q,\frac{g}{f},\frac{qf}{g},-fg,-\frac{q}{fg},\frac{abcd}{q^3};q)_\infty},
\end{eqnarray}
where $\Im f$, $\Im g$ and $\Im(f/g)$ are not $0(\!\!\!\mod 2\pi)$ and $\Im(fg) \ne \pi(\!\!\!\mod 2\pi)$.

\medskip
\noindent 
The fourth orthogonality relation that we present for continuous $q^{-1}$-Hermite polynomials 
is an infinite discrete bilateral orthogonality relation which can be 
found in Ismail and Masson \cite[\S6]{IsmailMasson1994} (see also \cite[(4.1)]{ChristiansenKoelink2008}).
{
\begin{thm}
Let $m,n,\in\N_0$, $q\in\CCdag$, $\alpha\in(q,1]$. Then there is the following infinite discrete bilateral orthogonality relation for continuous  $q^{-1}$-Hermite polynomials:
\begin{eqnarray}
&&\hspace{-1.0cm}\sum_{k=-\infty}^\infty 
H_m\left[\frac{q^{-k}}{\alpha}\Big|q^{-1}\right]
H_{n}\left[\frac{q^{-k}}{\alpha}\Big|q^{-1}\right]
q^{4\binom{k}{2}}(q\alpha^4)^k
(1-q^{2k}\alpha^2)
=(q,\alpha^2,\frac{q}{\alpha^2};q)_\infty \frac{q^{-\binom{n}{2}}}{(-q)^{n}}(q;q)_n\delta_{m,n}.
\end{eqnarray}
\end{thm}
}
\begin{proof}
This orthogonality relation is a result due to Ismail and Masson \cite[\S6]{IsmailMasson1994} (see also \cite[(4.1)]{ChristiansenKoelink2008}).
\end{proof}

\noindent The total mass of the above orthogonality relation is connected to the following bilateral infinite series, which can be found in \cite[Lemma 5.1]{IsmailZhangZhou2022} for the special
case {$a,b,c,d\mapsto 0$ of the following bilateral sum}. Let $\alpha\in(q,1]$, $|abcd|<|q|^{-1}$, $z_k=q^{-k}/\alpha$. Then
\begin{eqnarray}
&&\hspace{-2.5cm}\sum_{k=-\infty}^\infty q^{4\binom{k}{2}}(q\alpha^4)^k(1+q^{2k}\alpha^2)
(-qaz_k,\tfrac{qa}{z_k},-qbz_k,\tfrac{qb}{z_k},-qcz_k,\tfrac{qc}{z_k},-qdz_k,\tfrac{qd}{z_k};q)_\infty \nonumber\\
&&\hspace{2cm}=\frac{(q,-\alpha^2,-\frac{q}{\alpha^2},-qab,-qac,-qad,-qbc,-qbd,-qcd;q)_\infty}{(qabcd;q)_\infty}.
\label{blabcd}
\end{eqnarray}
As described in a detailed fashion in 
\cite{IsmailRahman1995}, the above bilateral sum is 
equivalent to the Bailey's
${}_6\psi_6$ summation \cite[(17.7.7)]{NIST:DLMF} (where 
$|qa^2|<|bcde|$), by making the replacement 
$(a,b,c,d,e)\mapsto(i\alpha^2,\frac{\alpha}{a},\frac{\alpha}
{b},\frac{\alpha}{c},\frac{\alpha}{d})$.

\subsection{The big $q$-Jacobi polynomials and functions}

In the next result
we present orthogonality relations for the big $q$-Jacobi polynomials and functions.

\noindent{
\begin{thm}
Let $m,m'\in\N_0$, $q\in\CCdag$, $a,b,c\in\CCast$ such that $qa\in(0,1)$, $qb\in[0,1)$ and $c<0$. Then 
the big $q$-Jacobi polynomials satisfy the following 
orthogonality relation \cite[(14.5.2)]{Koekoeketal}
\begin{eqnarray}
&&\hspace{-0.5cm}\int_{qc}^{qa}
P_m(x;a,b,c;q)P_{m'}(x;a,b,c,q)
\frac{(\frac{x}{a},\frac{x}{c};q)_\infty}
{(x,\frac{bx}{c};q)_\infty}\,\dd_q x\nonumber\\
&&\hspace{0.9cm}=q^{\binom{m}{2}+1}a(1-q)(-q^2ac)^m\frac{(q,q^2ab,\frac{c}{a},\frac{qa}{c};q)_\infty}
{(qa,qb,qc,\frac{qab}{c};q)_\infty}
\frac{(qab;q)_{2m}(q,qb,\frac{qab}{c};q)_m}
{(q^2ab;q)_{2m}(qa,qc,qab;q)_m}\delta_{m,m'}.
\label{bqJO}
\end{eqnarray}
\end{thm}
}
{
\begin{proof}
See proof of \cite[(14.5.2)]{Koekoeketal}.
\end{proof}
}

\noindent
Orthogonality relations for the big $q$-Jacobi polynomials in the complex plane are studied systematically
in \cite[\S 4]{MR2832754}.
One also has an infinite discrete orthogonality relation for big $q$-Jacobi polynomials, which can be obtained through duality between the continuous dual $q^{-1}$-Hahn polynomials which were given in \cite{AtakishiyevKlimyk2006}.

\begin{thm}{Atakishiyev and Klimyk \cite[(4.30)]{AtakishiyevKlimyk2006}.}
\label{thm316}
Let $n,n'\in\N_0$, $q\in\CCdag$, $|ab|>1$, $|qb|<|a|$. Then, the big $q$-Jacobi polynomials satisfy the following orthogonality relation
\begin{eqnarray}
&&\hspace{-0.1cm}\sum_{m=0}^\infty
\frac{q^{-\binom{m}{2}}}{(-q^2ac)^m}
\frac{(qa,qc,qab,\pm\sqrt{q^3ab};q)_m}{(q,qb,\frac{qab}{c},\pm\sqrt{qab};q)_m}
\,P_m\left(q^{n+1}a;a,b,c;q\right)
P_m\left(q^{n'+1}a;a,b,c;q\right)\nonumber\\
&&\hspace{7.5cm}=q^{-n}
\frac{(q^2ab,\frac{c}{a};q)_\infty(q,\frac{qa}{c};q)_n}{(qb,qc;q)_\infty(qa,\frac{qab}{c};q)_n}\delta_{n,n'}.
\end{eqnarray}
\end{thm}
\begin{proof}
Starting with the discrete orthogonality relation for
the continuous dual $q^{-1}$-Hahn polynomials \eqref{AKporth}
The orthogonality relation \eqref{AKporth}
is obtained from \cite[(4.30)]{AtakishiyevKlimyk2006}
using \cite[(4.29)]{AtakishiyevKlimyk2006}
and comparing the result with
\eqref{cdqiH:1} twice.
Then, inserting the duality relation
\eqref{dcdqHbqJab}, simplifying and then making the replacement
$(a,b,c)\mapsto ({1}/{\sqrt{qab}},
\sqrt{b/(qa)},{c}/{\sqrt{qab}})$ 
completes the proof.
\end{proof}

\noindent 
Now we present continuous orthogonality of big and little $q$-Jacobi functions by starting
from continuous orthogonality of the
continuous dual $q$-Hahn polynomials and
the Al-Salam--Chihara polynomials.
The dual orthogonality for these functions is 
due to the following
duality relations for continuous
dual $q$-Hahn polynomials
and the Al-Salam--Chihara polynomials
with the big and little $q$-Jacobi functions 
respectively, see Theorems \ref{thm3.33}, \ref{thm3.35}.

\medskip
\noindent One has the following
continuous orthogonality relation 
for the big $q$-Jacobi function. 
For fixed $q,a,b\in\CCast$, define the constant ${\sf A}$ as follows
\begin{equation}
{\sf A}:={\sf A}(a,b;q):=-\frac{\log(qab)}{2\log(q)}.
\label{Adef}
\end{equation}
In the sequel, we will use the constant ${\sf A}$ 
whenever we apply the following continuous orthogonality relations for big and little
$q$-Jacobi functions. Note that this orthogonality relation is an index transform (see, for example \cite{Yakubovich1996}).
\begin{thm}
\label{cobqJ}
Let $n,n'\in\mathbb N_0$, $q\in\CCdag$, $\mu\in({\sf A},{\sf A}-i\pi/\log q)$, $a,b,c\in\CCast$. 
Then the big $q$-Jacobi polynomials satisfy
the following continuous orthogonality
relation
\begin{eqnarray}
&&\hspace{-0.4cm}\int_{\sf A}^{{\sf A}-\frac{i\pi}{\log q}}
P_\mu(q^{-n};a,b,c;q)P_\mu(q^{-n'};a,b,c;q)
\,
{\sf W}_\mu(a,b,c;q)
\,\dd \mu\nonumber\\
&&\hspace{5.5cm}=\frac{-2\pi i}{\log(q)(q,qa,qc,\frac{qc}{b};q)_\infty}
\frac{(qab)^n(q,\frac{qc}{b};q)_n}{(qa,qc;q)_n}\delta_{n,n'},
\end{eqnarray}
where
\begin{eqnarray}
\label{wb}
&&\hspace{-4cm}{\sf W}_\mu(a,b,c;q):=\frac{(\frac{q^{-2\mu-1}}{ab},q^{2\mu+1}ab;q)_\infty}{(q^{-\mu},\frac{q^{-\mu}}{b},\frac{q^{-\mu}c}{ab},q^{\mu+1}a,q^{\mu+1}c,q^{\mu+1}ab;q)_\infty}.
\end{eqnarray}
\end{thm}
\begin{proof}
This follows directly from the continuous orthogonality of 
continuous dual $q$-Hahn polynomials
\cite[(14.3.2)]{Koekoeketal}
and the continuous duality relation
between continuous dual $q$-Hahn polynomials and big $q$-Jacobi function in Theorem \ref{thm3.33}.
\end{proof}

\subsection{The little $q$-Jacobi polynomials and functions}

\noindent {We also have the following orthogonality relation for little $q$-Jacobi polynomials, which can be obtained from the infinite discrete orthogonality of $q^{-1}$-Al-Salam--Chihara polynomials by using duality with little $q$-Jacobi polynomials \eqref{dqiASClqJa}.}

\begin{thm}[Askey--Ismail \cite{AskeyIsmail84} (1984)]
\label{thm314}
Let $n,n'\in\N_0$, $q\in\CCdag$, $|ab|>1$, $|qb|<|a|$. Then, the little $q$-Jacobi polynomials satisfy the following orthogonality relation:
\begin{eqnarray}
&&\hspace{-1.3cm}\sum_{m=0}^\infty
\left(\frac{1}{qa}\right)^m
\frac{(qa,qab,\pm\sqrt{q^3ab};q)_m}{(q,qb,\pm\sqrt{qab};q)_m}
\,p_m\left(q^n;a,b;q\right)
p_m\left(q^{n'};a,b;q\right)\nonumber\\
&&\hspace{6.5cm}=
\left(\frac{1}{qa}\right)^n \frac{(q^2ab;q)_\infty(q;q)_n}
{(qa;q)_\infty(qb;q)_n}\delta_{n,n'}.
\end{eqnarray}
\end{thm}
\begin{proof}
Starting with cf.~\cite[(3.4)]{Groenevelt2021} (corrected 
so that the degrees of the Al-Salam--Chihara polynomials 
are $n,n'$ respectively, and the Kronecker delta symbol 
is $\delta_{n,n'}$), namely
\eqref{ASCorthi},
followed by using the duality relation
\eqref{dqiASClqJa}, simplifying and then making the 
replacement $(a,b)\mapsto (1/\sqrt{qab},
\sqrt{a/(qb)})$ completes the proof.
\end{proof}
Note that this orthogonality is clearly different 
from the standard orthogonality of the little 
$q$-Jacobi polynomials 
since in Theorem \ref{thm314}, the degrees of the little $q$-Jacobi polynomials are fixed and are given by the sum index.
In the next result we present
orthogonality relation
\cite[(14.12.2)]{Koekoeketal}.
\begin{thm}
Let $n,n'\in\N_0$, $q\in\CCdag$, $a,b\in\CCast$. Then, one has the following infinite discrete orthogonality relation for little $q$-Jacobi polynomials
\cite[(14.12.2)]{Koekoeketal}:
\begin{eqnarray}
&&\hspace{-0.5cm}
\sum_{m=0}^\infty
(qa)^m\frac{(qb;q)_m}{(q;q)_m}
p_n(q^m;a,b;q)p_{n'}(q^{m};a,b;q)
=(qa)^n\frac{(q^2ab;q)_\infty}{(qa;q)_\infty}\frac{(qab;q)_{2n}(q,qb;q)_n}{(q^2ab;q)_{2n}(qa,qab;q)_n}\delta_{n,n'}.
\label{lqJO}
\end{eqnarray}
\end{thm}
\begin{proof}
See proof of \cite[(14.12.2)]{Koekoeketal}.
\end{proof}

\noindent For the little $q$-Jacobi functions, one has the following continuous orthogonality relation, which is an index transform.
\begin{thm}
\label{colqJ}
Let $n,n'\in\N_0$, $q\in\CCdag$, $\mu\in({\sf A},{\sf A}-i\pi/\log q)$, $n, n'\in\mathbb N_0$, $a,b\in\CCast$ where ${\sf A}$ is defined in \eqref{Adef}.
Then, the little $q$-Jacobi polynomials 
satisfy the following continuous 
orthogonality relation:
\begin{eqnarray}
&&\hspace{-1.85cm}\int_{\sf A}^{{\sf A}-\frac{i\pi}{\log q}}
p_\mu\left(\frac{q^{-1-n}}{b};a,b;q\right)p_\mu\left(\frac{q^{-1-n'}}{b};a,b;q\right)
\,{\sf w}_\mu(a,b;q)
\,\dd \mu\nonumber\\
&&\hspace{3.5cm}=\frac{-2\pi i}{\log(q)(q,qa,qa,qb,\frac{1}{b},\frac{1}{b};q)_\infty}
\frac{(qab)^n(q;q)_n}{(qb;q)_n}\delta_{n,n'},
\end{eqnarray}
where
\begin{eqnarray}
\label{wl}
&&\hspace{-4cm}{\sf w}_\mu(a,b;q):=\frac{(\frac{q^{-2\mu-1}}{ab},q^{2\mu+1}ab;q)_\infty}{(q^{-\mu},\frac{q^{-\mu}}{a},\frac{q^{-\mu}}{b},\frac{q^{-\mu}}{b},q^{\mu+1}a,q^{\mu+1}a,q^{\mu+1}b,q^{\mu+1}ab;q)_\infty}.
\end{eqnarray}
\end{thm}
\begin{proof}
This follows directly from the continuous orthogonality of Al-Salam--Chihara polynomials \cite[(14.8.2)]{Koekoeketal}
and the continuous duality relation
between Al-Salam--Chihara polynomials and little $q$-Jacobi function in Theorem \ref{thm3.35}.
\end{proof}

\subsection{The $q$-Bessel polynomials and $q^{-1}$-Bessel functions}

\noindent
The $q$-Bessel polynomials satisfy the following infinite discrete orthogonality relation
cf.~\cite[(14.22.2)]{Koekoeketal}.
\begin{thm}
Let $m,m'\in\N_0$, $q\in\CCdag$, $a\in(0,\infty)$. Then the $q$-Bessel polynomials satisfy the following infinite discrete orthogonality relation, namely,
\begin{eqnarray}
&&\hspace{-0.6cm}\sum_{n=0}^\infty \frac{q^{\binom{n}{2}}(qa)^n}{(q;q)_n}y_m(q^n;a;q)y_{m'}(q^n;a;q)=q^{\binom{m}{2}}(qa)^m(-qa;q)_\infty
\frac{(-a;q)_{2m}(q;q)_m}{(-qa;q)_{2m}(-a;q)_m}\delta_{m,m'}.
\end{eqnarray}
\label{thm368}
\end{thm}
\begin{proof}
See \cite[(14.22.2)]{Koekoeketal}.
\end{proof}

\noindent
There is also an infinite discrete orthogonality for $q$-Bessel polynomials that comes from its duality with continuous big $q^{-1}$-Hermite polynomials \eqref{dHqinym}. It is now given in the following result.
\begin{thm}
\label{thm248}
Let $n,n'\in\N_0$, $q\in\CCdag$, $a\in\CCast$. Then, the $q$-Bessel polynomials satisfy the following infinite discrete orthogonality relation:
\begin{eqnarray}
&&\hspace{-1.4cm}\sum_{m=0}^\infty
\frac{q^{-\binom{m}{2}}}{(qa)^m}
\frac{(-qa;q)_{2m}(-a;q)_m}
{(-a;q)_{2m}(q;q)_m}
\,y_m(q^{n};a;q)
y_m(q^{n'};a;q)
=\frac{q^{-\binom{n}{2}}}{(qa)^n}
(-qa;q)_\infty(q;q)_n
\delta_{n,n'}.
\label{orthogqB1}
\end{eqnarray}
\end{thm}
\begin{proof}
This orthogonality relation can be obtained by starting with the infinite discrete orthogonality relation for continuous big $q^{-1}$-Hermite polynomials and using the duality relation with $q$-Bessel polynomials \eqref{dHqinym} and making straightforward replacements.
\end{proof}

\noindent
From the continuous duality of continuous big $q$-Hermite polynomials with the $q^{-1}$-Bessel functions, one can obtain a continuous orthogonality for $q^{-1}$-Bessel functions. 
One has the following continuous orthogonality relation for the $q^{-1}$-Bessel function which
is an index transform.
\begin{thm}
Let $n,n'\in\N_0$, $q\in\CCdag$, $a<0$. Then
\begin{eqnarray}
&&\hspace{-0.6cm}\int_{\sf B}^{{\sf B}+\frac{i\pi}{\log q}}
y_\mu(q^{-n};a;q^{-1})
y_\mu(q^{-n'};a;q^{-1})
\frac{(-q^{-2\mu}a,-\frac{q^{2\mu}}{a};q)_\infty}{(q^{-\mu},-\frac{q^\mu}{a};q)_\infty}q^{4\binom{\mu}{2}}\left(\frac{q}{a}\right)^{2\mu}
\dd \mu
=\frac{2\pi i(q;q)_n\,\delta_{n,n'}}{(-a)^n(q;q)_\infty\log q},
\label{COqiBf}
\end{eqnarray}
where
${\sf B}:={\log(-a)}/({2\log q})$.
\end{thm}
\begin{proof}
Start with the continuous orthogonality relation for continuous big $q$-Hermite polynomials \eqref{cbqHO} and insert the duality relation between the continuous big $q$-Hermite polynomials with the $q^{-1}$-Bessel function \eqref{dHnymu}. After simplification and straightforward replacements, the result follows.
\end{proof}

\subsection*{Acknowledgements}
We would like to thank Wolter Groenevelt, Mourad Ismail, Tom Koornwinder, and Hjalmar Rosengren for their valuable discussions. We would like to point out that the duality relations for continuous big $q$ and big $q^{-1}$-Hermite polynomials to the $q$ and $q^{-1}$-Bessel polynomials were originally pointed out by Wolter Groenevelt in a personal discussion with the first author. 

\def\cprime{$'$} \def\dbar{\leavevmode\hbox to 0pt{\hskip.2ex \accent"16\hss}d}

\end{document}